\documentclass[a4paper,12pt]{article} 

\usepackage[margin=1in]{geometry}

\usepackage[utf8]{inputenc}
\usepackage[T1]{fontenc}
\usepackage[english]{babel}
\usepackage{amsmath,amssymb,amsthm}
\usepackage[]{authblk}
\usepackage{bbm}
\usepackage{color}
\usepackage{bbm}
\usepackage{verbatim}
\usepackage[colorlinks]{hyperref}
\usepackage[capitalise]{cleveref}
\usepackage{enumitem}
\usepackage{mathtools}
\usepackage{comment}
\usepackage{commath}
\usepackage{subfig}
\usepackage{esint,cancel}

\usepackage{cite}

\usepackage{pgf,tikz}
\usetikzlibrary{arrows}

\newtheorem{thm}{Theorem}
\crefname{thm}{Theorem}{Theorems}
\newtheorem{cor}[thm]{Corollary}
\crefname{cor}{Corollary}{Corollaries}
\newtheorem{lem}[thm]{Lemma}
\crefname{lem}{Lemma}{Lemmas}
\newtheorem{prop}[thm]{Proposition}
\crefname{prop}{Proposition}{Propositions}

\crefname{conj}{Conjecture}{Conjectures}

\crefname{ques}{Question}{Questions}
\theoremstyle{definition}
\newtheorem{defn}[thm]{Definition}
\crefname{defn}{Definition}{Definitions}
\newtheorem{rem}[thm]{Remark}
\crefname{rem}{Remark}{Remarks}
\newtheorem{ex}[thm]{Example}
\crefname{ex}{Example}{Examples}

\crefname{obs}{Observation}{Observations}

\crefname{claim}{Claim}{Claims}
\newtheorem{ass}[thm]{Assumption}
\crefname{ass}{Assumption}{Assumptions}

\numberwithin{thm}{section}

\newcommand{\cG}{\ensuremath{\mathcal G}}

\newcommand{\cP}{\ensuremath{\mathcal P}}

\newcommand{\cT}{\ensuremath{\mathcal T}}

\newcommand{\cV}{\ensuremath{\mathcal V}}

\newcommand{\cX}{\ensuremath{\mathcal X}}

\newcommand{\Lp}[1]{\mathrm{L}^{#1}}

\renewcommand{\a}{{\ensuremath{\alpha}}}%

\newcommand{\E}{{\mathcal{E}}}

\newcommand{\g}{{\ensuremath{\gamma}}}
\newcommand{\G}{{\ensuremath{\Gamma}}}

\newcommand{\x}{{\ensuremath{\xi}}}

\newcommand{\dist}{d}
\newcommand{\sfd}{\mathsf d}
\newcommand{\sfc}{\mathsf c}
\newcommand{\sfn}{\mathsf n}
\DeclareMathOperator{\proj}{pr}
\DeclareMathOperator{\id}{id}

\renewcommand{\leq}{\leqslant}
\renewcommand{\geq}{\geqslant}
\renewcommand{\le}{\leqslant}
\renewcommand{\ge}{\geqslant}
\renewcommand{\to}{\rightarrow}

\DeclareMathOperator*{\argmin}{arg\,min}
\DeclareMathOperator{\supp}{supp}

\newcommand{\comma}{\,\mathrm{,}\;\,}
\newcommand{\semicolon}{\,\mathrm{;}\;\,}
\newcommand{\fstop}{\,\mathrm{.}}

\def\N{{\mathbb N}}

\def\R{{\mathbb R}}

\newcommand{\eps}{\epsilon}
\newcommand{\dd}{{\, \mathrm{d}}}

\newcommand{\dds}{\frac{\mathrm{d}}{\mathrm{d}s}}

\begin{document}
\title{Kinetic Optimal Transport (OTIKIN) -- Part 1:\\
\Large Second-Order Discrepancies Between Probability Measures}
\author{Giovanni Brigati\thanks{\textsf{giovanni.brigati@ist.ac.at}} }
\author{Jan Maas\thanks{\textsf{jan.maas@ist.ac.at}} }
\author{Filippo Quattrocchi\thanks{\textsf{filippo.quattrocchi@ist.ac.at}} }
\affil{Institute of Science and Technology Austria (ISTA) \\ 
\footnotesize Am Campus 1, 3400 Klosterneuburg, Austria}
\date{\vspace{-0.25cm}\today}
\maketitle
\vspace{-0.75cm}

\begin{abstract}
    This is the first part of a general description in terms of mass transport for time-evolving interacting particles systems, at a mesoscopic level. Beyond kinetic theory, our framework naturally applies in biology, computer vision, and engineering.

    The central object of our study is a new discrepancy~$\mathsf d$ between two probability distributions in position and velocity states, which is reminiscent of the~$2$-Wasserstein distance, but of second-order nature. We construct~$\mathsf d$ in two steps. First, we optimise over transport plans. The cost function is  given by the minimal \emph{acceleration} between two coupled states on a fixed time horizon~$T$. Second, we further optimise over the time horizon~$T>0$.

    We prove the existence of optimal transport plans and maps, and study two time-continuous characterisations of~$\mathsf d$. One is given in terms of dynamical transport plans. The other one ---in the spirit of the Benamou--Brenier formula--- is formulated as the minimisation of an action of the acceleration field, constrained by Vlasov's equations. Equivalence of static and dynamical formulations of $\mathsf d$ holds true. While part of this result can be derived from recent, parallel developments in optimal control between measures, we give an original proof relying on two new ingredients: Galilean regularisation of Vlasov's equations and a kinetic Monge--Mather shortening principle. 

    Finally, we establish a first-order differential calculus in the geometry induced by~$\mathsf d$, and identify solutions to Vlasov's equations with curves of measures satisfying a certain~$\mathsf d$-absolute continuity condition. One consequence is an explicit formula for the~$\mathsf d$-derivative of such curves.
\end{abstract}


\pagebreak
\section{Introduction}\label{sec:sec1}
\paragraph{Scientific background}
Kinetic equations describe systems of \emph{many} interacting particles, at an intermediate level between the microscopic scale, where each particle is tracked individually, and the macroscopic scale of observable quantities, corresponding to, e.g., fluid dynamics or diffusion models. Particles are characterised via their position~$x \in \mathcal{X}$ and velocity~$v \in \mathcal{V}$. At the kinetic scale---which is our point of view thorough this paper---we do not track the evolution of each single particle. Rather, particles are indistinguishable, and the only available information is their distribution in $x,v$. The evolution of the system over time~$t \in [0,\infty)$ is modelled in a statistical mechanics fashion, as a time-dependent probability distribution on the phase space~$\Gamma  \coloneqq \mathcal X \times \mathcal V$.

The hierarchy between  scales was already considered by J.-C.~Maxwell and L.~Boltzmann \cite{maxwell1867iv,boltzmann1872weitere}, and later included in D.~Hilbert's problems for the~$\text{XX}^\text{th}$ century (Problem~VI) \cite{gorban2018hilbert}. Kinetic equations, their derivation from microscopic dynamics, and their macroscopic limit regimes---fluid dynamics or diffusion---have been a vast research field ever since, with important open questions still under active investigation.

On the other side, the classical optimal transport (OT) theory \cite{santambrogio2015optimal,MR2459454}, see \S\ref{ss:ss12}, is naturally connected to the macroscopic description of particle systems. Indeed, OT can be reformulated in terms of fluid mechanics  \cite{benamou2000computational}.
In addition, OT provides a deep interpretation of diffusion equations as gradient flows in the space of probability measures~\cite{MR2401600}, as well as variational (JKO \cite{MR1617171}) discrete approximation schemes. 

In this paper, we take a step towards a new \emph{kinetic optimal transport} (OTIKIN) theory, specifically tailored to the kinetic description of particle systems. Indeed, our main object, a new \emph{second-order discrepancy}~$\sfd$ between measures on~$\Gamma$, preserves the distinct nature of the variables~$x$ and~$v$.
We consider the case where particles are subject to Newton's laws of mechanics.

\paragraph{Structure of the paper} 
\begin{description}
    \item[Section \ref{sec:sec1}.] The main definitions and results are formulated in \S\ref{ss:ss11}. In \S\ref{ss:ss12}, we draw connections with related works and collect some motivations, applications, and perspectives.
    \item[Section \ref{sec:sec2}.] We consider the case of Dirac masses. In \S\ref{ss:21}, we study the minimal acceleration problem between states in~$\Gamma$. In \S\ref{ss:22}, we introduce a non-parametric minimal-acceleration discrepancy.
    \item[Section \ref{sec:sec3}.] In \S\ref{ss:31}, we generalise the construction to a minimal-acceleration discrepancy~$\mathsf{d}$ between \emph{probability measures}. The definition is given as a static mass transportation problem. Optimisers (transport plans and maps) are shown to exist in \S\ref{ss:32}. Additional results are given in \S\ref{ss:33}.
    \item[Section \ref{sec:sec4}.] We analyse two equivalent dynamical formulations of the minimal-acceleration discrepancy $\sfd$.  These are defined, respectively, by means of dynamical transport plans (\S\ref{ss:41}) and minimal action of solutions to Vlasov's equations (\S\ref{ss:43}). Further results on dynamical plans and Vlasov's equations are collected in \S\ref{ss:42} and \S\ref{ss:44}, respectively.
    \item[Section \ref{sec:sec5}.] We study a differential calculus induced by the structure of $\sfd$. In \S\ref{ss:51}-\ref{ss:52}, we prove the equivalence between solutions to Vlasov's equations and a class of \emph{physical}~$\mathsf d$-absolutely continuous curves. In \S\ref{ss:53}, we compute the $\mathsf{d}$-derivative of solutions to Vlasov's equations. Moreover, we show that, along such curves, the optimal transport plans are \emph{tangent} to the curve itself. Finally, in \S\ref{sec:repar}, we extend the result to reparametrisations of solutions to Vlasov's equations.
\end{description}

\subsection{Definitions and main results}\label{ss:ss11}

\paragraph{Static formulation}

Set~$\mathcal X \coloneqq \R^n$ and~$\mathcal V \coloneqq \R^n$, and let the phase space be~$\Gamma :=  \mathcal X \times \mathcal V$. Let~$\mathcal{P}_2(\Gamma)$ be the set of probability measures~$\mu \in \mathcal{P}(\Gamma)$, such that the second-order moments of~$\mu$ are finite, i.e.,
\[
 \int_\Gamma \left( |x|^2 + |v|^2 \right) \, \dif \mu(x,v) < \infty \fstop
\]
We aim at defining a minimal acceleration discrepancy between measures~$\mu,\nu \in \cP_2(\Gamma)$. 
Let us start with the case of Dirac masses~$\mu=\delta_{(x,v)}$ and~$\nu=\delta_{(y,w)}$. We can see the squared Euclidean distance between $x$ and $y$ as a variational problem where we minimise the integral of the squared velocity for all paths $\alpha$ joining $x$ and $y$ in one unit of time:
\begin{equation} \label{eq:infder}
	\abs{y-x}^2
		=
	\inf_{\alpha \in H^1(0,1;\cX)} \set{ \int_0^1 \abs{\alpha'(t)}^2 \, \dif t \quad \text{ subject to } \alpha(0) = x \text{ and } \alpha(1) = y } \fstop
\end{equation}
Therefore, one reasonable definition for an acceleration-based discrepancy would be
\begin{equation}\label{eq:naiveacc}
\inf_{\alpha \in \mathrm{H}^{2}(0,1;\cX)} \set{ \int_0^1 \abs{\alpha''(t)}^2 \, \dif t \quad \text{ subject to } (\alpha,\alpha')(0) = (x,v) \text{ and } (\alpha,\alpha')(1) = (y,w) } \comma
\end{equation}
namely, we compute the minimal squared~$\Lp{2}$-norm of a force~$F_t$ that moves~$(x,v)$ to~$(y,w)$ in one unit of time, under Newton's law
\[
		\dot x_t = v_t \comma \quad
		\dot v_t = F_t \fstop
\]
However, unlike in the first-order case, the choice of the time interval~$[0,1]$ is now arbitrary. Indeed, while we can write
\begin{equation} \label{eq:xyT}
	\abs{y-x}^2
		=
	\inf_{\alpha \in H^1(0,T;\cX)} \set{ T \int_0^T \abs{\alpha'(t)}^2 \, \dif t \quad \text{ subject to } \alpha(0) = x \text{ and } \alpha(1) = y }
\end{equation}
for every~$T > 0$, a direct calculation (see~\S\ref{sec:sec2}) shows that
\begin{multline} \label{eq:dynd}
	\inf_{\alpha \in \mathrm{H}^{2}(0,T;\cX)} \set{ T \int_0^T \abs{\alpha''(t)}^2 \, \dif t \quad \text{ s.t.~} (\alpha,\alpha')(0) = (x,v) \text{ and } (\alpha,\alpha')(T) = (y,w) } \\
		=
	12\abs{\frac{y-x}{T} - \frac{v+w}{2} }^2 + \abs{w-v}^2 \eqqcolon \tilde d^2_T\bigl((x,v),(y,w)\bigr) \comma
\end{multline}
which is \emph{not} independent of~$T$. %

Thus, we introduce a relaxed version of \eqref{eq:naiveacc}, where the time parameter is an additional resource to optimise:
\begin{equation}\label{eq:naiveacc2}
	\tilde d\bigl((x,v),(y,w) \bigr) 	\coloneqq \inf_{T>0} \tilde d_T\bigl((x,v),(y,w) \bigr) \fstop
\end{equation}
Problem~\eqref{eq:naiveacc2} admits a solution (see~\S\ref{sec:sec2}). 
Namely, for all $(x,v), (y,w) \in \Gamma$,
\begin{equation}
\begin{aligned}
\label{eq:st1.}
	\tilde \dist^2\bigl( (x,v), (y,w) \bigr) 
	= 
	\begin{cases}
		3 \abs{v+w}^2 - 3\left(\frac{y-x}{\abs{y-x}}\cdot (v+w) \right)_+^2 + \abs{w-v}^2 
	& \text{if } x \neq y 
		\comma 
	\\
		3 \abs{v+w}^2 + \abs{w-v}^2  
	& \text{if } x=y 
	\fstop
	\end{cases}
\end{aligned}
\end{equation}
This quantity is not lower-semicontinuous, but we can write its lower-semicontinuous envelope explicitly: for $(x,v), (y,w) \in \Gamma$,
\begin{equation}
	\begin{aligned}	
		\label{eq:st2.}
		\dist^2\bigl((x,v), (y,w) \bigr) = 
			\begin{cases}
				3 \abs{v+w}^2 - 3\left(\frac{y-x}{\abs{y-x}}\cdot (v+w) \right)_+^2 + \abs{w-v}^2 &\text{if } x \neq y \comma \\
				\abs{w-v}^2  &\text{if } x=y \fstop
			\end{cases}
	\end{aligned}
\end{equation}	
This function~$d$ is our second-order discrepancy in the case of Dirac deltas. Notice that~$d$ and~$\tilde d$ are not distances. A collection of their properties is given in~\S\ref{sec:sec2}.

For general probability measures~$\mu,\nu \in \cP_2(\Gamma)$, we define~$\tilde \sfd_T(\mu,\nu)$,~$\tilde \sfd(\mu,\nu)$, and~$\sfd(\mu,\nu)$, by optimising over couplings (or transport plans)~$\pi$ as follows:

        \begin{align}
        	\label{eq0.} 
        \tilde \sfd_T^2(\mu,\nu)
        	&\coloneqq
        \inf_{\pi \in \Pi(\mu,\nu)} \left( 12\norm{\frac{y-x}{T}-\frac{v+w}{2}}^2_{\Lp{2}(\pi)}+\norm{w-v}^2_{\Lp{2}(\pi)} \right) \comma \\
        \label{eq1.} 
        \tilde \sfd^2(\mu,\nu) &\coloneqq \inf_{\pi \in \Pi(\mu,\nu)} \begin{cases}  3  \norm{v+w}^2_{\Lp{2}(\pi)} - 3\frac{\bigl((y-x,v+w)_\pi\bigr)_+^2}{\norm{y-x}^2_{\Lp{2}(\pi)}}  + \norm{w-v}^2_{\Lp{2}(\pi)} &\text{if }  \norm{y-x}_{\Lp{2}(\pi)} > 0 \comma \\  3\norm{v+w}^2_{\Lp{2}(\pi)} + \norm{w-v}^2_{\Lp{2}(\pi)}, &\text{if } \norm{y-x}_{\Lp{2}(\pi)}  = 0 \comma
        \end{cases} \\
        \label{eq2.} 
        \sfd^2(\mu,\nu) &\coloneqq \inf_{\pi \in \Pi(\mu,\nu)} \begin{cases}  3  \norm{v+w}^2_{\Lp{2}(\pi)} - 3\frac{\bigl((y-x,v+w)_\pi\bigr)_+^2}{\norm{y-x}^2_{\Lp{2}(\pi)}}  + \norm{w-v}^2_{\Lp{2}(\pi)} &\text{if }  \norm{y-x}_{\Lp{2}(\pi)} > 0 \comma \\  \norm{w-v}^2_{\Lp{2}(\pi)}, &\text{if } \norm{y-x}_{\Lp{2}(\pi)}  = 0 \comma
        \end{cases}
        \end{align}
where~$(\cdot,\cdot)_\pi$ is the scalar product in~$\Lp{2}(\pi)$, and%
\[
	\Pi(\mu,\nu)
		=
	\set{ \pi \in \mathcal{P}(\Gamma \times \Gamma) \, : \, (\proj_{x,v})_\# \pi = \mu \comma \quad (\proj_{y,w})_\# \pi = \nu} \fstop
\]
Observe that~\eqref{eq0.}-\eqref{eq2.} define finite non-negative quantities for every choice of~$\mu,\nu \in \mathcal{P}_2(\Gamma)$. This is proved as in the classical OT theory, by testing with~$\pi \coloneqq \mu \otimes \nu \in \Pi(\mu,\nu)$.
Another observation from OT theory is that~\eqref{eq0.} admits a minimiser for all~$\mu,\nu$.

Our first result establishes the existence of optimisers for~\eqref{eq2.} and, under the assumption of absolute continuity for~$\mu$, of an optimal transport map, in analogy with the classical OT theory~\cite{santambrogio2015optimal,MR2401600,MR2459454}.%

\begin{thm}[Optimal plans and maps]\label{thm1}
	The following statements hold.
\begin{enumerate}
	\item \label{thm1.1} (\Cref{prop:thm1.1}) We have
	\begin{equation} \label{eq:dinfT}
		\tilde \sfd(\mu,\nu) = \inf_{T > 0} \tilde \sfd_T(\mu,\nu) \comma \qquad \mu,\nu \in \mathcal{P}_2(\Gamma) \fstop
	\end{equation}
	\item \label{thm1.2} (\Cref{prop:thm1.2}) The second-order discrepancy~$\sfd$ is the lower-semicontinuous envelope of~$\tilde \sfd$ with respect to the $2$-Wasserstein distance on~$\mathcal{P}_2(\Gamma)$. 
	\item \label{thm1.3} (\Cref{prop:thm1.3}) For all~$\mu,\nu \in \mathcal{P}_2(\Gamma)$, there exists a minimiser for~\eqref{eq2.}. 
	\item \label{thm1.4} (\Cref{prop:thm1.4}) If~$\mu$ is absolutely continuous with respect to the Lebesgue measure, then, there exists a $\sfd$-optimal transport map between~$\mu$ and~$\nu$, i.e.,~a measurable function~$M \colon \Gamma \to \Gamma$ such that~$M_\# \mu = \nu$ and~$\pi \coloneqq (\id, M)_\# \mu$ is a minimiser for~\eqref{eq2.}.
\end{enumerate}
\end{thm}

While~\Cref{prop:thm1.3} shows that~$\sfd$-optimal transport plan exist, we will see that this is not always the case for~$\tilde \sfd$, i.e.,~minimisers for~\eqref{eq1.} may not exist, see \Cref{ex:minimisertildesfd}.
In~\S\ref{sec:nonuniq}, we show that \emph{uniqueness} of $\sfd$-optimal kinetic transport plans (i.e.,~minimisers for~\eqref{eq2.}) and maps is not to be expected in general.

\paragraph{Dynamical formulations}

Even though~\eqref{eq0.} generalises~\eqref{eq:naiveacc2}---which is derived from the dynamical optimal control problem~\eqref{eq:dynd}---it is not immediate to recognise a minimal acceleration in the cost of~\eqref{eq0.}. However, there are at least two natural ways to generalise~$\tilde d_T$ to a discrepancy between probability measures via \emph{dynamical formulations}. In what follows, we discuss them and state our second theorem: these formulations are indeed equivalent to the static one.

Fix~$\mu,\nu \in \cP_2(\Gamma)$ and~$T > 0$. To build our first dynamical formulation, the idea is to take a mixture of curves~$(\alpha,\alpha') \colon [0,T] \to \Gamma$ connecting points of~$\supp(\mu)$ and~$\supp(\nu)$. Precisely, we consider measures~$\mathbf{m} \in \mathcal{P}\bigl(\mathrm H^2(0,T;\cX)\bigl)$ such that 
\begin{equation} \label{eq:bdryConds}
\left(\proj_{\alpha(0),\alpha'(0)}\right)_\# \mathbf{m} = \mu \comma \qquad \left(\proj_{\alpha(T),\alpha'(T)}\right)_\# \mathbf{m} = \nu \comma
\end{equation}
where $\proj_{\alpha(t),\alpha'(t)}$ denotes the evaluation map $\alpha \mapsto \bigl(\alpha(t),\alpha'(t)\bigr)$,
and we define
\begin{equation}
    \label{eq:dtp1.}
    \tilde \sfn^2_T(\mu,\nu)
    	\coloneqq
    \inf_{\mathbf{m} \in \mathcal{P}\bigl(\mathrm H^2(0,T;\cX)\bigr) } \set{  T \int_0^T \int \abs{\alpha''(t)}^2 \, \dif \mathbf{m}(\alpha) \, \dif t \quad \text{ subject to~\eqref{eq:bdryConds}}} \fstop
\end{equation}
The function~$\tilde \sfn_T$ is a natural generalisation of~$\tilde d_T$, i.e.,
\[
	\tilde d_T\bigl((x,v),(y,w)\bigr)
		=
	\tilde \sfn_T\left(\delta_{(x,v)},\delta_{(y,w)}\right) \comma \qquad (x,v) \comma (y,w) \in \Gamma \comma \quad T > 0 \fstop
\]

To write our second dynamical formulation, we observe that, for any~$\alpha \in \mathrm H^2(0,T; \cX)$ and~$\varphi \in \mathrm{C}_\mathrm{c}^\infty\bigl((0,T) \times \cX \times \cV\bigr)$, we have
\begin{align*}
	0
		&=
	\int_0^T \frac{\dif}{\dif t} \varphi\bigl(t,\alpha,\alpha'\bigr) \, \dif t = \int_0^T \bigl(\partial_t \varphi(t,\alpha,\alpha') + \alpha' \cdot  \nabla_x \varphi(t,\alpha,\alpha') + \alpha'' \cdot \nabla_v \varphi(t,\alpha,\alpha'') \bigr) \, \dif t \\
		&=
	\int_0^T \int_\Gamma \bigl(\partial_t \varphi(t,x,v) + v \cdot  \nabla_x \varphi(t,x,v) + \alpha''(t) \cdot \nabla_v \varphi(t,x,v) \bigr)  \, \dif \delta_{\bigl(\alpha(t),\alpha'(t)\bigr)}(x,v) \, \dif t \comma
\end{align*}
meaning that~$\mu_t \coloneqq \delta_{\bigl(\alpha(t),\alpha'(t)\bigr)}$ and~$F_t(x,v) \coloneqq \alpha''(t)$ satisfy Vlasov's equation
\begin{equation}
    \label{eq:vlasov.}
    \partial_t \mu_t + v\cdot \nabla_x \mu_t + \nabla_v \cdot (F_t \, \mu_t) = 0
\end{equation}
weakly in~$(0,T) \times \Gamma$. 
For given $\mu, \nu \in \mathcal{P}_2(\Gamma)$ we define the $T$-\emph{minimal action} as
\begin{equation} \label{eq:vlasov2.}
	\widetilde{\mathcal{MA}}^2_T(\mu,\nu)
		\coloneqq
	\inf_{(\mu_t,F_t)_{t \in [0,T]}} \set{ T \int_0^T \norm{F_t}_{\Lp{2}(\mu_t)}^2 \, \dif t \quad \text{s.t.~}\eqref{eq:vlasov.} \text{ and } \mu_0 = \mu \comma \mu_T = \nu } \comma
\end{equation}
which is reminiscent of the Benamou--Brenier formulation of the~$2$-Wasserstein distance~\cite{benamou2000computational}, see \eqref{eq:bb.} below. 
As before,
\begin{equation}
	\label{eq:equiv-Dirac}
	\tilde d_T\bigl((x,v),(y,w)\bigr)
		=
	\widetilde{\mathcal{MA}}_T \left(\delta_{(x,v)},\delta_{(y,w)}\right) \comma \qquad (x,v) \comma (y,w) \in \Gamma \comma \quad T>0 \fstop
\end{equation}
Indeed, the inequality~$\ge$ follows from the discussion above. To justify the converse: when a given curve~$(\mu_t,F_t)_{t \in (0,T)}$ solves~\eqref{eq:vlasov.}, then~$t \mapsto \alpha(t) \coloneqq \int_{\Gamma} x \, \dif \mu_t$ (formally) satisfies
\begin{align*}
	\alpha_i'(t)
		=
	\frac{\dif}{\dif t} \int_\Gamma x_i \, \dif \mu_t
		\stackrel{\eqref{eq:vlasov.}}{=}
	\int_\Gamma \bigl( v \cdot \nabla_x x_i + F_t \cdot \nabla_v x_i \bigr) \, \dif \mu_t = \int_\Gamma v_i \, \dif \mu_t \comma \qquad i \in \set{1,\dots,n} \comma \\
	\alpha_i''(t)
	=
	\frac{\dif}{\dif t} \int_\Gamma v_i \, \dif \mu_t
	\stackrel{\eqref{eq:vlasov.}}{=}
	\int_\Gamma \bigl( v \cdot \nabla_x v_i + F_t \cdot \nabla_v v_i \bigr) \, \dif \mu_t = \int_\Gamma (F_t)_i \, \dif \mu_t \comma \qquad i \in \set{1,\dots,n} \comma
\end{align*}
and, by Jensen's inequality,~$\abs{\alpha''(t)} \le \norm{F_t}_{\Lp{2}(\mu_t)}$ for all~$t \in (0,T)$, which yields $\leq$ in \eqref{eq:equiv-Dirac}. 
The following result extends \eqref{eq:equiv-Dirac} from Dirac measures to all of $\mathcal{P}_2(\Gamma)$. 
\begin{thm}[Equivalence of static and dynamic formulations]
    \label{thm2}
    For every~$\mu,\nu \in \mathcal{P}_2(\Gamma)$ and~$T > 0$, the problems \eqref{eq:dtp1.} and \eqref{eq:vlasov2.} admit a minimiser. Moreover, we have the identities
    \begin{equation}
    \tilde \sfn_T(\mu,\nu)
    	=
    \widetilde{\mathcal{MA}}_T(\mu,\nu)
    	=
    \tilde \sfd_T(\mu,\nu) \fstop
    \end{equation}
\end{thm}
This result is proved in two steps, corresponding to Theorem \ref{lifthm} and Theorem \ref{thmbb}.
After posting a first version of this manuscript on arXiv, we were informed of the preprint~\cite{elamvazhuthi2025benamoubrenierformulationoptimaltransport} by K.~Elamvazhuthi---building on a previous work~\cite{ElamvazhuthiLiuLiOsher}---which contains a generalised version of the second equality in \Cref{thm2}, in the context of optimal control systems. In \Cref{thm2}, we prove further equivalence with the formulation $\tilde{ \mathsf n}_T$, and existence of minimisers for all three problems. Distinctive features of our approach are an original kinetic Monge--Mather principle (cf.~\Cref{prop:inj} and \Cref{injectiveMap}) and the regularisation of solutions to Vlasov's equation via Galilean convolution (cf.~\Cref{approxlem}), which may be of independent interest. %

\paragraph{A variational characterisation of Vlasov's equations}
In the classical optimal transport theory, the Benamou--Brenier formula is constrained by the continuity equation~$\partial_t \rho_t + \nabla \cdot (V_t \rho_t) = 0$, for a velocity field~$V_t$. Solutions to the continuity equation on a bounded open interval $(a,b)$ turn out to coincide with absolutely continuous curves in the Wasserstein space, under appropriate integrability conditions \cite{MR2401600,santambrogio2015optimal}. Although~$\sfd$ is not a distance, we will give a similar characterisation for solutions to Vlasov's equations. 

\begin{defn}[Physical curves]
    Let $(\mu_t)_t : (a,b) \to \Gamma$ be a~$2$-Wasserstein absolutely continuous curve (see \S\ref{sss121}). We say that $(\mu_t)_t$ is physical if, in addition, for all $s<t \, \in(a,b)$, and for a function $\ell \in \mathrm{L}_{\geq0}^2(a,b)$, it holds true that
    \begin{equation}
        \label{eq:phys}
        \tilde\sfd_{t-s}(\mu_s,\mu_t) \leq \int_s^t \ell(r) \, \dif r  \fstop
    \end{equation}
    
\end{defn}

Heuristically, the physicality condition for a curve $( \mu_t)_t$ yields some differentiable control in the velocity marginal $(\proj_v)_\# \mu_t$, together with the fact that the variation of the spatial marginal $(\proj_x)_\# \mu_t$ is given by $v$ (hence, $\mu_t$ would solve \eqref{eq:vlasov.}).
This idea is made rigorous in the next result.

\begin{thm}[Identification of the tangent I: physical curves and Vlasov's equations]\label{thm:main:new:1} %
The following hold true.
\begin{enumerate}
        \item If~$(\mu_t,F_t)_t$ is a weak solution to \eqref{eq:vlasov.} on $( a, b)$ for some force field~$( F_t)_t$ such that 
    \begin{equation}
    \int_{ a}^{b} \left( \norm{v}_{\Lp{2}(\mu_t)}^2 + \norm{F_t}_{\Lp{2}( \mu_t)}^2 \right) \, \dif t < \infty \comma
    \end{equation}
    then the curve~$(\mu_t)_t$ is~physical with
    \begin{equation} \ell(t) = 2\, \|F_t\|_{\mathrm{L}^2( \mu_t)} \fstop
    \end{equation}
    \item	Assume that~$(\mu_t)_{t \in (a,b)}$ is a physical  curve. 
    Then, there exists a vector field~$( F_t)_t$ with~$\|F_t\|_{\mathrm L^2(\mu_t)} \leq \ell(t)$ for a.e.~$t \in (a,b)$, such that~$(\mu_t,F_t)_t$ is a weak solution to Vlasov's equation~\eqref{eq:vlasov.} and we have the limit
	\begin{equation}
	\lim_{h \downarrow 0} \frac{\tilde \sfd_{h}(\mu_t, \mu_{t+h})}{h} = \norm{ F_t}_{\Lp{2}(\mu_t)}
	\end{equation}
	for a.e.~$t \in (a,b)$.
    \end{enumerate}%
\end{thm}

\noindent The proof of this result can be found in \S\ref{sec:sec5}, as a combination of \Cref{prop:vlasovtoac1}, \Cref{cor:actovlasov_physical}, and \Cref{prop:metricder_bis}.

\paragraph{Hypoelliptic Riemannian structure}

The class of solutions to Vlasov's equations~\eqref{eq:vlasov.} is rather rigid, as it is not closed under Lipschitz time-reparametrisation. 
The latter is a desirable property for ``absolutely continuous'' curves, which we define below.

\begin{defn}[$\sfd$-absolutely continuous curves]
    Let~$(\tilde \mu_s)_{s \in (\tilde a,\tilde b)}$ be a~$2$-Wasserstein absolutely continuous curve (see \S\ref{sss121}). We say that~$(\tilde \mu_s)_{s \in (\tilde a, \tilde b)}$ is~$\sfd$-absolutely continuous if there exists a function~$\tilde \ell \in \Lp{2}_{\ge 0}(\tilde a,\tilde b)$ such that for every~$s,t \in (\tilde a, \tilde b)$ with~$s<t$, we have \begin{equation}\label{eq:2ac.}
	   \sfd( \tilde \mu_s, \tilde  \mu_t) \leq \int_s^t \tilde \ell(r) \, \dif r \fstop
	\end{equation}
\end{defn}

All physical curves are $\sfd$-absolutely continuous. The converse is not true, e.g., a time-reparametrisation of a physical curve is still absolutely continuous (but not physical). 
We may wonder how general this example is and, consequently, how large the class of absolutely continuous curves is compared to that of physical curves.

We find that, under a suitable regularity condition (Assumption~\ref{defn:ac}), $\sfd$-absolutely continuous curves coincide with the closure of physical curves under regular reparametrisations in time.
Heuristically, $\sfd$-absolute continuity is enough to have a differentiable control on the velocity marginal for a curve $(\tilde\mu_s)_s$, together with the fact that the variation of the space marginal of $\tilde\mu_s$ is positively proportional to $v$, i.e, it amounts to $\tilde \lambda(s) v$, for some~$\tilde \lambda(s) \ge 0$ (independent of the space-velocity variables). The proportionality factor~$\tilde \lambda(s)$ can be renormalised to one via a time reparametrisation.%

The factor $\tilde \lambda(s)$ is related to the infinitesimal ratio between the optimal time horizon in the definition of~$\sfd(\tilde \mu_s,\tilde \mu_{s+\tilde h})$ and the physical time~$\tilde h$.
More precisely, given $s,\tilde h$, we define~$T(s,s+\tilde h) \coloneqq \argmin_T \, \tilde{\sfd}_T(\tilde \mu_s,\tilde \mu_{s+\tilde h})$. Then, whenever $\tilde h \mapsto T(s,s+\tilde h)$ is right-differentiable at~$0$, we have that $\tilde \lambda(s)$ is finite, and 
\[
\tilde \lambda(s) = \lim_{\tilde h \to 0^+} \frac{T(s,s+\tilde h)}{\tilde h}. 
\]
We state these ideas precisely in \Cref{thm:main:new:2} below.%

\begin{rem}
	If~$(\mu_t,F_t)_{t \in (a,b)}$ solves Vlasov's equation~\eqref{eq:vlasov.}, and~$\tau \colon (\tilde a, \tilde b) \to (a, b)$ is a bi-Lipschitz reparametrisation, then the curve~$s \mapsto \tilde \mu_s \coloneqq \mu_{\tau(s)}$ solves the reparametrised Vlasov equation
	\begin{equation} \label{eq:vlasov3.}
		\partial_s \tilde \mu_s + \tilde \lambda(s) \, v \cdot \nabla_x \tilde \mu_s + \nabla_v \cdot (\tilde F_s \tilde \mu_s) = 0 \comma \qquad s \in (\tilde a, \tilde b) \comma
	\end{equation}
	with~$\tilde \lambda(s) \coloneqq \tau'(s)$ and~$\tilde F_s \coloneqq \tilde \lambda(s) F_{\tau(s)}$.
\end{rem}

\begin{thm}[Identification of the tangent II: $\sfd$-absolutely continuous curves, $\sfd$-derivative]  \label{thm:main:new:2}
The following hold true.
    \begin{enumerate}
        \item Assume that~$(\tilde \mu_s,\tilde F_s)_s$ is a weak solution to \eqref{eq:vlasov3.} on $(\tilde a,\tilde b)$ for some force field~$(\tilde F_s)_s$ such that 
    \begin{equation}
    \int_{\tilde a}^{\tilde b} \left( \norm{v}_{\Lp{2}(\tilde \mu_s)}^2 + \norm{\tilde F_s}_{\Lp{2}(\tilde \mu_s)}^2 \right) \, \dif s < \infty \comma
    \end{equation}
	and for a function~$\tilde \lambda$ bounded from above and below by positive constants.
	
    If the Wasserstein metric derivative of the spatial marginal 
		$\tilde \rho_s(\cdot) \coloneqq \tilde \mu_s(\cdot \times \mathcal{V})$
		satisfies
	\begin{equation} \label{eq:pos_der}
    \abs{\tilde \rho_s'}_{\mathrm{W}_2} > 0 \quad \text{for a.e. } s \in (\tilde a,\tilde b) \comma   
    \end{equation}
    then, the curve~$(\tilde\mu_s)_s$ is~$\sfd$-absolutely continuous and satisfies \Cref{defn:ac}, with
    \begin{equation} \tilde \ell(s) = 2\, \|\tilde F_s\|_{\mathrm{L}^2(\tilde \mu_s)}  \quad \text{and} \quad \tilde \lambda_\mathrm{ac}(s)=\tilde \lambda(s) \comma \qquad s \in (\tilde a, \tilde b)\fstop
    \end{equation}
    \item Assume that~$(\tilde \mu_s)_{s \in (\tilde a,\tilde b)}$ is a~$\mathsf d$-absolutely continuous curve satisfying \Cref{defn:ac}.
    If~$\tilde \rho_s$
	satisfies~\eqref{eq:pos_der},
    then there exists a vector field~$(\tilde F_s)_s$ with~$\|\tilde F_s\|_{\mathrm L^2(\tilde \mu_s)} \leq \tilde \ell(s)$ for a.e.~$s \in (\tilde a, \tilde b)$, such that~$(\tilde \mu_s,\tilde F_s)_s$ is a solution to \eqref{eq:vlasov3.} with~$\tilde \lambda = \tilde \lambda_\mathrm{ac}$, %
	and we have the limit
	\begin{equation}
	\lim_{\tilde h \downarrow 0} \frac{\sfd(\tilde \mu_s, \tilde \mu_{s+\tilde h})}{\tilde h} = \norm{\tilde F_s}_{\Lp{2}(\tilde \mu_s)}
	\end{equation}
	for a.e.~$s \in (\tilde a, \tilde b)$.
    \end{enumerate}%
\end{thm}
\noindent The proof of this result can be found in \S\ref{sec:sec5}, see in particular \Cref{thm:rep}.

\begin{rem}[The flow velocity]
	In both statements in \Cref{thm:main:new:2}, we assume positivity a.e.~of the quantity $\abs{\tilde \rho'_s}_{\mathrm{W}_2}$, i.e., of the Wasserstein metric derivative of the spatial density $(\tilde \rho_s)_s.$ This can be interpreted as macroscopic non-steadiness of the system. Using the theory of optimal transport (cf.~\cite{MR2401600} and \Cref{lem:derW2} below) it is possible to prove the following. When~$(\tilde \mu_s)_s$ solves the reparametrised Vlasov equation~\eqref{eq:vlasov3.} the metric derivative~$\abs{\tilde \rho_s'}_{\mathrm{W}_2}$ is equal to the~$\Lp{2}(\tilde \rho_s)$-norm of the irrotational part of the vector~$\tilde \lambda(s) \, \tilde j_s(x)$, where~$\tilde j_s(x)$ is the \emph{flow velocity} 
	\[
	\tilde j_s(x) \coloneqq \int_{\cV} v \, \dif \tilde \mu_s(x,v) \fstop
	\]
	Notice that~$(\tilde \rho_s, \tilde j_s)_s$ solves Euler's equation
	\[
		\partial_s \tilde \rho_s + \tilde \lambda(s) \nabla_x \cdot \tilde j_s = 0 \fstop
	\] 
\end{rem}

\begin{rem}[The tangent cone I: admissible directions]\label{rem:tgc1}
\Cref{thm:main:new:2} asserts that $\sfd$-absolutely continuous curves are identified with solutions to \eqref{eq:vlasov3.}, which in turn are induced by vector fields $(\tilde\lambda(s)v,\tilde F_s)_{s\in(\tilde a,\tilde b)}$.
A narrowly continuous curve of measures $(\tilde \mu_s)_s$ solves \eqref{eq:vlasov3.} for two different vector fields $(\tilde F_s)_s$ and $(\tilde G_s)_s$ if and only if, by linearity of~\eqref{eq:vlasov3.}, we have $\nabla_v \cdot \bigl((\tilde F_s -\tilde G_s)\,\tilde \mu_s\bigr) = 0$ in the sense of distributions, that is, if and only if~$\tilde F_s$ and~$\tilde G_s$ have the same projection onto 
\[
\mathrm{cl}_{\mathrm{L}^2(\tilde \mu_s)}\set{ \nabla_v \varphi \, : \, \varphi \in \mathrm{C}^1_\mathrm{c}(\Gamma)} \fstop
\]
As in the classical case, see \S\ref{sss121}, we can define the hypoelliptic tangent cone at $\mu$ as
\begin{equation} \label{eq:formaltangent}
\cT_{\mu,\sfd} \mathcal{P}_2(\Gamma) := \mathrm{cl}_{\mathrm{L}^2(\mu)} \bigl\{ (\lambda \, v,\nabla_v \varphi) \, :\, \lambda \in \mathbb R_{\geq 0},\ \varphi\in \mathrm{C}^1_\mathrm{c}(\Gamma)\bigr\} \comma
\end{equation}
and equip $\cT_{\mu,\sfd} \mathcal{P}_2(\Gamma)$ with the degenerate Riemannian form
\begin{equation} \label{eq:deg_riem}
\big\langle (\lambda^{(1)} v, F^{(1)}) , (\lambda^{(2)} v, F^{(2)}) \big\rangle_{\cT_{\mu,\sfd} \mathcal{P}_2(\Gamma)} := \int F^{(1)} \cdot F^{(2)} \, \dif\mu \fstop 
\end{equation}
Finally, recall that Equation \eqref{eq:vlasov3.} is the closure under time-reparametrisation of \eqref{eq:vlasov.}. At the geometric level, we formally interpret this fact as follows. The hypoelliptic tangent $\cT_{\mu,\sfd} \mathcal{P}_2(\Gamma)$ is the conical envelope of the vectors $\set{(v,\nabla_v \varphi), \quad \varphi\in \mathrm{C}^1_\mathrm{c}(\Gamma)},$ which correspond exactly to the vector fields inducing \eqref{eq:vlasov.}. See also Remark \ref{rem:tgc2} below.
\end{rem}

\begin{ass}[Regularity] \label{defn:ac}
	Let~$(\tilde \mu_s)_{s \in (\tilde a,\tilde b)}$ be a $\sfd$-absolutely continuous (hence $2$-Wasserstein a.c.)~curve, and define the open set of times
	\begin{equation}
	\tilde \Omega\coloneqq\set{s \in (\tilde a, \tilde b) \, : \,\|v\|_{\mathrm{L}^2(\tilde \mu_s)}>0} \fstop
	\end{equation}
	We assume that there exist
	\begin{enumerate}
		\item a measurable selection~$(s,t) \mapsto \tilde \pi_{s,t} \in \Pi(\tilde \mu_s, \tilde \mu_t)$ of~$\sfd$-optimal transport plans, i.e.,~minimisers in~\eqref{eq2.},
		\item a measurable function~$\tilde \lambda_\mathrm{ac} \colon (\tilde a, \tilde b) \to \R_{>0}$ bounded from above and below by positive constants,
	\end{enumerate}
	such that, defining the \emph{optimal time}\footnote{This is a minimiser for~\eqref{eq:dinfT} between~$\tilde \mu_s$ and~$\tilde \mu_t$, see \Cref{prop:thm1.1}}
        \begin{equation}
        	\tilde T_{s,t} = \begin{cases}
        		2 \displaystyle \frac{\norm{y-x}_{\Lp{2}(\tilde \pi_{s,t})}^2}{(y-x,v+w)_{\tilde \pi_{s,t}}} &\text{if } (y-x,v+w)_{\tilde \pi_{s,t}} > 0 \comma \\
        		0 &\text{if } \norm{y-x}_{\Lp{2}(\tilde \pi_{s,t})} = 0 \comma \\
        		\infty &\text{otherwise,}
        	\end{cases}\qquad \tilde a<s<t<\tilde b \comma
        \end{equation}
        then, the convergence \begin{equation}
		 \frac{\tilde T_{s,s+\tilde h}}{\tilde h}  \to \tilde \lambda_\mathrm{ac}(s) \quad \text{as } \tilde h\downarrow 0\comma  \end{equation}
	holds for a.e.~$s \in (\tilde a,\tilde b)$, with $\mathrm{L}^1$-domination on every compact subset of~$\tilde \Omega$.
\end{ass}

We conclude by stating an immediate consequence of \Cref{thm:main:new:1} and \Cref{thm:main:new:2}.

\begin{cor}
    Let~$(\mu_t)_{t}$ be a~$2$-Wasserstein absolutely continuous curve with~$\abs{\rho_t'}_{\mathrm{W}_2} > 0$ for a.e.~$t$. The curve~$(\mu_t)_t$ is physical if and only if:
    \begin{enumerate}
        \item it is~$\sfd$-absolutely continuous, and
        \item it satisfies \Cref{defn:ac} with~$\tilde \lambda_\mathrm{ac} \equiv 1$.
    \end{enumerate}
\end{cor}

\subsection{Motivation, related contributions, and perspectives}\label{ss:ss12}
In this section, we review the recent literature that motivated or inspired our construction. We also give a perspective on future developments after this work in~\S\ref{sec:perspectives}.
\subsubsection{Comparison with standard optimal transport} \label{sss121}

\paragraph{Optimal transport (OT)} Let~$\rho_0,\rho_1 \in \mathcal{P}_2(\cX)$. One way to define the standard $2$-Wasserstein distance between~$\rho_0$ and~$\rho_1$ is given by
\begin{equation}\label{eq:ot.}
\mathrm{W}_2(\rho_0,\rho_1) \coloneqq \inf_{\pi \in \Pi(\rho_0,\rho_1)} \, \sqrt{\int \abs{y-x}^2 \, \dif \pi(x,y)} \comma
\end{equation}
where~$\Pi(\rho_0,\rho_1)$ is the set of all couplings~$\pi \in \cP_2(\cX \times \cX)$ of~$\rho_0$ and~$\rho_1$.
This variational problem, in which the average distance between coupled points~$x \in \mathrm{supp}(\rho_0)$ and~$y \in \mathrm{supp}(\rho_1)$ is minimised, is known as \emph{Kantorovich formulation}. The existence of minimisers is classical \cite{santambrogio2015optimal} and they are referred to as \emph{optimal transport plans}. Under mild conditions on~$\rho_0$,  it is also possible to establish existence of an \emph{optimal transport map} between~$\rho_0$ and~$\rho_1$, i.e.,~a function~$M \colon \cX \to \cX$ for which~$\pi \coloneqq (\id,M)_\# \rho_0$ is an optimal plan~\cite{brenier1991polar}. The optimal transport map is~$\rho_0$-a.e.~uniquely determined and can be found by solving a Monge--Ampère equation. Its regularity is a major research topic \cite{santambrogio2015optimal,figalli2017monge}.

Recalling~\eqref{eq:infder}, $\abs{y-x}^2$ is the squared length of the line joining~$x$ with~$y$ that minimises~$\int_0^1 \abs{\alpha'(t)}^2\, \dif t$, among all~$\mathrm{H}^1$-regular curves~$\alpha \colon (0,1) \to \cX$ with~$x$ and~$y$ as endpoints. Thus, the~$2$-Wasserstein distance can also be written in its \emph{dynamical formulation}
\begin{multline}\label{eq:dot.}
	\mathrm{W}_2^2(\rho_0,\rho_1) = \inf \Biggl\{ \int_0^1 \int \abs{\alpha'(t)}^2 \, \dif \mathbf{m} \, \dif t \comma \quad \text{s.t. } \mathbf{m} \in \mathcal{P}\bigl(H^1(0,1;\cX)\bigr) \comma \\ \left(\proj_{\alpha(0)}\right)_\# \mathbf{m} = \rho_0 \comma \left(\proj_{\alpha(1)}\right)_\# \mathbf{m} = \rho_1 \Biggr\} \fstop
\end{multline}
Moreover, J.-D.~Benamou and Y.~Brenier~\cite{benamou2000computational} provided the following fluid-mechanical characterisation:
\begin{multline}\label{eq:bb.}
\mathrm{W}_2^2(\rho_0,\rho_1) = 
\inf \Biggl\{ \int_0^1 \norm{V_t}^2_{\Lp{2}(\rho_t)} \, \dif t \comma \quad \text{s.t.~} \partial_t \rho_t + \nabla_x \cdot (V_t \rho_t) = 0 \text{ in } \mathcal{D}^\star\bigl((0,1) \times \cX\bigr) \comma \\ \rho_{t=i} = \rho_i \text{ for } i=0,1 \Biggr\}.
\end{multline}
The idea is that the characteristic ODE of the continuity equation~$\partial_t \rho_t + \nabla_x \cdot (V_t \rho_t) = 0$ in~\eqref{eq:bb.} is~$\dot x_t = V_t$. Thus, the squared norm of the velocity field~$(V_t)_t$ is equal to the average squared path-wise speed, cf.~\cite[Chapter 8]{MR2401600}. 
As it turns out, the curves~$(\rho_t)_t$ solving the continuity equation for some vector field~$(V_t)_t$ such that~$\int_0^1 \norm{V_t}_{\Lp{2}(\rho_t)}^2 \, \dif t < \infty$ are exactly the~$\mathrm{W}_2$-$2$-absolutely continuous curves~\cite{MR2401600}, i.e.,~those satisfying
\begin{equation}
	\mathrm{W}_2(\rho_s,\rho_t) \le \int_s^t \ell(r) \, \dif r \comma \qquad 0<s<t<1 \fstop
\end{equation}
for some function~$\ell \in \Lp{2}_{\ge 0}(0,1)$. For such curves, the \emph{metric derivative} is
\begin{equation}
	\abs{\rho_t'}_{{\mathrm{W}_2}} \coloneqq \lim_{h \to 0} \frac{\mathrm{W}_2(\rho_t,\rho_{t+h})}{h} = \norm{V_t}_{\Lp{2}(\rho_t)} \comma \qquad \text{for a.e.~} t \in (0,1) \comma
\end{equation}
if~$(\rho_t,V_t)_t$ solves the continuity equation and~$V_t$ is chosen in the~$\Lp{2}(\rho_t)$-closure of the set~$\set{\nabla_x \phi \, : \, \phi \in \mathrm{C}^1_\mathrm{c}(\cX)}$. Formally, the distance~$\mathrm{W}_2$ induces a Riemannian structure~\cite{otto2001geometry}:
\begin{equation} \label{eq:riem_W2}
\cT_\rho \mathcal{P}_2(\cX) \coloneqq \mathrm{cl}_{\Lp{2}(\rho)} \set{ \nabla_x \phi \, : \, \phi \in \mathrm{C}^1_\mathrm{c}(\cX)} \comma \qquad \langle F , G \rangle_{\cT_\rho \mathcal{P}_2(\cX)} \coloneqq \int F \cdot G \, \dif \rho \fstop
\end{equation}
\paragraph{Comparison of OTIKIN and OT}
\begin{itemize}
    \item At the level of static problems (i.e.,~optimal transport plans), existence of minimisers is true for both OT (i.e.,~in the problem~\eqref{eq:ot.} defining~$\mathrm{W}_2$) and OTIKIN (i.e.,~in~\eqref{eq2.}, defining~$\sfd$). For both, also existence of an optimal transport map holds under the assumption of absolute continuity of the starting measure w.r.t.~Lebesgue (see~\Cref{prop:thm1.4}), but uniqueness in OTIKIN does not hold, see~\S\ref{sec:nonuniq}.

	\item By minimising the acceleration as in~\eqref{eq:dtp1.}, we find a discrepancy~$\tilde \sfn_T = \tilde \sfd_T$ that depends on the time parameter~$T$. The non-parametric discrepancy~$\sfd$ is then found by optimising in~$T$ and taking the~$\mathrm{W}_2$-relaxation, see \Cref{thm1}. Note that this relaxation is necessary in order to ensure existence of optimal transport plans, see \Cref{prop:thm1.3} and \Cref{ex:minimisertildesfd}. In OT, the Wasserstein distance~$W_2$ is, instead, naturally non-parametric, in the sense that the choice of the time interval in~\eqref{eq:dot.} is inconsequential. This follows from our discussion in~\S\ref{ss:ss11}.
	
    \item In~\eqref{eq:dot.}, an optimiser exists and is supported on constant-speed straight lines connecting points~$x \in \mathrm{supp}(\rho_0)$ to points~$y \in \mathrm{supp}(\rho_1)$, according to the optimal coupling in~\eqref{eq:ot.}. In this case, there is no loss of generality in considering only curves on~$[0,1]$, see~\eqref{eq:xyT}. In \eqref{eq:dtp1.}, we have the existence of a minimiser~$\mathbf{m}_T$, and this measure is supported on cubic~$T$-splines \cite{benamou2019second} (see~\S\ref{sec:sec2}), for every~$T$. However, time-reparametrisations of~$\mathbf{m}_T$ on another interval~$[0,T']$ might no longer satisfy the desired boundary conditions. This happens even for one single \emph{spline} between two Dirac masses. Furthermore, even if a reparametrisation satisfies the boundary conditions, it may not be optimal for the problem with time~$T'$.
    \item While the class of continuity equations~$\partial_t \rho_t + \nabla_x \cdot (V_t \rho_t) =0$ is invariant under time-reparametrisation, the class of Vlasov's equations~$\partial_t \mu_t + v \cdot \nabla_x \mu_t + \nabla_v \cdot (F_t \mu_t) =0$ is not. %
    When working with Dirac deltas, the same observation arises by comparing the first-order ODE~$\dot x_t = V_t$ with the second-order ODE~$\dot x_t = v_t \comma \dot v_t = F_t$.
    \end{itemize}
    \begin{rem}[The tangent cone II: tangency of optimal plans]\label{rem:tgc1bis} %
In \cite[Chapter~8]{MR2401600}, the Wasserstein tangent space at $\rho \in \cP_2(\cX)$---denoted by $\cT_\rho\mathcal{P}_2(\mathcal X)$---is identified with the closure in $\mathrm{L}^2(\rho;\mathbb R^{d})$ of $\set{ \nabla_{x} \varphi \, : \, \varphi \in \mathrm{C}^1_\mathrm{c}(\mathcal X) }$. A time-dependent vector field~$V_t \in \cT_{\rho_t}\mathcal{P}_2(\mathcal X)$ can be interpreted as the velocity field of a curve solving~$\partial_t \rho_t + \nabla_{x} \cdot (\rho_t \, V_t) =0$ (which is a necessary and sufficient condition for a curve to be $\mathrm{W}_2$-$2$-a.c.).  
In \cite[Proposition~8.4.6]{MR2401600}, it is shown that, taken two measures~$\rho_t,\rho_{t+h}$ along such a curve, one has~$y - x - h \, V_t = o(h)$, as~$h \to 0$, on the support of any $\mathrm{W}_2$-optimal plan between~$\rho_t$ and~$\rho_{t+h}$. %

In our setting, we establish a similar structure. Let $(\mu_t)_t$ be a solution to the Vlasov equation \eqref{eq:vlasov.}, i.e. $\partial_t \mu_t + v\cdot \nabla_x \mu_t + \nabla_v \cdot (F_t \mu_t) = 0$, with~$(v,F_t) \in \cT_{\mu_t,\sfd} \cP_2(\Gamma)$ (see~\eqref{eq:formaltangent}). Note that, by \Cref{thm:main:new:1}, this is equivalent to physicality. In \S\ref{ss:tangent}, we will prove that, if $\pi_{t,t+h}$ is an optimal transport plan for~$\mathsf d(\mu_t,\mu_{t+h})$, then, $\pi_{t,t+h}$-almost everywhere, 
\[
\begin{aligned}
    &y-x-h\,v =o(h) \comma \\
    &w-v - h \, F_t(x,v) = o(h) \fstop 
\end{aligned}
\]
Whenever the \emph{total momentum} of $\mu_t$ is non-zero, we gain a further order of precision in our Taylor expansions on the support of $\pi_{t,t+h}$:
\[
y = x + hv + \frac{1}{2} h^2 F_t(x,v) + o(h^2) \fstop
\]
\end{rem}

    \begin{rem}[The tangent cone III: geometry of the tangent bundle]\label{rem:tgc2}
    We continue the formal geometric considerations of Remark \ref{rem:tgc1}. Formally, the $2$-Wasserstein distance induces a Riemannian structure on~$\mathcal{P}_2(\cX),$ with a clear identification of the tangent bundle \cite{otto2001geometry,MR2401600}. The discrepancy~$\mathsf d$ of OTIKIN yields a sort of \emph{hypoelliptic} Riemannian structure \cite{hormander1967hypoelliptic}. Vlasov's equation \eqref{eq:vlasov.} and its time-reparametrisations \eqref{eq:vlasov3.}, can be rewritten as~$\partial_s \tilde \mu_s + \nabla_{x,v} \cdot \bigl((\tilde \lambda(t) v,\tilde F_t) \, \tilde \mu_t\bigr) = 0$, which is a special case of the continuity equation~$\partial_s \tilde \mu_s + \nabla_{x,v} \cdot (X_s \tilde \mu_s) =0$ associated with the~$\mathrm{W}_2$-distance over~$\mathcal{P}_2(\Gamma),$ for a~$2n$-component vector field~$X_s : \Gamma \to \R^{2n}$. Thus, we can formally see the geometry of~$\mathsf d$---i.e.,~the hypoelliptic tangent cone~$\cT_{\tilde \mu_s,\sfd} \mathcal{P}_2(\Gamma)$ with the form \eqref{eq:deg_riem}---as a distribution of vectors in~$\cT_{\tilde \mu_s,\mathrm{W}_2} \cP_2(\Gamma)$, equipped with a degenerate version of~\eqref{eq:riem_W2} that measures only the acceleration component~$\tilde F_t$.%
   \end{rem}

    \begin{rem}[Comparison with sub-Riemannian optimal transport]
    A.~Figalli and L.~Rifford developed a theory for optimal transport on sub-Riemannian manifolds \cite{figalli2010mass}.   Thy consider a $m$-dimensional Riemannian manifold $(\mathcal M,\langle\cdot,\cdot\rangle$), equipped with a distribution of vector fields  $\Delta=\mathrm{span}\{X_1,\cdots,X_k\}$, with $k<m$, such that $\mathrm{Lie}(\Delta) = \cT\mathcal{M}$, i.e., the \emph{H\"ormander condition} \cite{hormander1967hypoelliptic} is satisfied. The Wasserstein distance is replaced by 
    \[
    \mathrm{W}^2_{\mathrm{SR},2}(\mu,\nu) :=\inf_{\pi \in \Pi(\mu,\nu)} \int \mathrm{d}_{\mathrm{SR}}^2(x,y) \, \dif \pi(x,y) \comma
    \]
    where $\mathrm{d}_{\mathrm{SR}}$ denotes the sub-Riemannian distance on $(\mathcal M, \langle\cdot,\cdot\rangle,\Delta)$. 
    Notice that~$\mathrm{d}_{\mathrm{SR}}$ is obtained via minimal length of curves tangent to $\Delta$, 
    \[
    \mathrm{d}_{\mathrm{SR}}^2(x,y) = \inf_{\substack{\alpha:[0,1] \to \mathcal{M} \\ \alpha' \in \Delta}} \, \int_0^1  \langle \alpha'(t), \alpha'(t)\rangle_{\alpha(t)} \, \dif t \fstop
    \]
    Our case is different, since $\sfd$-absolutely continuous curves are induced by vector fields tangent to\footnote{Note that~$\mathrm{Lie}(\Delta_{\mathrm{kin}}) = \cT \,\mathbb R^{2n}$, so that H\"ormander's condition holds.}
    \[
    \Delta_{\mathrm{kin}}:= \mathrm{span}\, \{Z,Y_1,\cdots,Y_n\} \comma\quad \text{with } Z = v\cdot \nabla_x \comma \quad Y_i = \frac{\partial}{\partial_{v_i}} \comma
    \]
    but we only measure the speed of curves in the directions $\{Y_i\}^n_{i=1}$, while the vector field $Z$ acts solely as a constraint. The resulting hypoelliptic geometry combines degenerate Riemannian with symplectic effects. 
    Heuristically, the difference between sub-Riemannian geometry and our hypoelliptic geometry is analogous to the distinction between H\"ormander operators of the \emph{first kind}, like the sub-elliptic Laplacian $\sum_{i=1}^k (X_i)^2$, and H\"ormander operators of the \emph{second kind}, like the Kolmogorov operator $Z + \sum_{i=1}^n \, (Y_i)^2$. 
    \end{rem}
    
\subsubsection{Minimal acceleration costs, kinetic Wasserstein, and related distances}\label{sss122}
Optimal transport with minimal acceleration cost has appeared in the context of variational schemes for fluid dynamics \cite{gangbo2009optimal,cavalletti2019variational}, see also \S\ref{sss124}. 
There, a discrete time-step~$T>0$ is fixed, and the authors consider both~$\tilde \sfd_T$ and
\begin{align}
	\mathbf{W}^2_T(\mu,\nu) &\coloneqq \inf_{\pi \in \Pi(\mu,\nu)} \int \left(12 \, \abs{ \frac{y-x}{T} - \frac{w-v}{2} }^2 + \abs{w-v}^2 \right) \, \dif \pi \comma \qquad \mu,\nu \in \mathcal{P}_2(\Gamma) \comma
\end{align}
which differ in the sign of~$\frac{w-v}{2}$.

For our purposes, the functionals~$\tilde \sfd_T$ and~$\mathbf{W}_T$ cannot be used as such, indeed:
\begin{enumerate}
	\item smooth curves~$t \mapsto \mu_t$ are not, in general,~$\tilde \sfd_T$-continuous, in the sense that
	\[	\liminf_{h \downarrow 0} \tilde \sfd_T(\mu_t,\mu_{t+h}) > 0  \fstop \]
    Therefore, absolute continuity is not meaningful for~$\tilde \sfd_T$;
    \item on the one hand, for all~$T>0$, the functional~$\mathbf{W}_T$ is a distance on~$\mathcal{P}_2(\Gamma)$, which is equivalent---with equivalence constants depending on $T$---to $\mathrm{W}_2$, as the cost function satisfies~$\mathsf w_T\bigl((x,v),(y,w)\bigr) \coloneqq 12 \, \abs{ \frac{y-x}{T} - \frac{w-v}{2} }^2 + \abs{w-v}^2 \asymp |(x,v)-(y,w)|^2$. However, the derivative of~$\mathsf w_T$ along Newton's ODE~$\dot x_t=v_t, \, \,\dot v_t = F_t$ is given by
$$\lim_{h \to 0} \, \frac{\mathsf w_T^2\bigl((x_t,v_t),(x_{t+h},v_{t+h})\bigr)}{h^2} =  12 \left| \frac{v_t}{T} - \frac{F_t}{2} \right|^2+|F_t|^2 \fstop$$
This quantity is not \emph{natural} in our setting, which demands the squared force $|F_t|^2$ instead, as we build a purely acceleration-based theory. 
\end{enumerate}
In a recent paper \cite{iacobelli2022new}, M.~Iacobelli explored two other families of `kinetic Wasserstein distances', yielding to new results for Vlasov's PDEs. The first idea is to build perturbations of the standard Wasserstein distance~$\mathrm{W}_2$ on~$\cP_2(\Gamma)$, using transportation costs of the form~$a\,\abs{y-x}^2 + b \, (y-x) \cdot (w-v) + c\, \abs{w-v}^2$, with~$a,b,c \in \R$ to be tuned. Then, time-dependent and non-linear generalisations are considered. Twisting the reference distance (usually~$\mathrm{H}^{2}$ or~$\Lp{2}$) to better capture the interaction between space and velocity variables has been a fruitful technique in kinetics, to prove both regularity (\emph{hypoellipticity}~\cite{kolmogoroff1934zufallige,hormander1967hypoelliptic}) and long-time convergence to equilibrium (\emph{hypocoercivity}~\cite{villani2009hypocoercivity,dolbeault2015hypocoercivity}). The interplay between hypocoercivty and optimal transport has been analysed in a few papers \cite{baudoin2013bakry,salem2021optimal}.
A second class of distances is constructed by adding a time-shift as follows \cite{iacobelli2022new}:
\[
\inf_{\pi \in \Pi(\mu,\nu)} \int \left( \abs{(x-tv) - (y-tw)}^2 + \abs{w-v}^2 \right) \, \dif \pi \comma \qquad t>0 \comma \quad  \mu,\nu \in \mathcal{P}_2(\Gamma) \fstop
\]

Our discrepancy~$\mathsf d$ differs from previous constructions, as it involves an optimisation over~$T>0$. As a result,~$\mathsf d$ is not a distance, but it has the physical dimension of a speed, so that its time derivative along curves of measures is naturally an acceleration. 

\subsubsection{Variational approximation schemes for kinetic equations}\label{sss124}
Minimising the squared acceleration in optimal transport originated from a series of papers about variational approximation schemes for dissipative kinetic PDEs~\cite{huang2000variational, huang2000variationall,duong2013generic,duong2014conservative,park2024variational}, before being readapted to fluid dynamics~\cite{gangbo2009optimal,cavalletti2019variational}. The goal there is to approximate, by means of De Giorgi minimising movement schemes~\cite{MR2401600}, the solution to Kramer's equation
\begin{equation}
    \label{vfp}
    \partial_t f + v \cdot \nabla_x f = \Delta_v f \comma \qquad f \colon (a,b) \times \Gamma \to \R \comma %
\end{equation}
and various generalisations thereof.
One prototypical result is the following.
\begin{thm}[\cite{duong2014conservative,huang2000variationall}\textcolor{white}{.}]
	Let~$\mathcal{E} \colon \mathcal{P}_2(\Gamma) \to [0,\infty]$ be the Boltzmann--Gibbs entropy 
	\begin{equation}
	\E(\mu) \coloneqq \begin{cases}
		\int_\Gamma f \log f \, \dif x \, \dif v &\text{if } \mu = f(x,v) \, \dif x \, \dif v \comma \\
		+\infty \qquad &\text{otherwise}.
	\end{cases} 
	\end{equation}
	Given an initial datum~$f_0 \in \Lp{1}(\Gamma)$, define the sequence 
	\begin{align}\label{eq:dg.}
		\begin{split}
			\mu^h_0 &\coloneqq f_0 \, \dif x\, \dif v \comma \\
			\mu^h_{(k+1)h} &\in \argmin_{\nu \in \mathcal{P}_2(\Gamma)} \, \left( \E(\nu) + \frac{1}{2h} \, \tilde \sfd_h^2\bigl(\mu^h_{kh},\nu\bigr)\right) \comma \qquad k \in \N \fstop
		\end{split}
	\end{align}
	Then, as~$h \to 0$, the piece-wise constant interpolation~$(\mu^h_t)_{t \ge 0}$ converges to the solution to~\eqref{vfp} with initial datum~$f_0$.
\end{thm}
As observed in~\cite{duong2014conservative}, the choice of the entropy~$\E$ and the penalisation~$\tilde \sfd_h$ can be motivated by a \emph{large deviation principle}.
However, the penalisation~$\tilde \sfd_h$ depends on the timestep~$h$, so that \eqref{eq:dg.} does not exactly define a De Giorgi scheme~\cite{MR2401600}.
Instead, the discrepancy~$\sfd$ we introduce does not depend on the timestep (in fact, we optimise over the time parameter). This way, the intrinsic time parameter of $\mathsf d$ will depend on the evolution itself, without being conditioned by the time discretisation. The analysis of a scheme akin to~\eqref{eq:dg.}, with~$\sfd$ in place of~$\tilde \sfd_h$, will be the subject of a forthcoming work.

\subsubsection{Wasserstein splines and interpolation}\label{sss123}

\paragraph{Splines between probability measures}

Splines arise in numerics as an interpolation method for a set of data~$(t_i,x_i)_{i=1,\dots,k} \subseteq [0,1] \times \cX$. An interpolating curve is a function~$\alpha \colon [0,1] \to \cX$ with~$\alpha(t_i) = x_i$, subject to certain constraints (e.g.,~regularity). One interesting and common example is the solution to
\[
\bar\alpha \in \argmin_{\alpha \in \mathrm{H}^{2}(0,1;\cX)} \set{ \int_0^1 \abs{\alpha''(t)}^2 \, \dif t \comma \text{ s.t.~} \alpha(t_i) = x_i \text{ for every } i } \fstop
\]
Spline interpolation between probability measures is the problem of finding a \emph{minimal acceleration curve}~$(\rho_t)_{t \in [0,1]}$ such that~$\rho_{t_i} = \rho_i$, for a given dataset~$(t_i,\rho_i)_{i=1,\dots,k} \subseteq [0,1] \times \mathcal{P}_2(\cX)$. This matter has recently attracted increasing interest and has been the subject of both theoretical and numerical investigations~\cite{benamou2019second,chen2018measure,clancy2021interpolating,chewi2021fast}. One possible approach is to interpret~$(\rho_t)_t$ as a curve taking values in the Riemannian-like space~$\bigl(\mathcal{P}_2(\cX),\mathrm{W}_2\bigr)$ (see~\S\ref{sss121}) and to measure the \emph{covariant acceleration} of~$\rho$ via the Levi--Civita connection~\cite{lott2008some,MR2920736}. A more tractable strategy is to consider measure-valued path splines: %
\begin{equation}\label{eq:dtp2.}
\bar{\mathbf m} \in \argmin_{\mathbf{m} \in \mathcal{P}\bigl(\mathrm{H}^{2}(0,1;\cX)\bigr)} \set{ \int_0^1 \int |\alpha''(t)|^2  \, \dif \mathbf{m}\, \dif t \comma \text{ s.t.~} (\proj_{t_i})_\# \mathbf m = \rho_{i} \text{ for every } i} \fstop
\end{equation}
Notice that this problem differs from~\eqref{eq:dtp1.} with~$T=1$, where we interpolate only between two measures and, most notably, we also fix the velocity marginals. However,~\eqref{eq:dtp2.} writes as a relaxation of \eqref{eq:dtp1.}, by further minimising over all possible velocity marginals~\cite{chen2018measure}.

\paragraph*{Kinetic Optimal Transport for measure interpolation}

We address two issues from the recent literature related to spline interpolation:
\begin{enumerate}
    \item Theorem \ref{thm2} provides a fully rigorous proof of the kinetic Benamou--Brenier formula with fixed space-velocity marginals. As a corollary, we prove a version thereof where only the space marginals are assigned, as conjectured by Y.~Chen, G.~Conforti, and T.~T.~Georgiou~\cite[Claim 4.1]{chen2018measure}. 
    Details are given in Remark \ref{rem:rem}.
    \item We outline a variation on the algorithm by S.~Chewi, J.~Clancy, T.~Le Gouic, P.~Rigollet, G.~Stepaniants, and A.~Stromme~\cite[Section 3]{chewi2021fast} for the construction of splines in~$\cP_2(\cX)$. As above, the problem is to construct a curve interpolating a given dataset~$(t_i,\rho_{i})_{i=1,\dots,k} \subseteq [0,1] \times \cP_2(\cX)$, with~$t_1 < t_2 < \dots < t_k$. The procedure in \cite{chewi2021fast} is briefly described as follows. Firstly, one computes the optimal multimarginal $2$-Wasserstein coupling~$\pi \in \cP_2(\cX^k)$. Secondly, one connects each tuple~$(x_1,\dots,x_k) \in \cX^k$ by means of an acceleration-minimising spline~$\bar \alpha_{x_1,\dots,x_k}$, and defines the interpolating spline of measures~$t \mapsto \rho_t$ at time~$t$ as the push-forward of~$\pi$ through the map~$\bar \alpha_\cdot (t) \colon (x_1,\dots,x_k) \mapsto \bar \alpha_{x_1,\dots,x_k}(t)$. The modification we propose is to use the construction above only to assign velocity marginals to each~$\rho_{i}$, i.e.,~we set~$\mu_{i} \coloneqq (\bar \alpha_{\cdot}(t), \bar \alpha_{\cdot}'(t))_\# \pi \in \cP_2(\Gamma)$. Given~$t \in (t_i, t_{i+1})$, we take the optimal~$\tilde{ \mathbf n}_{(t_{i+1}-t_{i})}$-optimal dynamical plan~$\mathbf m$ between~$\mu_{i}$ and~$\mu_{{i+1}}$, set
    \begin{equation}
        \mu_t \coloneqq \left(\proj_{\alpha(t),\alpha'(t)}\right)_\# \mathbf m \comma
    \end{equation}
    and define~$\rho_t$ as the space marginal of~$\mu_t$. Remarkably, our interpolation is deterministic and injective, in the sense that
    \[\left(\proj_{\bigl(\alpha({t_i}),\alpha'(t_{i})\bigr),\bigl(\alpha(t),\alpha'(t)\bigr)}\right)_\# \mathbf m
    \]
    is induced by an \emph{injective map}~$\Gamma \to \Gamma$ if~$\mu_{i}$ is absolutely continuous, see \S\ref{ss:42}.
    
\end{enumerate}

\subsubsection{Applications to biology and engineering}
The problem of finding a \emph{minimal acceleration path} between two measures~$\mu,\nu$ appears naturally in applications. 
\begin{itemize}
    \item \emph{Trajectory inference} is aimed at reconstructing a time-continuous evolution from a few time-separated observations. This technique has recently gained relevance in mathematical biology to study the development of cells~\cite{lavenant2021towards,schiebinger2021reconstructing,bunne2024optimal,chizat2022trajectory}, with potential applications in regenerative medicine. Wasserstein splines (see~\S\ref{sss123}) and our Proposition \ref{injtinterp} provide a smooth interpolation scheme to this purpose~\cite{rohbeck2024modeling,chen2021controlling}. In addition, the action functional in \eqref{eq:dtp1.}, which encodes minimal acceleration/consumption along paths, can be easily adapted to the specific model under investigation (taking into account, e.g.,~potential energy, drift).
    \item Images in computer vision can be cast into probability measures. Various applications involve continuous interpolations between images, which are often performed using classical OT~\cite{rabin2012wasserstein,santambrogio2015optimal}. More recently, an alternative has been formulated using Wasserstein splines~\cite{benamou2019second,clancy2021interpolating,friesecke2023gencol,justiniano2023consistent,zhang2022wassersplines}, with applications to texture generation models \cite{justiniano2024approximation}. Our construction, see \S\ref{sec:sec4}, proposes a twofold variation. Firstly, we remove the  dependence on the timespan~$T$, which is usually a datum of the problem in the literature. Secondly, we also assign the velocity marginals.
    \item The minimal acceleration problem~\eqref{eq:dynd}-\eqref{eq:naiveacc2} arises naturally also in optimal control. Indeed, we look for the optimal time-dependent force~$F_t = \alpha''(t)$ and timespan~$T$ required to connect two states~$(x,v)$ and~$(y,w)$ in~$\Gamma$. The quantity we minimise is the squared norm of~$F_t$ over time, which is reminiscent of resource consumption in steering of robots and space vehicles \cite{mwitta2024integration,longuski2014optimal}. In particular,~\eqref{eq:dynd}-\eqref{eq:naiveacc2} is used for rockets powered by Variable Specific Impulse engines~\cite{marec2012optimal,KECHICHIAN1995357}. Our work yields a natural generalisation, i.e.,~a mathematical framework for the optimal steering of a fleet of agents between two configurations, described by~$\mu, \nu \in \mathcal{P}_2(\Gamma)$.
\end{itemize}
\subsubsection{Conclusions and perspectives} \label{sec:perspectives}
\begin{itemize}
    \item In the current work, we build an optimal-transportation discrepancy~$\mathsf d$ between probability measures, which is based on the minimal acceleration. Also the timespan of the minimal-acceleration path is optimised. 
    In addition, we give a characterisation of~$\mathsf d$-absolutely continuous curves, see Theorem \ref{thm:main:new:2}. Such a result allows us to recast kinetic equations of Vlasov type, driven by the transport operator~$v \cdot \nabla_x$  and various collision terms, as paths in the space of probability measures, with the metric derivative depending only on the collisional effects. This is the starting point of the forthcoming~$\emph{Part II - Kinetic Gradient Flows}$. There, we are going to recast dissipative kinetic PDEs as steepest descent curves, among all~$\mathsf d$-absolutely continuous curves, for some given free-energy functionals. The steepest descent will be characterised via optimality in an energy-dissipation functional inequality. As a further justification, we aim at finding a large-deviation principle behind our kinetic gradient flow structure \cite{duong2014conservative}.
    \item In \emph{Part II - Kinetic Gradient Flows}, we aim to define JKO discrete variational schemes~\cite{MR1617171} with the discrepancy~$\mathsf d$ and prove their convergence to dissipative kinetic PDEs. To this end, we will treat a kinetic equation as a whole, exploiting the interplay between space and velocity variables, which is also the \emph{leitmotiv} of the current work. This strategy is also at the core of the \emph{hypocoercivity} theory~\cite{villani2009hypocoercivity,dolbeault2015hypocoercivity}, where the reference norm is twisted precisely in order to capture the interaction of~$x$ and~$v$ via~$v \cdot \nabla_x$. This is in contrast with the \emph{splitting} numerical schemes \cite{park2024variational}, where the transport and the collision terms are treated separately at each iteration.
    \item In this work, we specialise to a model case: particles are subject to Newton's laws~$\dot x_t = v_t \comma \dot v_t=F_t$, without external confinement. However, this suggests a general scheme to build adapted versions of the discrepancy~$\mathsf d$ for systems of interacting particles. The collision/irreversible effects in the phenomenon are measured via an action functional to be minimised under a constraint, given by a suitable continuity equation, see~\eqref{eq:vlasov2.}. Such a continuity equation (Vlasov's equation \eqref{eq:vlasov.} in our setting) is determined by the conservative/reversible  dynamics of the system. The resulting discrepancy captures the interaction between reversible and irreversible effects, while clearly distinguishing their roles. The induced geometry on~$\mathcal{P}_2(\Gamma)$ formally reads as a degenerate Riemann-like structure (where the $\sfd$-derivative of curves corresponds to the instantaneous action), constrained on a symplectic form, which allows only the \emph{physically} admissible directions.    
    We will explore generalisations of our theory in forthcoming papers. In particular, we aim at giving an optimal-transport interpretation of systems belonging to the GENERIC (General Equation for Non-Equilibrium Reversible-Irreversible Coupling) framework \cite{grmela1997dynamics,ottinger1997dynamics}.  
\end{itemize}

\section{The particle model}\label{sec:sec2}

In this section, we describe the kinetic optimal transport model for Dirac measures.

\subsection{The fixed-time discrepancy}\label{ss:21}

Let~$T > 0$ be a time parameter, and fix two points~$(x,v)$ and~$(y,w)$ in the phase space~$\Gamma \coloneqq \cX \times \cV = \R^n \times \R^n$. Recall the minimisation problem
\begin{equation} \label{eq:dynd2}
	\inf_{\alpha \in \mathrm{H}^{2}(0,T;\cX)} \set{ T \int_0^T \abs{\alpha''(t)}^2 \, \dif t \quad \text{ s.t.~} (\alpha,\alpha')(0) = (x,v) \text{ and } (\alpha,\alpha')(T) = (y,w) } \fstop
\end{equation}
This problem is strictly convex and coercive, 
hence, it admits a unique minimiser~$\a^T_{x,v,y,w}$. 
This curve satisfies the Euler--Lagrange equation~$\a'''' \equiv 0$ (i.e.,~it is a degree-3 polynomial in~$t$) with the prescribed boundary conditions. Straightforward computations yield
\begin{equation} \label{eq:gammaT}
	\a^T_{x,v,y,w}(t) = \left( \frac{v+w}{T^2}-2\frac{y-x}{T^3} \right)t^3 + \left(3\frac{y-x}{T^2}-\frac{2v+w}{T}\right)t^2 + vt+x \comma \qquad t \in (0,T) \comma
\end{equation}
or, equivalently,
\[
	\a^T_{x,v,y,w}(\xi T) 
	= 
	x 
	+ 
	\xi^2 (3 - 2 \xi) (y-x)
	+
	T \xi (1-\xi) \bigl( (1- \xi) v - \xi w \bigr) \comma
	\qquad \xi \in (0,1) \fstop
\]
We thus find the identity~\eqref{eq:dynd}, i.e.,~the minimal value of~\eqref{eq:dynd2} coincides with
\begin{equation} \label{eq:dtildeT}
	\tilde d_T^2 \bigl((x,v),(y,w)\bigr) \coloneqq 12 \abs{\frac{y-x}{T} - \frac{v+w}{2}}^2 + \abs{w-v}^2
	\fstop
\end{equation}

\begin{rem}
	It was shown by Kolmogorov~\cite{kolmogoroff1934zufallige} that the function
	\[
		\Psi\bigl((x,v),(y,w),t\bigr) 
			\coloneqq 
		\left(\frac{3}{2\pi t^2}\right)^d 
		\exp\left(-\frac{\tilde d_t^2\bigl((x,v),(y,w)\bigr)}{4t}\right)
	\]
	is the fundamental solution to the Kramers equation
	\[
		\partial_t \Psi + v \cdot \nabla_x \Psi 
			= 
		\Delta_v \Psi \fstop
	\]
\end{rem}

\begin{rem}
Contrary to the minimal~$\Lp{2}$-norm of the velocity (see~\eqref{eq:xyT}), 
the function~$\tilde d_T$ is not a distance on~$\Gamma$. 
Indeed, 
it is not symmetric, 
does not vanish on the diagonal of~$\Gamma \times \Gamma$ 
(i.e.,~points~$(x,v) = (y,w)$), 
and does not satisfy the triangle inequality. 
The latter can be easily checked on
\[
	(x_1,v_1) \coloneqq (0, \bar v) 
			\comma \quad
	(x_2,v_2) \coloneqq (T\bar v, \bar v) 
			\comma \quad 
	(x_3,v_3) \coloneqq (2T\bar v, \bar v) 
			\comma
\]
which satisfy 
	$\tilde d_T\bigl((x_1,v_1),(x_2,v_2)\bigr) 
	= \tilde d_T\bigl((x_2,v_2),(x_3,v_3)\bigr)= 0$,
while $\tilde d_T\bigl((x_1,v_1),(x_3,v_3)\bigr) = \sqrt{12} \, \abs{\bar v} \neq 0$
for any~$\bar v \in \cV \setminus \{0\}$. 
However, as observed in~\cite{gangbo2009optimal}, 
we have the equivalence
\begin{equation}
	\tilde d_T^2 \bigl( (x,v), (y,w) \bigr) = 0 
		\quad \Longleftrightarrow \quad 
	y=x+Tv \text{ and } v=w \fstop
\end{equation}
In this case, following~\cite{gangbo2009optimal}, we say that~$(y,w)$ is the $T$-\emph{free transport} of~$(x,v)$ and write~$(y,w) = \cG_T(x,v)$. Note that~$(\mathcal G_T)_{T \ge 0}$ enjoys the semigroup property
\begin{equation}
	\mathcal G_{T_1} \circ \mathcal G_{T_2} = \mathcal G_{T_1+T_2} \comma \qquad T_1,T_2 \ge 0 \fstop
\end{equation}
\end{rem}

\begin{rem}
	The fact that~$\tilde d_T\bigl((x,v), (y,w)\bigr)$ is finite for every~$(x,v),(y,w) \in \Gamma$ can be seen as an elementary version of \emph{hypoellipticity} (see~\cite{hormander1967hypoelliptic}). In fact, solutions to Newton's equation~$\dot x_t = v_t \comma \dot v_t = F_t$ are generated by vector fields in the space~$\set{(v,F) \, : \, F \in \R^n}$. This vector space has dimension~$n$, but it generates a Lie algebra of full rank~$2n$:
	\[
		\bigl[ (v,0),(v,e_i) \bigr] = (e_i,0) \comma \qquad i \in \set{1,\dots,n} \fstop
	\]
\end{rem}
\begin{rem}
	Fix~$(x,v), (y,w) \in \Gamma$. Let~$(v_t,F_t)_{t \in (0,T)}$ be the solution to Newton's equations, connecting~$(x,v)$ to~$(y,w)$ with the minimal action (i.e.,~$F_t = \left(\alpha_{x,v,y,w}^T\right)''$). The first norm in~\eqref{eq:dtildeT} (i.e.,~$\abs{\frac{y-x}{T}-\frac{v+w}{2}}$) is the distance between the average velocity~$\fint_0^T v_t \, \dif t$ and the arithmetic mean of the velocities at the endpoints. The second norm (i.e.,~$\abs{w-v}$) can be written as~$\abs{\int_0^T F_t \, \dif t}$.
\end{rem}

\begin{rem}
	It is interesting to analyse the behaviour of the optimal curves 
		$\a_{x,v,y,w}^T$
	for large~$T$.
	When~$t,T \to \infty$ with~$t \ge a\, T$ for some~$a > 0$, the formula~\eqref{eq:gammaT} gives
	\begin{align*}
		\left(\a_{x,v,y,w}^T\right)'(t) 
			&= \frac{t^2}{T^2}  (3v+3w)
			- \frac{t}{T} (4v+2w) 
			+ v 
			+ \left( -\frac{t^2}{T^2} + \frac{t}{T} \right)    \frac{6(y-x)}{T} 
			\\
			&= \xi^2 (3v+3w)  - \xi \, (4v+2w) + \xi \, (1-\xi) \frac{6(y-x)}{T} \comma \qquad \xi \coloneqq t/T \fstop
	\end{align*}
	Therefore, two cases may occur.
	\begin{itemize}
		\item \textbf{Case 1: slow line.} If the vector-valued polynomial
		\[
			p(\x) \coloneqq \xi^2 (3v+3w)  - \x \, (4v+2w) + v
		\]
		is identically equal to~$0$, then~$v=w=0$, the curve~$t \mapsto \a_{x,0,y,0}^T(t)$ is a parametrisation of the segment connecting~$(x,0)$ to~$(y,0)$, and~$\left(\a_{x,0,y,0}\right)'(t) = O(T^{-1})$ uniformly in~$t \in (0,T)$ as~$T \to \infty$.
		\item \textbf{Case 2: long curve.} If~$\x \mapsto p(\x)$ is not identically equal to~$0$, then there must exist an interval~$[a,b] \subseteq [0,1]$ (with~$0\le a<b \le 1$) where it never vanishes. On such an interval, we have
		\begin{equation*}
			\abs{\left(\a_{x,v,y,w}^T\right)'(t)} \ge \min_{\x \in [a,b]} \abs{p(\x)} - O(T^{-1}) \comma \qquad
			\text{as } T \to \infty \comma \text{ uniformly in }t \in [a,b] \comma
		\end{equation*}
		which shows that \emph{the total length of~$\a_{x,v,y,w}^T$ is of order~$\Theta(T)$}. On the other hand, the curvature is bounded as
		\[
			\kappa(t) \le \frac{\abs{\left(\a_{x,v,y,w}^T\right)''(t)}}{\abs{\left(\a_{x,v,y,w}^T\right)'(t)}^2} = \frac{T^{-1}\abs{p'(t/T)} + O(T^{-2})}{\abs{p(t/T)}^2+O(T^{-1})} = O(T^{-1}) \qquad \text{as } T \to \infty \comma
		\]
		uniformly in~$t \in [a,b]$. Moreover, in dimension~$d \ge 2$, we can often choose~$[a,b] = [0,1]$. This is because the set
		\[
			\set{(v,w) \in \cV \times \cV \, : \, \exists \x \in [0,1] \text{ with }p(\x) = 0}
		\]
		has dimension~$d+1$ (hence, it is negligible). All these observations indicate that, typically, this second case corresponds to~$\a_{x,v,y,w}^T$ resembling a large loop in the limit~$T \to \infty$.
	\end{itemize}

	\begin{figure}
		\centering
		\includegraphics[width=.45\textwidth, trim={170mm 70mm 160mm 80mm},clip]{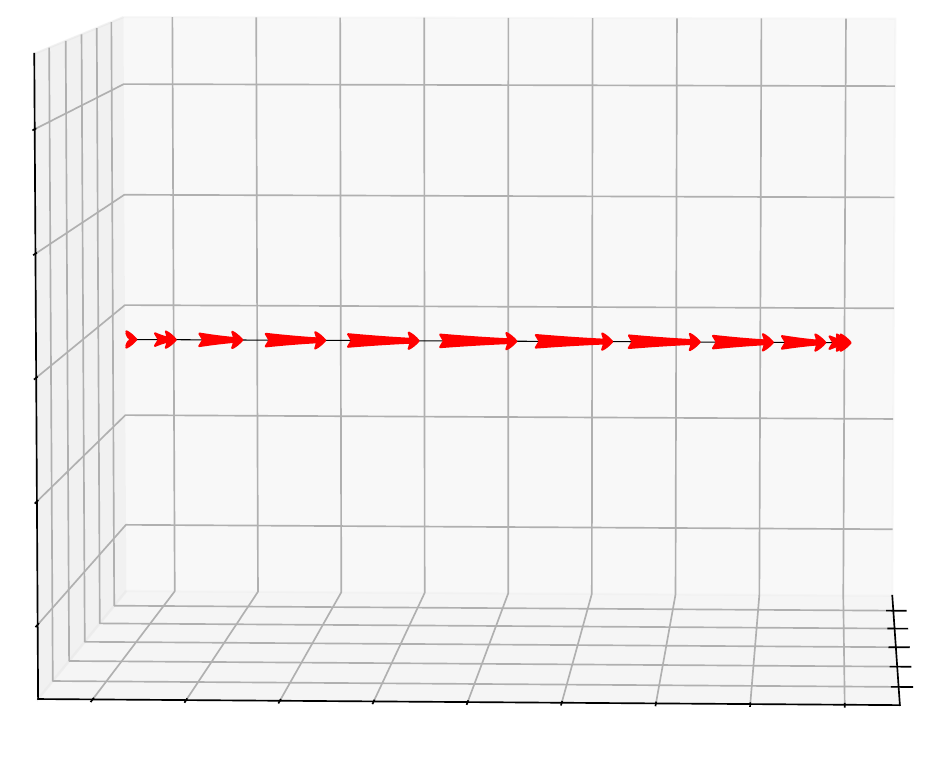}
			\hspace{60pt}
			\includegraphics[width=.39\textwidth, trim={48mm 20mm 40mm 25mm},clip]{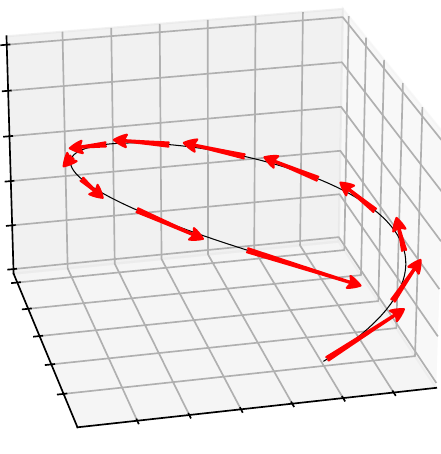}
		\caption{Examples of trajectories. In the left figure,~$v=w=0$.}
	\end{figure}
\end{rem}

We conclude this section with a version of the Monge--Mather's shortening principle, cf.~\cite[Chapter~8]{MR2459454}. Namely, we show that, given the initial and final configurations of two indistinguishable particles\footnote{We also optimise---with respect to the average squared acceleration---how particles in the initial configuration are coupled with those in the final configuration. Namely, a coupling is a matching of each particle in the initial configuration to one in the final configuration. Then, for every pair, one can compute the minimal acceleration as in \eqref{eq:dynd2}, and average such contribution over all pairs.} in different locations, their optimal trajectories for the minimal acceleration problem cannot meet \emph{at the same time, at the same point, with the same velocity}.

\begin{figure}
	\centering
	\includegraphics[width=.35\textwidth, trim={18mm 180mm 93mm 3mm},clip]{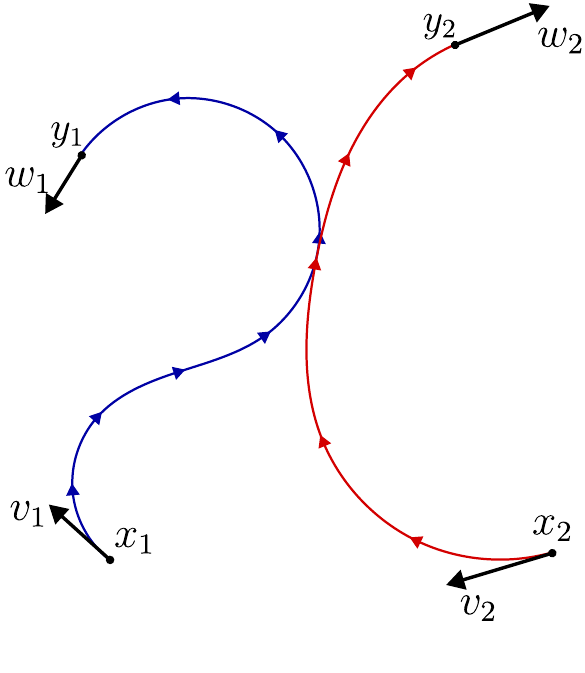}
	\caption{Trajectories that meet at the same time with the same velocity are not optimal.}
\end{figure}

\begin{prop} \label{prop:inj}
	Fix~$T > 0$ and let~$(x_1,v_1),(y_1,w_1),(x_2,v_2),(y_2,w_2) \in \Gamma$. Let~$\a_1,\a_2$ be the optimal (polynomial) curves for the problem \eqref{eq:dynd} between~$(x_1,v_1)$ and~$(y_1,w_1)$, and between~$(x_2,v_2)$ and~$(y_2,w_2)$, respectively, i.e.,
	\begin{equation}
		\tilde d_T^2\bigl((x_1,v_1),(y_1,w_1)\bigr) 
			= T \int_0^T \abs{\a_1''(t)}^2 \, \dif t
		 \quad \text{and} \quad 
		\tilde d_T^2\bigl((x_2,v_2),(y_2,w_2)\bigr) 
			= T \int_0^T \abs{\a_2''(t)}^2 \, \dif t \fstop
	\end{equation}
	If there exists~$\bar t \in (0,T)$ such that~$(\a_1(\bar t),\a_1'(\bar t)) = (\a_2(\bar t),\a_2'(\bar t))$, and if
	\begin{multline} \label{eq:mon}
			\tilde d_T^2\bigl((x_1,v_1),(y_1,w_1)\bigr) 
		+	\tilde d_T^2\bigl((x_2,v_2),(y_2,w_2)\bigr) 
		\\ \le 
			\tilde d_T^2\bigl((x_1,v_1),(y_2,w_2)\bigr) 
		+ 	\tilde d_T^2\bigl((x_2,v_2),(y_1,w_1)\bigr) \comma
	\end{multline}
	then,~$(x_1,v_1) = (x_2,v_2)$ and~$(y_1,w_1) = (y_2,w_2)$.
\end{prop}

\begin{proof}
	Define the curves
	\begin{equation*}
		\tilde \a_1(t) := \begin{cases}
			\a_1(t) &\text{if } t \in [0,\bar t]\comma \\
			\a_2(t) &\text{if } t \in [\bar t,T]\comma
		\end{cases}
		\qquad
		\tilde \a_2(t) := \begin{cases}
			\a_2(t) &\text{if } t \in [0,\bar t]\comma \\
			\a_1(t) &\text{if } t \in [\bar t,T]\comma
		\end{cases}
	\end{equation*}
	which, by our assumptions, are of class~$ \mathrm{H}^2$. They are competitors for the problem \eqref{eq:dynd2} between~$X_1 \coloneqq (x_1,v_1)$ and~$Y_2 \coloneqq (y_2,w_2)$, and between~$X_2 \coloneqq (x_2,v_2)$ and~$Y_1 \coloneqq (y_1,w_1)$, respectively. We also notice that, by additivity of the integral in the domain of integration,
	\begin{equation} \label{eq:integrEq}
		\int_0^T \abs{\a_1''(t)}^2 \, \dif t + \int_0^T \abs{\a_2''(t)}^2 \, \dif t = \int_0^T \abs{\tilde \a_1''(t)}^2 \, \dif t + \int_0^T \abs{\tilde \a_2''(t)}^2 \, \dif t.
	\end{equation}
	Exploiting the assumption \eqref{eq:mon}, we obtain
	\begin{align*}
		\tilde d_T^2(X_1,Y_1) + \tilde d_T^2(X_2,Y_2) &\stackrel{\eqref{eq:mon}}{\le} \tilde d_T^2(X_1,Y_2) + \tilde d_T^2(X_2,Y_1) \\
		&\le T \int_0^T \abs{\tilde \a_1''(t)}^2 \, \dif t + T \int_0^T \abs{\tilde \a_2''(t)}^2 \, \dif t \\
		&\stackrel{\eqref{eq:integrEq}}{=} T \int_0^T \abs{\a_1''(t)}^2 \, \dif t + T \int_0^T \abs{\a_2''(t)}^2 \, \dif t \\
		&= \tilde d_T^2(X_1,Y_1) + \tilde d_T^2(X_2,Y_2) \fstop
	\end{align*}
	We infer that the two inequalities in the latter formula are in fact equalities and, therefore,~$\tilde \a_1$ and~$\tilde \a_2$ are \emph{optimal} for the problem \eqref{eq:dynd2}. Consequently,~$\tilde \a_1$ and~$\tilde \a_2$ are polynomials. Since~$\a_1(t) = \tilde \a_1(t)$ for~$t \in [0,\bar t]$ and~$\bar t > 0$, the two polynomials~$\a_1$ and~$\tilde \a_1$ coincide. Similarly,~$\tilde \a_1 \equiv \a_2$; therefore,~$\a_1 \equiv \a_2$. We conclude that
	\begin{equation*}
		(x_1,v_1) = \a_1(0) = \a_2(0) = (x_2,v_2) \quad \text{and} \quad (y_1,w_1) = \a_1(T) = \a_2(T) = (y_2,w_2) \fstop \qedhere
	\end{equation*}
\end{proof}

\subsection{The non-parametric discrepancy}\label{ss:22}
Using~$\tilde d_T$, we shall now define a discrepancy~$d$ which is not parametric in time.

\begin{defn} \label{defn:defdist}
	For all~$(x,v), (y,w) \in \Gamma$, set
	\begin{equation}
		\tilde \dist\bigl( (x,v), (y,w) \bigr) \coloneqq \inf_{T > 0} \tilde d_T \bigl( (x,v), (y,w) \bigr) \fstop
	\end{equation}
	We denote by~$\dist \colon \Gamma \times \Gamma \to \R_{\ge 0}$ the lower-semicontinuous envelope of~$\tilde \dist$. We give~$d$ the name \emph{second-order discrepancy} between particles.
\end{defn}

\begin{prop} \label{prop:dist}
	The following hold.
	\begin{enumerate}
		\item \label{st1} The function~$\tilde \dist \colon \Gamma \times \Gamma \to \R_{\ge 0}$ is upper-semicontinuous. For every~$(x,v),(y,w) \in \Gamma$, we have
		\begin{equation}
			\tilde \dist\bigl((x,v), (y,w) \bigr) 
			= \begin{cases}
				\lim_{T \to \infty} \tilde d_T\bigl((x,v),(y,w)\bigr) &\text{if } (y-x) \cdot (v+w) \le 0 \comma \\
				\tilde d_{T^*} \bigl((x,v),(y,w)\bigr) &\text{if } (y-x) \cdot (v+w) > 0 \comma
			\end{cases}
		\end{equation}
		where
		\begin{equation}
			\label{eq:Tstar}
			T^* \coloneqq 2\frac{\abs{y-x}^2}{(y-x)\cdot (v+w)} \fstop
		\end{equation}
		Hence, the following formula
		\begin{multline} \label{eq:st1}
			\tilde \dist^2\bigl((x,v), (y,w) \bigr) = \begin{cases}
				3 \abs{v+w}^2 - 3
				\left(\frac{y-x}{\abs{y-x}}\cdot (v+w) \right)_+^2 + \abs{w-v}^2 &\text{if } x \neq y \comma \\
				3 \abs{v+w}^2 + \abs{w-v}^2  &\text{if } x=y
			\end{cases}\\
		(x,v), (y,w) \in \Gamma \fstop
		\end{multline}
		\item \label{st2} The second-order discrepancy~$\dist \colon \Gamma \times \Gamma \to \R_{\ge 0}$ is given by the formula
		\begin{multline} \label{eq:st2}
			\dist^2\bigl((x,v), (y,w) \bigr) = \begin{cases}
				3 \abs{v+w}^2 - 3\left(\frac{y-x}{\abs{y-x}} \cdot (v+w) \right)_+^2 + \abs{w-v}^2 &\text{if } x \neq y \comma \\
				\abs{w-v}^2  &\text{if } x=y \comma
			\end{cases}\\
			(x,v), (y,w) \in \Gamma \fstop
		\end{multline}
	\item \label{st3} We have~$\dist\bigl((x,v),(y,w)\bigr) = 0$ if and only if either~$(y,w) = \cG_T(x,v)$ for some~$T \ge 0$, or~$x \neq y$ and~$v=w=0$.
	\end{enumerate}
\end{prop}

\begin{proof}
	\textbf{Proof of~\ref{st1}.} Upper-semicontinuity is trivial, because~$\tilde \dist$ is defined as an infimum of continuous functions.
	
	Fix~$(x,v),(y,w) \in \Gamma$. If~$y=x$, then~$T \mapsto \tilde d_T^2\bigl((x,v),(y,w)\bigr)$ is constant, hence always equal to its limit as~$T \to \infty$. Otherwise, let us rewrite~\eqref{eq:dtildeT} as a convex quadratic polynomial in~$T^{-1}$, namely:
	\[
		\tilde d_T^2\bigl((x,v),(y,w)\bigr) = 12 T^{-2} \abs{y-x}^2 - 24 T^{-1} (y-x) \cdot (v+w) + 3 \abs{v+w}^2 + \abs{w-v}^2 \fstop
	\]
	The vertex of this parabola is found at
	\[
		T^{-1} = \frac{(y-x) \cdot (v+w)}{2 \abs{y-x}^2} \fstop
	\]
	Therefore, when~$(y-x) \cdot (v+w) \le 0$, the minimum of~$\tilde d_T^2$, constrained to~$T > 0$, is approached as~$T^{-1} \to 0$. In formulae:
	\[
		\tilde d^2\bigl((x,v),(y,w) \bigr) = \lim_{T \to \infty} \tilde d_T^2\bigl((x,v),(y,w)\bigr) = 3 \abs{v+w}^2 + \abs{w-v} \fstop
	\]
	Instead, when~$(y-x) \cdot (v+w) > 0$, we have
	\[
		\tilde \dist^2\bigl((x,v), (y,w)\bigr) = \tilde d_{T^*}^2 \bigl((x,v),(y,w)\bigr) = 3 \abs{v+w}^2-3 \biggl(\frac{y-x}{\abs{y-x}} \cdot (v+w)\biggr)^2 + \abs{w-v}^2 \fstop
	\]
	
	\textbf{Proof of~\ref{st2}.} The right-hand side in~\eqref{eq:st2} is lower-semicontinuous. Since it coincides with~$\tilde \dist\bigl((x,v), (y,w)\bigr)$ when~$x \neq y$ or~$v+w = 0$, we are only left with showing that, for every~$x,v,w$ with~$v+w \neq 0$, there exists a sequence~$y_k \to x$ such that
	\[
		\abs{w-v}^2 \ge \limsup_{k \to \infty} \tilde \dist^2\bigl((x,v), (y_k,w)\bigr) \fstop
	\]
	We simply choose
	\[
		y_k \coloneqq x + \frac{1}{k}(v+w) \comma \qquad k \in \N_{1} \fstop
	\]
	
	\textbf{Proof of~\ref{st3}.} Assume that the right-hand side of~\eqref{eq:st2} equals~$0$. We infer that~$v = w$. If~$x=y$, then~$(y,w) = \cG_0(x,v)$. If~$x \neq y$, then
	\[
		2\abs{v} = 2 \frac{y-x}{\abs{y-x}} \cdot v
	\]
	and, therefore, either~$v = 0$, or~$y = x+T v$ for some~$T > 0$, that is,~$(y,w) = \cG_T(x,v)$. The converse implication is a direct computation.
\end{proof}

\begin{rem} \label{rem:counterexamplesd}
	It follows from~$\eqref{eq:st1}$ and~$\eqref{eq:st2}$ that neither~$\tilde \dist$ nor~$\dist$ is symmetric. Neither of the two satisfies the triangle inequality: consider
	\[
		(x_1,v_2) \coloneqq (0,\bar v) \comma \quad (x_2,v_2) \coloneqq (\bar v,0) \comma \quad (x_3,v_3) \coloneqq (-\bar v, 0)
	\]
	for any~$\bar v \in \cV \setminus \set{0}$. Moreover, we have the characterisation
	\begin{equation}
		d^2\bigl((x,v),(y,w)\bigr) = 0 \quad \iff \quad \left[ \,  v=w=0 \text{ or } (y,w)=\mathcal G_T(x,v) \text{ for some } T \ge 0 \, \right] \fstop
	\end{equation}
\end{rem}

In analogy with the metric setting of~\cite[Theorem~1.1.2]{MR2401600}), we give the following.

\begin{defn} \label{defn:ddiff}
	We say that a curve~$\g = (x_\cdot,v_\cdot) \colon (a,b) \to \Gamma$ is~\emph{$\dist$-differentiable} at~$t \in (a,b)$ when the one-sided limit
	\begin{equation} \label{eq:incrratio}
	\lim_{h \downarrow 0} \frac{\dist\bigl(\g(t), \g(t+h)\bigr)}{h} \fstop
	\end{equation}
	exists. In this case, we denote it by~$\abs{\g'}_\dist(t)$ and call it the \emph{$\dist$-derivative of~$\g$} at~$t$.
\end{defn}

\begin{rem}
	As the discrepancy~$d$ is not symmetric, taking left or right limits to define~$d$-differentiability is not the same, even for smooth curves. We argue that~\eqref{eq:incrratio} is the natural definition. Indeed, if, for example, we consider the straight constant-speed line~$\gamma(t) \coloneqq (t\bar v,\bar v)$,~$t \in \R$, for some~$\bar v \in \cV \setminus \set{0}$, we have
    \[
        d\bigl(\gamma(t),\gamma(t+h) \bigr) = 0 \quad \text{for every } h > 0 \comma
    \]
    and
    \[
        d\bigl(\gamma(t),\gamma(t-h) \bigr) = \abs{2v} > 0 \quad \text{for every } h > 0 \fstop
    \]
    In particular,~$\gamma$ is~$d$-differentiable in the sense of \Cref{defn:ddiff}, but
    \[
        \lim_{h \downarrow 0} \frac{\dist\bigl(\g(t), \g(t-h)\bigr)}{\abs{h}} = \infty \fstop
    \]
\end{rem}

Our next aim is to formulate necessary and sufficient conditions for $d$-differentiability.

\begin{prop} \label{prop:ddiff1}
    Let~$\g = (x_\cdot, v_\cdot) \colon (a,b) \to \Gamma$ be a curve such that~$x_\cdot$ is of class~$\mathrm{C}^1$ and~$v_\cdot$ is of class~$\mathrm{C}^0$. If~$\gamma$ is~$d$-differentiable at~$t \in (a,b)$ and~$v_t \neq 0$, then there exists~$\lambda(t) \ge 0$ such that~$\dot x_t = \lambda(t) v_t$.
\end{prop}
\begin{proof}
    If~$\gamma$ is~$\dist$-differentiable at~$t \in (a,b)$, then~$\dist\bigl(\g(t),\g(t+h)\bigr) \leq C h$ for suitable $C > 0$,
	whenever $h > 0$ is small enough.
    If~$\dot x_t = 0$, then it suffices to choose~$\lambda(t) = 0$. Otherwise,
	we have $x_{t+h} \neq x_t$ for small enough $h > 0$,
	hence $\dist^2\bigl(\g(t),\g(t+h)\bigr) 
		= \tilde \dist^2\bigl(\g(t),\g(t+h)\bigr)$.
	Pick $T(h) > 0$ such that 
	$
	\tilde \dist_{T(h)}^2\bigl(\g(t),\g(t+h)\bigr)
	\leq 
	\tilde \dist^2\bigl(\g(t),\g(t+h)\bigr)
	+ h^2
	$.
	In particular, for $h > 0$ small enough,
	\begin{align*}
		12\abs{\frac{x_{t+h}-x_t}{T(h)}-\frac{v_{t}+v_{t+h}}{2}}^2 
			\leq
		\tilde \dist_{T(h)}^2\bigl(\g(t),\g(t+h)\bigr)
			& \leq 
		\tilde \dist^2\bigl(\g(t),\g(t+h)\bigr) + h^2
			\\& =
		\dist^2\bigl(\g(t),\g(t+h)\bigr) + h^2
			\leq
		(C^2 + 1) h^2 \fstop
	\end{align*}
	Since 
		$\frac{x_{t+h}-x_t}{h} \to \dot x_t \neq 0$ 
	and
		$\frac{v_{t}+v_{t+h}}{2} \to  v_t \neq 0$,
	we infer that $T(h)/h \to \lambda(t)$ as $h \to 0$ for some $\lambda(t) \in (0,\infty)$ satisfying~$\dot x_t = \lambda(t) v_t$.
\end{proof}

\begin{rem}[Reparametrisation] \label{rem:repar}
	Let~$\gamma = (x_\cdot, v_\cdot) \colon (a,b) \to \Gamma$ be a curve such that~$x_\cdot$ is of class~$\mathrm{C}^{k+1}$ and~$v_\cdot$ is of class~$\mathrm{C}^k$ for some~$k \in \N_0$. Assume that~$\dot x_t=\lambda(t) v_t$ for every~$t \in (a,b)$, for a function~$\lambda \colon (a,b) \to (0,\infty)$ of class~$\mathrm{C}^{k}$. Let~$\tau \colon (\tilde a, \tilde b) \to (a,b)$ be a function of class~$\mathrm{C}^{k+1}$ with~$\tau' > 0$. Set
	\[
		\tilde x_s \coloneqq x_{\tau(s)} \comma \quad \tilde v_s \coloneqq v_{\tau(s)} \comma \quad \tilde \gamma(s) \coloneqq (\tilde x_s, \tilde v_s) \comma \qquad s \in (\tilde a, \tilde b) \fstop
	\]
	Then, we have
	\[
		\tilde v_s = v_{\tau(s)} = \frac{1}{\lambda\bigl(\tau(s)\bigr)} \dot x_{\tau(s)} = \frac{1}{\lambda\bigl(\tau(s)\bigr)\tau'(s)} \dot {\tilde x}_s \comma \qquad s \in (\tilde a, \tilde b) \comma
	\]
	that is, also~$\tilde \gamma$ has the property~$\dot {\tilde x}_s = \tilde \lambda(s) \tilde v_s$ for every~$s \in (\tilde a, \tilde b)$, for a function~$\tilde \lambda \colon (\tilde a, \tilde b) \to (0,\infty)$ of class~$\mathrm{C}^k$.
	
	Let~$\bar t \in (a,b)$ and assume that~$\lambda > 0$ on~$(a,b)$.  By solving the Cauchy problem
	\begin{equation} \label{eq:cauchy}
		\tau' = \frac{1}{\lambda(\tau)} \comma \quad \tau(0) = \bar t \comma
	\end{equation}
	we can find a reparametrisation~$\tilde \gamma$ with~$\dot{\tilde x}_s = \tilde v_s$. Indeed, by classical results, the ODE~\eqref{eq:cauchy} admits a maximal solution on a neighbourhood of~$\bar t$. Moreover, since~$\lambda$ is bounded on every compact~$K \Subset (a,b)$, the values of~$\tau$ exit~$K$ in finite time. Therefore, the maximal solution has the full set~$(a,b)$ as its image.
\end{rem}

\begin{prop} \label{prop:ddiff2}
    Let~$\g = (x_\cdot, v_\cdot) \colon (a,b) \to \Gamma$ be a curve 
	such that $x_\cdot$ is of class~$\mathrm{C}^2$ and 
	$v_\cdot$ is of class~$\mathrm{C}^1$. Assume that there exists~$\lambda \colon (a,b) \to (0,\infty)$ continuous, such that~$\dot x_t = \lambda(t) v_t$ for every~$t \in (a,b)$. Then,~$\gamma$ is~$d$-differentiable on~$(a,b)$ with~$\abs{\gamma'}_d(t)=\abs{\dot v_t}$.
\end{prop}

\begin{proof}
    On the one hand, we have
    \[
        \liminf_{h \downarrow 0} \frac{d^2\bigl( \gamma(t),\gamma(t+h) \bigr)}{h^2}
            \ge
        \liminf_{h \downarrow 0} \frac{\abs{v_t-v_{t+h}}^2}{h^2} = \abs{\dot v_t}^2 \comma \qquad t \in (a,b) \fstop
    \]
    To prove the opposite inequality, momentarily assume that~$\lambda \equiv 1$, i.e.,~$\dot x_t = v_t$ for every~$t \in (a,b)$. We obtain
    \begin{align} \label{eq:limsupddiff} 
        \limsup_{h \downarrow 0} \frac{d^2\bigl( \gamma(t),\gamma(t+h) \bigr)}{h^2} \nonumber
            &\le
        \limsup_{h \downarrow 0} \frac{\tilde d_h^2\bigl( \gamma(t),\gamma(t+h) \bigr)}{h^2} \nonumber \\
            &=
        \limsup_{h \downarrow 0} \frac{1}{h^2} \left(12\abs{\frac{x_{t+h}-x_t}{h} - \frac{v_{t+h}+v_t}{2}}^2 + \abs{v_{t+h}-v_t}^2 \right) \nonumber \\
            &=
        \abs{\dot v_t}^2 + \limsup_{h \downarrow 0} \frac{12}{h^2}\abs{\dot x_t + \frac{h}{2} \ddot x_t - v_t - \frac{h}{2} \dot v_t + o(h)}^2 = \abs{\dot v_t}^2
    \end{align}
    for every~$t \in (a,b)$. In the general case, we apply the reparametrisation of \Cref{rem:repar} to find a diffeomorphism~$\tau \colon (\tilde a, \tilde b) \to (a,b)$ such that~$\tau'(s) = \frac{1}{\lambda(\tau(s))}$ for every~$s \in (\tilde a, \tilde b)$, so that the computation~\eqref{eq:limsupddiff} can be performed on~$\tilde \gamma \colon s \mapsto \gamma\bigl(\tau(s)\bigr)$. Given~$t = \tau(s) \in (a,b)$, we thus find
    \begin{align*}
        \limsup_{h \downarrow 0} \frac{d^2\bigl( \gamma(t),\gamma(t+h) \bigr)}{h^2}
            &=
        \limsup_{h \downarrow 0} \frac{d^2\bigl( \gamma\bigl(\tau(s)\bigr),\gamma\bigl(\tau(s)+h\bigr) \bigr)}{h^2} \\
            &=
        \limsup_{\tilde h \downarrow 0} \frac{d^2\bigl( \gamma\bigl(\tau(s)\bigr),\gamma\bigl(\tau(s+\tilde h)\bigr) \bigr)}{\bigl(\tau(s+\tilde h)-\tau(s)\bigr)^2} \\
            &=
        \limsup_{\tilde h \downarrow 0} \frac{d^2\bigl( \tilde \gamma(s),\tilde \gamma(s+\tilde h) \bigr)}{\tilde h^2}\frac{\tilde h^2}{\bigl(\tau(s+\tilde h)-\tau(s)\bigr)^2}\\
            &\stackrel{\eqref{eq:limsupddiff}}{=}
        \lambda\bigl(\tau(s)\bigr)^2 \abs{\frac{\dif}{\dif s} v_{\tau(s)}}^2
            =
            \abs{\dot v_t}^2 \comma
    \end{align*}
    and this concludes the proof.
\end{proof}

\begin{rem}
    The last result, specialised to curves satisfying~$\dot x_t = v_t$ (i.e.,~with $\lambda \equiv 1$), yields~$\abs{\gamma'}_d(t) =\abs{\dot v_t}= \abs{\ddot x_t}$, for all $t\in(a,b)$. In view of Newton's second law of motion, we can say that the~$d$-derivative equals the magnitude of the force driving the motion. 
\end{rem}

\begin{rem} \label{rem:vequal0}
    A curve~$\gamma = (x_\cdot,v_\cdot) \to \Gamma$ is everywhere $d$-differentiable also when~$v_\cdot \equiv 0$, regardless of~$x_\cdot$. In this case,~$\abs{\gamma'}_d \equiv 0$.
\end{rem}

\begin{cor}
    Let~$\gamma = (x_\cdot,v_\cdot) \colon (a,b) \to \Gamma$ be a curve such that~$x_\cdot$ is of class~$\mathrm{C}^2$ and~$v_\cdot$ is of class~$\mathrm{C}^1$. Assume that~$\dot x_t \neq 0$ and~$v_t \neq 0$ for every~$t \in (a,b)$. Then,~$\gamma$ is everywhere~$d$-differentiable if and only if there exists~$\lambda \colon (a,b) \to (0,\infty)$ of class~$\mathrm{C^1}$ such that~$\dot x_t = \lambda(t) v_t$ for every~$t \in (a,b)$.
\end{cor}

\begin{proof}
    If~$\gamma$ is everywhere~$d$-differentiable, by \Cref{prop:ddiff1}, there exists~$\lambda \colon (a,b) \to \R_{\ge 0}$ such that~$\dot x_t = \lambda(t) v_t$. This function is of class~$\mathrm{C}^1$ because~$v_t \neq 0$ for every~$t \in (a,b)$ and because both~$\dot x_\cdot$ and~$v_\cdot$ are of class~$\mathrm{C}^1$. Moreover~$\lambda(t) \neq 0$ for every~$t$, because this property holds for~$\dot x_t$.
    The converse follows from \Cref{prop:ddiff2}. In this case,~$\abs{\gamma'}_d(t) = \abs{\dot v_t}$.
\end{proof}

\section{Kinetic optimal plans and maps}\label{sec:sec3}

This section is divided into three parts:
\begin{enumerate}
	\item In~\S\ref{ss:31}, we prove Statements~\ref{thm1.1}-\ref{thm1.3} in \Cref{thm1}, including the semicontinuity of~$\sfd$ and the existence of optimal transport \emph{plans}.
	\item In~\S\ref{ss:32}, we prove Statement~\ref{thm1.4} in \Cref{thm1}, i.e.,~the existence of optimal \emph{maps}.
	\item In~\S\ref{ss:33}, we discuss additional results, including \emph{non-uniqueness} of optimal plans and maps, and the characterisation of the pairs~$(\mu,\nu)$ for which~$\sfd(\mu,\nu) = 0$.
\end{enumerate}

\subsection{Semicontinuity and existence of optimal plans}\label{ss:31}

Let~$\cP_2(\Gamma \times \Gamma)$ be the set of probability measures on~$\Gamma \times \Gamma$ with finite second moment. For every~$\pi \in \cP_2(\Gamma \times \Gamma)$, set
\begin{align}
	\tilde \sfc_T(\pi)
		&\coloneqq
	12 \norm{\frac{y-x}{T} - \frac{v+w}{2}}_{\Lp{2}(\pi)}^2 + \norm{w-v}_{\Lp{2}(\pi)}^2 \comma \qquad T > 0 \\
	\label{eq:tildesfc}
	\tilde \sfc(\pi)
		&\coloneqq
	\begin{cases}  3  \norm{v+w}^2_{\Lp{2}(\pi)} - 3\frac{\bigl((y-x,v+w)_\pi\bigr)_+^2}{\norm{y-x}^2_{\Lp{2}(\pi)}}  + \norm{w-v}^2_{\Lp{2}(\pi)} &\text{if }  \norm{y-x}_{\Lp{2}(\pi)} > 0 \comma \\  3\norm{v+w}^2_{\Lp{2}(\pi)} + \norm{w-v}^2_{\Lp{2}(\pi)}, &\text{if } \norm{y-x}_{\Lp{2}(\pi)}  = 0 \comma
	\end{cases} \\
	\label{eq:sfc}
	\sfc(\pi)
		&\coloneqq 
	\begin{cases}  3  \norm{v+w}^2_{\Lp{2}(\pi)} - 3\frac{\bigl((y-x,v+w)_\pi\bigr)_+^2}{\norm{y-x}^2_{\Lp{2}(\pi)}}  + \norm{w-v}^2_{\Lp{2}(\pi)} &\text{if }  \norm{y-x}_{\Lp{2}(\pi)} > 0 \comma \\  \norm{w-v}^2_{\Lp{2}(\pi)}, &\text{if } \norm{y-x}_{\Lp{2}(\pi)}  = 0 \fstop
	\end{cases}
\end{align}
Note that~$\tilde \sfc(\pi) = \sfc(\pi)$ whenever~$\norm{y-x}_{\Lp{2}(\pi)}  > 0$.

\begin{rem}
	It may appear tempting
	to consider instead of $\tilde \sfc$ a different object, namely 
	$\hat \sfc(\pi) := \int_{\Gamma \times \Gamma} \tilde d^2\bigl((x,v),(y,w)\bigr) \, \dif \pi$.
	Let us start by noticing that $\hat \sfc(\pi)$ involves first a point-wise optimisation of $\tilde d_T^2((x,v),(y,w))$ in $T$, for each pair of states $(x,v)$ and $(y,w)$, and then an integration over all pairs $((x,v),(y,w)) \in \mathrm{supp}(\pi)$. By contrast, as the next result shows, $\tilde \sfc (\pi)$ is given by one \emph{synchronous} minimisation over $T>0$ for the cost $\tilde \sfc_T(\pi)$. We justify why $\tilde \sfc$ is more natural for our purposes.
    Let us consider $\mu \in \mathcal P_2(\Gamma)$, such that $(\proj_v)_\# \mu \neq \delta_0$. Let $\sigma : \Gamma \to [0,\infty)$ be a measurable map, and let  $G_\sigma :\Gamma \to \Gamma$ be defined by the formula $G_\sigma (x,v) = (x + \sigma(x,v) ,v)$. Then, by calling $\nu_\sigma = (G_\sigma)_\# \mu$, we have that 
    \[
    \inf_{\pi \in \Pi(\mu,\nu_\sigma)} \, \hat \sfc(\pi) = 0 \fstop
    \]
    By contrast, 
    \[
    \inf_{\pi \in \Pi(\mu,\nu_\sigma)} \, \tilde \sfc(\pi) =0
    \]
    if and only if $\sigma \equiv T$, for some $T \in [0,\infty)$. This shows that the optimal-transport problem associated with $\hat \sfc$ is much more degenerate than the one associated with $\tilde \sfc$. 
    Finally, by taking the curve $t \mapsto \mu_t := \nu_{t \sigma}$, we have that
    \[
    \forall \, 0<t<s \comma \qquad  \inf_{\pi \in \Pi(\mu_t,\mu_s)} \hat \sfc (\pi) = 0 \comma
    \]
    which means that the curve $(\mu_t)_t$ is everywhere differentiable in the optimal-transport discrepancy induced by $\hat \sfc$. On the other hand, it is easy to see that this curve does not solve any Vlasov's equation \eqref{eq:vlasov.} in general. Thus, we would not be able to recover the PDE representation of \Cref{thm:main:new:2} in case we used $\hat \sfc$ instead of $\tilde \sfc$.
\end{rem}

\begin{prop}[\Cref{thm1}, Statement~\ref{thm1.1}] \label{prop:thm1.1}
	For every~$\pi \in \cP_2(\Gamma \times \Gamma)$,
	we have
	\begin{equation}
		\tilde \sfc(\pi) = \inf_{T > 0} \tilde \sfc_T(\pi) \fstop
	\end{equation}
	The infimum is obtained for
	\begin{equation} \label{eq:defT}
		\begin{cases}
			T = 2 \displaystyle \frac{\norm{y-x}_{\Lp{2}(\pi)}^2}{(y-x,v+w)_\pi} &\text{if } (y-x,v+w)_\pi > 0 \comma \\
			\text{any } T>0 &\text{if } \norm{y-x}_{\Lp{2}(\pi)} = 0 \comma \\
			T \to \infty &\text{otherwise.}
		\end{cases}
	\end{equation}
	In particular,
	the function~$\pi \mapsto \tilde \sfc(\pi)$ is concave, and~\eqref{eq:dinfT} holds for every~$\mu,\nu \in \cP_2(\Gamma)$.
\end{prop}

\begin{proof}
	The proof is identical to that of \Cref{prop:dist},~\ref{st1}. The function~$\tilde \sfc$ is concave because it is an infimum of linear functions.
\end{proof}

\begin{lem} \label{lem:semicontsfc}
	Let~$\mu_k \to \mu$ and~$\nu_k \to \nu$ two converging sequences in the $2$-Wasserstein distance. Let~$\pi_k \in \Pi(\mu_k,\nu_k)$ for every~$k$, and assume that~$(\pi_k)_k$ narrowly converges to a measure~$\pi \in \mathcal P(\Gamma \times \Gamma)$. Then, convergence holds in~$\mathrm{W}_2$, we have~$\pi \in \Pi(\mu,\nu)$, and
	\begin{equation} \label{eq:semicontsfc}
		\sfc(\pi) \le \liminf_{k \to \infty} \sfc(\pi_k) \fstop
	\end{equation}
\end{lem}

\begin{proof}
	Convergence holds in $\mathrm W_2$ by \cite[Remark~7.1.11]{MR2401600}. The measure~$\pi$ lies in~$\Pi(\mu,\nu)$ by narrow continuity of the projection maps.
    We claim that the four functions
	\begin{align*}
		&F_1(x,v,y,w) \coloneqq \abs{v+w}^2 \comma \quad &&F_2(x,v,y,w) \coloneqq  \abs{(y-x) \cdot (v+w)} \comma\\
		&F_3(x,v,y,w) \coloneqq \abs{y-x}^2 \comma \qquad &&F_4(x,v,y,w) \coloneqq \abs{w-v}^2
	\end{align*}
	are uniformly integrable with respect to~$(\pi_k)_k$. %
	Indeed, as
	\[
	F_i(x,v,y,w) \le 4\max\set{\abs{x}^2+\abs{v}^2,\abs{y}^2+\abs{w}^2} \comma \qquad i \in \set{1,2,3,4} \comma
	\]
	for~$a > 0$, we find that
	\begin{align*}
		\int_{\set{F_i \ge a}} F_i \, \dif \pi_k
		&\le
		4 \int_{\set{\abs{x}^2 + \abs{v}^2 \ge \frac{a}{4}}} \left(\abs{x}^2 + \abs{v}^2 \right) \, \dif \mu_k +  4 \int_{\set{\abs{y}^2 + \abs{w}^2 \ge \frac{a}{4}}} \left(\abs{y}^2 + \abs{w}^2 \right) \, \dif \nu_k \fstop 
	\end{align*}
	Then, the claim follows from the uniform integrability of the second moments of~$(\mu_k)_k$ and~$(\nu_k)_k$, given by~\cite[Proposition~7.1.5]{MR2401600}.
	
	If~$\norm{y-x}_{\Lp{2}(\pi)} > 0$, then~$\norm{y-x}_{\Lp{2}(\pi_k)} > 0$ eventually; hence~$\sfc(\pi) = \tilde \sfc(\pi)$ and~$\sfc(\pi_k) = \tilde \sfc(\pi_k)$ for every~$k$ sufficiently large. 
	Since the functions $F_i$ are uniformly integrable, through~\cite[Lemma~5.1.7]{MR2401600}, we find that~$\tilde \sfc(\pi_k) \to \tilde \sfc(\pi)$. Therefore,
	\[
		\sfc(\pi) = \tilde \sfc(\pi) = \lim_{k \to \infty} \tilde \sfc(\pi_k) = \lim_{k \to \infty} \sfc(\pi_k) \fstop
	\]
	
	If, instead,~$\norm{y-x}_{\Lp{2}(\pi)} = 0$, then
	\[
		\sfc(\pi) = \norm{w-v}_{\Lp{2}(\pi)}^2 \le \liminf_{k \to \infty} \norm{w-v}_{\Lp{2}(\pi_k)}^2 \le \liminf_{k \to \infty} \sfc(\pi_k) \fstop \qedhere
	\]
\end{proof}

\begin{prop}[\Cref{thm1}, Statement~\ref{thm1.2}] \label{prop:thm1.2}
	The lower-semicontinuous envelope of~$\tilde \sfd$ w.r.t.~the $2$-Wasserstein distance over~$\mathcal{P}_2(\Gamma)$ is the discrepancy~$\sfd$. 
\end{prop}

\begin{proof}
	Firstly, let us show that~$\sfd$ is lower-semicontinuous. Let~$\mu_k \to \mu$ and~$\nu_k \to \nu$ be two~${\mathrm{W}_2}$-convergent sequences. For every~$k \in \N$, choose~$\pi_k \in \Pi(\mu_k,\nu_k)$ so that we have~$\abs{\sfc(\pi_k) - \sfd^2(\mu_k,\nu_k)} \to 0$. By Prokhorov’s theorem, see~\cite[Theorem~8.6.2]{Bogachev07}, up to subsequences,~$(\pi_k)_k$ is narrowly convergent to a certain measure~$\pi$. By \Cref{lem:semicontsfc} we deduce that~$\pi \in \Pi(\mu,\nu)$, and therefore
	\[
		\sfd^2(\mu,\nu)
		\stackrel{\eqref{eq2.}}{\le}
		\sfc(\pi)
			\stackrel{\eqref{eq:semicontsfc}}{\le}
		\liminf_{k \to \infty}
		\sfc(\pi_k)
			=
		\liminf_{k \to \infty}
		\sfd^2(\mu_k,\nu_k) \fstop
	\]
	
	Secondly, we shall find sequences~$\bar \mu_k \to \mu$ and~$\bar \nu_k \to \nu$ (w.r.t.~${\mathrm{W}_2}$) such that
	\[
		\sfd^2(\mu,\nu) \ge \limsup_{k \to \infty} \tilde \sfd^2(\bar \mu_k, \bar \nu_k) \fstop
	\]
	Let~$(\bar \pi_k)_{k \in \N} \subseteq \Pi(\mu,\nu)$ be such that~$\sfc(\bar \pi_k) \to \sfd^2(\mu,\nu)$ as~$k \to \infty$. If~$\sfc(\bar \pi_k) = \tilde \sfc(\bar \pi_k)$ for infinitely many~$k$'s, then, up to subsequences,
		\[
			\sfd^2(\mu,\nu) = \lim_{k \to \infty} \sfc(\bar \pi_{k}) = \lim_{k \to \infty} \tilde \sfc(\bar \pi_{k}) \ge \tilde \sfd^2(\mu,\nu) \comma
		\]
		i.e.,~it suffices to take the constant sequences~$\bar \mu_k \coloneqq \mu$ and~$\bar \nu_k \coloneqq \nu$. Otherwise, up to subsequences, we have~$\sfc(\bar \pi_k) < \tilde \sfc(\bar \pi_k)$
		for every~$k$, which implies that
		\begin{equation} \label{eq:propbarpi}
			\norm{y-x}_{\Lp{2}(\bar \pi_k)} = 0 \quad \text{and} \quad \norm{v+w}_{\Lp{2}(\bar \pi_k)} > 0 \comma \qquad k \in \N \fstop
		\end{equation}
		In this case, we set
		\begin{equation}\label{eq:propbarpi2}
			R_k(x,v,y,w) \coloneqq \left(x,v,y+\frac{v+w}{k+1}, w \right) \comma \quad \tilde \pi_k \coloneqq (R_k)_\# \bar \pi_k \comma \qquad k \in \N \comma
		\end{equation}
		as well as
		\[
			\bar \mu_k \coloneqq (\proj_{x,v})_\# \tilde \pi_k = \mu \comma \quad 
			\bar \nu_k \coloneqq (\proj_{y,w})_\# \tilde \pi_k \in \cP_2(\Gamma) \comma \qquad k \in \N \fstop
		\]
		From~\eqref{eq:propbarpi} and~\eqref{eq:propbarpi2}, it follows that 
			$\norm{y-x}_{\Lp{2}(\tilde \pi_k)} 
			= \frac{1}{k+1}\norm{v+w}_{\Lp{2}(\bar \pi_k)}> 0$,
		and since $x = y$~holds $\bar\pi_k$-a.e., 
		we infer that 
			$y - x = \frac{v + w}{k+1}$ 
		holds~$\tilde\pi_k$-a.e.
		Consequently,
		\[
			\tilde \sfc (\tilde \pi_k) = \norm{w-v}_{\Lp{2}(\tilde \pi_k)}^2 = \norm{w-v}_{\Lp{2}(\bar \pi_k)}^2 \fstop
		\]
		Thus,
		\[
			\tilde \sfd^2(\bar \mu_k, \bar \nu_k) 
				\le 
			\tilde \sfc(\tilde \pi_k) 
				= 
			\norm{w-v}_{\Lp{2}(\bar \pi_k)}^2 
				\le 
			\sfc(\bar \pi_k) 
				\to 
			\sfd^2(\mu,\nu)
			\quad \text{as } k \to \infty \fstop
		\]
		Using that~$(\proj_{y,w},\proj_{y,w}\circ R_k)_\# \bar \pi_k \in \Pi(\nu, \bar \nu_k)$, we find
		\[
			{\mathrm{W}_2}(\nu, \bar \nu_k)
				\le
			\frac{\norm{v+w}_{\Lp{2}(\bar \pi_k)}}{k+1}
				\le
			\frac{\norm{v}_{\Lp{2}( \mu)}+\norm{v}_{\Lp{2}( \nu)}}{k+1}
				\to 0 \quad \text{as } k \to \infty \comma
		\]
		which shows that~$\bar \nu_k \to \nu$, as desired.
\end{proof}

\begin{prop}[\Cref{thm1}, Statement~\ref{thm1.3}] \label{prop:thm1.3} Problem~\eqref{eq2.} admits a minimiser.
\end{prop}

\begin{proof}
	By \Cref{lem:semicontsfc}, the function~$\sfc$ is narrowly lower-semicontinuous on~$\Pi(\mu,\nu)$. The set~$\Pi(\mu,\nu)$ is narrowly compact by Prokhorov's theorem, hence a minimiser of~$\sfc$ exists.
\end{proof}

We will denote by~$\Pi_{\mathrm o, \sfd}(\mu,\nu)$ the set of minimisers.
	An analogue of \Cref{prop:thm1.3} does not hold for~$\tilde \sfd$, namely, it is possible that no minimiser in~\eqref{eq1.} exists.%

\begin{ex} \label{ex:minimisertildesfd}
	Let~$n = 2$. For every~$\epsilon \ge 0$ and~$t \in \R$, set
	\begin{align*}
		M_\epsilon^x (t) &\coloneqq (\sin 2\pi (t+\epsilon), \cos 2\pi (t+\epsilon)) \in \cX \comma \\
		M_\epsilon^v (t) &\coloneqq \frac{\dif}{\dif t} M_\epsilon^x(t) \in \cV \\
		M_\epsilon(t) &\coloneqq \bigl(M_\epsilon^x(t), M_\epsilon^v(t)\bigr) \in \Gamma \comma
	\end{align*}
	and observe that these functions are of class~$\mathrm{C}^\infty$ with bounded derivatives, uniformly in~$t$ and~$\epsilon$. 
	Define the measure~$\mu \coloneqq (M_\epsilon)_\# \bigl(\dif t|_{(0,1)}\bigr)$, which is independent of~$\epsilon$, and choose
	\[
		\pi_\epsilon \coloneqq (M_0, M_\epsilon)_\# \bigl(\dif t|_{(0,1)}\bigr) \in \Pi(\mu,\mu) \comma \qquad \epsilon \ge 0 \fstop
	\]
	If~$0 < \epsilon \ll 1$, then~$\norm{y-x}_{\Lp{2}(\pi_\epsilon)} > 0$, and we can write
	\begin{align*}
		\tilde \sfc(\pi_\epsilon) &\stackrel{\eqref{eq:tildesfc}}{=} 12 \int_0^1 \abs{M_0^v(t)}^2 \, \dif t - 12 \frac{\left( \int_0^1 M_0^v(t) \cdot \bigl(M_\epsilon^x(t)-M_0^x(t)\bigr) \, \dif t + o(\epsilon) \right)^2_+}{\int_0^1 \abs{M_\epsilon^x(t)-M_0^x(t)}^2 \, \dif t} + o(1) \\
		&= 12 \int_0^1 \abs{M_0^v(t)}^2 \, \dif t - 12 \frac{\left( \int_0^1 \epsilon\abs{M_0^v(t)}^2 \, \dif t + o(\epsilon) \right)^2_+}{\epsilon^2 \int_0^1 \abs{M_0^v(t)}^2 \, \dif t + o(\epsilon^2)} + o(1) = o(1) \comma
	\end{align*}
	where, in the last identity, we used that~$\int_0^1 \abs{M_0^v(t)}^2 \, \dif t = \norm{v}_{\Lp{2}(\mu)}^2 > 0$. This proves that~$\tilde \sfd(\mu,\mu) = 0$. However, assume that there exists~$\pi \in \Pi(\mu,\mu)$ such that~$\tilde \sfc(\pi) = 0$. If~$\norm{y-x}_{
	\Lp{2}(\pi)} = 0$, then~$\norm{v+w}_{\Lp{2}(\pi)} = \norm{v-w}_{\Lp{2}(\pi)} = 0$, which implies that~$\norm{v}_{\Lp{2}(\mu)} = 0$, which is absurd. If, instead,~$\norm{y-x}_{\Lp{2}(\pi)} > 0$, then~$v=w$ for~$\pi$-a.e.~$(v,w)$, and we have equality in the Cauchy--Schwarz inequality
	\[
		(y-x,v+w)_\pi \le \norm{y-x}_{\Lp{2}(\pi)} \norm{v+w}_{\Lp{2}(\pi)} \comma
	\]
	which means that either~$v=w=0$ for~$\pi$-a.e.~$(v,w)$ (hence~$\norm{v}_{\Lp{2}(\mu)} = 0$), or~$y=x+Tv$ for some~$T > 0$, for~$\pi$-a.e.~$(x,y,v)$. The latter case is excluded by observing that~$(\mathcal G_T)_\# \mu \neq \mu$ for every~$T > 0$, as its space marginal~$(\proj_x \circ \,\mathcal G_T)_\# \mu$ lies on a circle with radius strictly larger than~$1$.
\end{ex}

\begin{cor} \label{cor:meassel}
	There exists a \emph{measurable}  selection~$(\mu,\nu) \mapsto \pi_{\mu,\nu}  \in \Pi_{\mathrm o, \sfd}(\mu,\nu)$.
\end{cor}

\begin{rem}
	We prove measurability w.r.t.~the Borel $\sigma$-algebra of the $2$-Wasserstein topology, which is the same as that of the narrow topology, e.g.,~by the Lusin--Suslin theorem~\cite[Theorem~15.1]{MR1321597}.
\end{rem}

\begin{proof}[Proof of~\Cref{cor:meassel}]
    We shall invoke \cite[Corollary~1]{MR432846}. By~\cite[Theorem~6.18]{MR2459454}, the metric spaces~$X \coloneqq \bigl(\cP_2(\Gamma) \times \cP_2(\Gamma), {\mathrm{W}_2} \oplus {\mathrm{W}_2}\bigr)$ and~$Y \coloneqq\bigl(\mathcal P_2(\Gamma \times \Gamma),{\mathrm{W}_2}\bigr)$ are complete and separable. The set
    \[
        D \coloneqq \set{\bigl((\mu,\nu),\pi\bigr) \in \bigl(\cP_2(\Gamma) \times \cP_2(\Gamma)\bigr) \times \mathcal P_2(\Gamma \times \Gamma) \, : \, \pi \in \Pi(\mu,\nu)}
    \]
    is Borel, as it is the preimage of~$0$ through the continuous map
    \[
        \bigl((\mu,\nu),\pi \bigr) \longmapsto {\mathrm{W}_2}\bigl((\proj_{x,v})_\#  \pi, \mu\bigr) +  {\mathrm{W}_2}\bigl((\proj_{y,w})_\#  \pi, \nu\bigr) \fstop
    \]
    Each section
    \[
        D_{\mu,\nu} = \Pi(\mu,\nu) \comma \qquad \mu,\nu \in \cP_2(\Gamma)
    \]
    is compact by Prokhorov's theorem and \Cref{lem:semicontsfc}. Again by \Cref{lem:semicontsfc}, the real-valued function~$\sfc$ is lower-semicontinuous on~$D_{\mu,\nu}$, for every~$\mu,\nu$. Furthermore, by \Cref{prop:thm1.3}, for every~$\mu,\nu$, there exists~$\pi \in D_{\mu,\nu}$ such that~$\sfc(\pi) = \inf_{\tilde \pi \in D_{\mu,\nu}} \sfc(\tilde \pi)$. Therefore, the hypotheses of \cite[Corollary~1]{MR432846} are satisfied, and there exists a measurable function~$(\mu,\nu) \mapsto \pi_{\mu,\nu} \in D_{\mu,\nu}$ such that~$\sfc(\pi_{\mu,\nu}) = \inf_{\tilde \pi \in D_{\mu,\nu}} \sfc(\tilde \pi)$ for every~$\mu,\nu \in \cP_2(\Gamma)$.
\end{proof}

\begin{rem}
	With a similar proof, one can show the existence of a measurable selection~$(T,\mu,\nu) \mapsto \pi_{T,\mu,\nu}$, where~$\pi_{T,\mu,\nu}$ is a~$\tilde \sfd_T$-optimal plan between~$\mu$ and~$\nu$.
\end{rem}

\subsection{Existence of kinetic optimal maps}\label{ss:32}

\begin{prop}[\Cref{thm1}, Statement~\ref{thm1.4}] \label{prop:thm1.4}
	Let~$\mu,\nu \in \cP_2(\Gamma)$. Assume that~$\mu$ is absolutely continuous with respect to the Lebesgue measure. Then, for every~$T > 0$, there exists a \emph{unique} transport plan~$\pi_T \in \Pi(\mu,\nu)$ optimal for~$\tilde \sfd_T(\mu,\nu)$. Moreover,~$\pi_T$ is induced by a measurable function~$M_T \coloneqq \Gamma \to \Gamma$, i.e.,~$\pi_T = (\id,M_T)_\# \mu$.
	
	Furthermore, there exists a transport map~$M$ such that~$(\id,M)_\# \mu$ is optimal for the time-independent discrepancy~$\sfd(\mu,\nu)$, i.e.,~$(\id,M)_\# \mu \in \Pi_{\mathrm o, \sfd}(\mu,\nu)$.
\end{prop}
Note that we state uniqueness of the map for~$\tilde \sfd_T$, but not for~$\sfd$, see also \S\ref{sec:nonuniq} below.

\begin{proof}
	The first part of the statement, namely the uniqueness of~$\pi_T$ and its representation~$\pi_T = (\id, M_T)_\# \mu$, follows from the classical theory of optimal transport, see in particular \cite[Theorems~10.26 \& 10.38]{MR2459454}. To apply these theorems, we observe that
	\begin{enumerate}
		\item the function~$\tilde d_T^2$ is smooth,
		\item the \emph{twist condition} is satisfied, i.e.,
		\[
		(y,w) \longmapsto \nabla_{x,v} \tilde d^2_T \bigl((x,v),(y,w)\bigr)
		\]
		is injective for every~$(x,v) \in \Gamma$.
	\end{enumerate}

	Let us now move to the proof of the second part of the statement. Let~$\pi \in \Pi_{\mathrm o, \sfd}(\mu,\nu)$ be an optimal transport plan. We will distinguish three cases.
	
	\textbf{Case 1.} Assume that~$\norm{y-x}_{\Lp{2}(\pi)} = 0$, i.e.,~$y=x$ for~$\pi$-a.e.~$(x,y)$. By disintegration, there exists a measure-valued measurable map~$x \mapsto \pi_x \in \mathcal P(\cV \times \cV)$ with
	\begin{equation} \label{eq:disint}
		\int \varphi(x,v,y,w) \, \dif \pi = \iint \varphi(x,v,x,w) \, \dif \pi_x(v,w) \, \dif \eta(x) \comma \qquad \varphi \in \mathrm{C_b}(\Gamma \times \Gamma) \comma
	\end{equation}
	where~$\eta \coloneqq (\proj_x)_\# \pi = (\proj_x)_\# \mu$. Note that we can also write~$\eta = (\proj_y)_\# \pi = (\proj_x)_\# \nu$. Set
	\begin{equation} \label{eq:disint2}
		\mu_x \coloneqq (\proj_v)_\# \pi_x \comma \quad \nu_x \coloneqq (\proj_w)_\# \pi_x \comma \qquad x \in \cX \fstop
	\end{equation}
	Since~$\mu$ admits a density, so does~$\mu_x$ for~$\eta$-a.e.~$x \in \cX$. In particular, there exists a unique ${\mathrm{W}_2}$-optimal transport map from~$\mu_x$ to~$\nu_x$ for~$\eta$-a.e.~$x$, see~\cite[Theorem~6.2.4]{MR2401600}. Therefore, we can apply \cite[Theorem~1.1]{MR2643592} and get a Borel map~$M_2 \colon \Gamma \to \cV$ such that, for~$\eta$-a.e.~$x \in \cV$, the transport plan~$\bigl(\id,M_2(x,\cdot)\bigr)_\# \mu_x$ is optimal for the~$2$-Wasserstein distance between~$\mu_x$ and~$\nu_x$. In particular,
	\begin{align*}
		\int \abs{v-M_2(x,v)}^2 \, \dif \mu
			&=
		\int \abs{v-M_2(x,v)}^2 \, \dif \pi \\
			&\stackrel{\eqref{eq:disint}}{=}
		\iint \abs{v-M_2(x,v)}^2 \, \dif \, (\proj_v)_\# \pi_x(v) \, \dif \eta(x) \\
			&\stackrel{\eqref{eq:disint2}}{=}
		\iint \abs{v-M_2(x,v)}^2 \, \dif \mu_x(v) \, \dif \eta(x) \\
			&\le
		\iint \abs{w-v}^2 \, \dif \pi_x(v,w) \, \dif \eta(x) \\
			&\stackrel{\eqref{eq:disint}}{=}
		\int \abs{w-v}^2 \, \dif \pi = \sfd^2(\mu,\nu) \comma
	\end{align*}
	where the inequality follows from the optimality of~$M_2$. Moreover,~$M \colon (x,v) \mapsto \bigl(x,M_2(x,v)\bigr)$ defines a transport map from~$\mu$ to~$\nu$. We conclude that~$(\id,M)_\# \mu \in \Pi_{\mathrm o, \sfd}(\mu,\nu)$.
	
	\textbf{Case 2.} Assume that~$\norm{y-x}_{\Lp{2}(\pi)} > 0$ and~$(y-x,v+w)_\pi > 0$. Define~$T$ as in~\eqref{eq:defT}. In this case,~$\pi$ is optimal for~$\tilde \sfd_T(\mu,\nu)$, and it is induced by a map.
	
	\textbf{Case 3.} Assume that~$\norm{y-x}_{\Lp{2}(\pi)} > 0$ and~$(y-x,v+w)_\pi \le 0$. We apply disintegration to~$\mu$ and~$\nu$ to find maps~$v \mapsto \mu_v$ and~$v \mapsto \nu_v$ such that
	\[
		\mu(\dif x, \dif v) = \mu_v(\dif x) (\proj_v)_\# \mu(\dif v) \quad \text{and} \quad \nu(\dif x,\dif v) = \nu_v(\dif x) (\proj_v)_\# \nu(\dif v) \fstop
	\]
	Note that $(\proj_v)_\# \mu$ and~$(\proj_v)_\# \mu$-almost every measure~$\mu_v$ are absolutely continuous.
	
	Consider the cost~$(v,w) \mapsto 3\abs{v+w}^2 + \abs{w-v}^2$. Since this function is smooth and satisfies the twist condition, once again we infer the existence of a Borel map~$B \coloneqq \cV \to \cV$ optimal for such a cost from~$(\proj_v)_\# \mu$ to~$(\proj_v)_\# \nu$. Let~$A \colon \Gamma \to \cX$ be {any} Borel function such that~$A(\cdot,v)_\# \mu_v = \nu_{B(v)}$ for~$(\proj_v)_\# \mu$-a.e.~$v \in \cV$. The existence of~$A$ can be deduced, e.g.,~from~\cite[Theorem~1.1]{MR2643592}.
	We claim that the map~$M \colon (x,v) \mapsto \bigl(A(x,v), B(v) \bigr)$ defines an optimal transport map between~$\mu$ and~$\nu$. By construction,~$M_\# \mu = \nu$ and, by optimality of~$B$, we conclude the proof of our claim:
	\begin{multline*}
		\sfc\bigl((\id,M)_\# \mu \bigr) \stackrel{\eqref{eq:sfc}}{\le} 3\norm{v+B(v)}_{\Lp{2}(\mu)}^2 + \norm{v-B(v)}_{\Lp{2}(\mu)}^2 \\ \le 3\norm{v+w}_{\Lp{2}(\pi)}^2 + \norm{w-v}_{\Lp{2}(\pi)}^2 = \sfd^2(\mu,\nu) \fstop \qedhere
	\end{multline*}
\end{proof}

\subsection{Additional results}\label{ss:33}

\subsubsection{Non-uniqueness} \label{sec:nonuniq}
Fix~$\mu,\nu \in \cP_2(\Gamma)$. If~$\mu$ is absolutely continuous, then for every~$T > 0$, there exists a unique minimiser for~$\tilde \sfc_T$ in~$\Pi(\mu,\nu)$, and this plan is induced by a map. This fact follows from the classical theory of optimal transport, cf.~\cite[Theorems~10.26 \& 10.38]{MR2459454}. Nonetheless, non-uniqueness in~$\Pi_{\mathrm o, \sfd}(\mu,\nu)$ may arise in several ways, for example:
\begin{itemize}
		\item there may be two different times~$T_1,T_2 > 0$ for which
		\[
			\sfd^2(\mu,\nu) = \inf_{\pi \in \Pi(\mu,\nu)} \tilde \sfc_{T_1}(\pi) = \inf_{\pi \in \Pi(\mu,\nu)} \tilde \sfc_{T_2}(\pi)
		\]
		\item in the proof of Case 3 in \Cref{prop:thm1.4} there is freedom in the choice of the map~$A \colon \Gamma \to \cX$ for which~$M \colon (x,v) \mapsto \bigl(A(x,v), B(v) \bigr)$ is optimal.
\end{itemize}

Let us provide an example of non-uniqueness.

\begin{ex} \label{ex:nonuniq}
	Fix
	\[
		X_1 = (x_1,v_1) \in \Gamma \comma \quad X_2 = (x_2,v_2) \in \Gamma \comma \quad S > 0 \comma
	\]
	and set
	\[
		\mu \coloneqq \frac{1}{2} \delta_{X_1} + \frac{1}{2} \delta_{X_2} \comma \qquad \nu \coloneqq \frac{1}{2} \delta_{\cG_T(X_1)} + \frac{1}{2} \delta_{X_2} \fstop
	\]
	Note that~$\Pi(\mu,\nu)$ coincides with the set of all the convex combinations of
	\[
		\pi_1 \coloneqq \frac{1}{2} \delta_{(X_1,\cG_S(X_1))} + \frac{1}{2} \delta_{(X_2,X_2)} \comma \quad \pi_2 \coloneqq \frac{1}{2} \delta_{(X_1,X_2)} + \frac{1}{2} \delta_{(X_2,\cG_S(X_1))} \fstop
	\]
	Let us assume that~$\set{x_1,x_2} \neq \set{x_1+Sv_1,x_2}$. Then,~$\sfc = \tilde \sfc$ on~$\Pi(\mu,\nu)$. Taking also into account the concavity of~$\tilde \sfc$ (see~\Cref{prop:thm1.1}), we deduce that
	\[
		\sfd^2(\mu,\nu) = \inf_{\pi \in \Pi(\mu,\nu)} \sfc(\pi) =\inf_{\pi \in \Pi(\mu,\nu)} \tilde \sfc(\pi) = \min\set{\tilde \sfc(\pi_1), \tilde \sfc(\pi_2)} \fstop
	\]
	Straightforward computations yield
	\[
		\tilde \sfc(\pi_1)  =6 \abs{v_2}^2 \comma \quad \tilde \sfc(\pi_2) = 3\abs{v_1+v_2}^2-\frac{3S^2}{2} \frac{\left(\abs{v_1}^2+v_1\cdot v_2\right)_+^2}{\abs{x_2-x_1}^2+\abs{x_2-x_1+Sv_1}^2} + \abs{v_2-v_1}^2 \fstop
	\]
	If, for example, we choose~$x_1=x_2$ and~$v_1 \perp v_2$, we find
	\[
		\tilde \sfc(\pi_1) = \tilde \sfc_S(\pi_1) =6 \abs{v_2}^2 \comma \quad \tilde \sfc(\pi_2) = \tilde \sfc_{2S}(\pi_2) = \frac{5}{2} \abs{v_1}^2 + 4 \abs{v_2}^2 \fstop
	\]
	Therefore, when, additionally,~$5\abs{v_1}^2=4\abs{v_2}^2$, both plans~$\pi_1$ and~$\pi_2$ are optimal. Note that they are induced by maps, and that their corresponding optimal times are different:~$S$ for~$\pi_1$ and~$2S$ for~$\pi_2$. 
    We also observe that~$S$ is exactly the optimal time for~\eqref{eq:naiveacc2} between~$(x_1,v_1)$ and~$(x_1+Sv_1,v_1)$, while~$2S$ is the optimal time between~$(x_1,v_2)$ and~$(x_1+Sv_1,v_1).$ In this case, the structure of~$\sfc$ disadvantages intermediate times between $S$ and $2S$.
\end{ex}

\subsubsection{Characterisation of~$\sfd=0$}

We provide a characterisation of the measures~$\mu,\nu$ such that~$\sfd(\mu,\nu) = 0$, analogous to the particle case of \Cref{rem:counterexamplesd}.

\begin{prop} \label{prop:deq0}
	Let~$\mu,\nu \in \cP_2(\Gamma)$. We have~$\sfd(\mu,\nu) = 0$ if and only if one of the following holds:
	\begin{enumerate}
		\item $\nu = \mathcal (\mathcal G_T)_\# \mu$ for some~$T \ge 0$,
		\item or~$(\proj_v)_\# \mu = (\proj_v)_\# \nu = \delta_0$.
	\end{enumerate}
	If~$(\proj_v)_\# \mu \neq \delta_0$ and~$\nu = (\mathcal G_T)_\# \mu$ for a~$T \ge 0$, then such a~$T$ is unique, and~$\Pi_{\mathrm o, \sfd}(\mu,\nu) = \set{(\id,\mathcal G_T)_\# \mu}$.
\end{prop}

\begin{proof}
	If~$\nu = (\mathcal G_T)_\# \mu$ for some~$T \ge 0$, we have~$\sfc\bigl((\id,\mathcal G_T)_\# \mu \bigr) = 0$. If~$(\proj_v)_\# \mu = (\proj_v)_\# \nu = \delta_0$, then every~$\pi \in \Pi(\mu,\nu)$ has zero cost.
	
	Conversely, assume that~$\sfd(\mu,\nu) = 0$. By \Cref{prop:thm1.3}, there exists~$\pi \in \Pi(\mu,\nu)$ with~$\sfc(\pi) = 0$. By the definition of~$\sfc$, we must have~$v=w$ for~$\pi$-a.e.~$(v,w)$. If~$x=y$ for~$\pi$-a.e.~$(x,y)$, then~$\pi = (\id,\id)_\# \mu$. Otherwise, we have equality in the inequality
	\[
		(y-x,v+w)_\pi \le \norm{y-x}_{\Lp{2}(\pi)} \norm{v+w}_{\Lp{2}(\pi)} \fstop
	\]
	This can happen only if~$v=w=0$ for~$\pi$-a.e.~$(v,w)$, or if there exists~$T \ge 0$ such that~$y = x + Tv$ for~$\pi$-a.e.~$(x,y,v)$.
	
	Assume that~$(\proj_v)_\# \mu \neq \delta_0$. We have already proved that every~$\pi \in \Pi_{\mathrm o, \sfd}(\mu,\nu)$ is of the form~$\pi = (\id, \mathcal G_T)_\# \mu$ for some~$T$. Let us assume, by contradiction, that~$\nu = (\mathcal G_{T_1})_\# \mu = (\mathcal G_{T_2})_\# \mu$ for some~$T_1,T_2 \ge 0$ with~$T_1 < T_2$. By the semigroup property:
	\[
		\nu = (\mathcal G_{T_2})_\# \mu = (\mathcal G_{T_2-T_1})_\# (\mathcal G_{T_1})_\# \mu = (\mathcal G_{T_2-T_1})_\# \nu = \cdots = (\mathcal G_{(T_2-T_1)k})_\# \nu 
	\]
	for every~$k \in \N_{>0}$. For every~$\varphi \in \mathrm{C_c}(\Gamma)$ and~$k \in \N_{>0}$, we thus find
	\[
		\int \varphi(x,v) \, \dif \nu = \int \varphi\bigl(x+k(T_2-T_1)v,v\bigr) \, \dif \nu \fstop
	\]
	Note that~$\lim_{k \to \infty} \abs{x+k(T_2-T_1)v} = \infty$ whenever~$v \neq 0$; hence, since~$\varphi$ is compactly supported, the dominated convergence theorem yields
	\[
		\int \varphi(x,v) \, \dif \nu = \int_{\set{v = 0}} \varphi(x,0) \, \dif \nu \fstop
	\]
	Hence,
	\[
		\int \varphi(x+T_1v,v) \, \dif \mu = \int_{\set{v = 0}} \varphi(x,0) \, \dif \mu \comma
	\]
	which can hold for every~$\varphi \in \mathrm{C_c}(\Gamma)$ only if~$(\proj_v)_\# \mu = \delta_0$.
\end{proof}

\begin{cor} \label{cor:pitoid}
	Let~$\mu_k \rightharpoonup \mu$ and~$\nu_k \rightharpoonup \nu$ be two \emph{narrowly} convergent sequences in~$\cP_2(\Gamma)$. For every~$k \in \N$, pick one~$\pi_k \in \Pi_{\mathrm o, \sfd}(\mu_k,\nu_k)$. Assume:
	\begin{enumerate}[label=(\alph*)]
		\item $\lim_{k \to \infty} \sfd(\mu_k,\nu_k) = 0$,
		\item $(\proj_v)_\# \mu \neq \delta_0$,
		\item \label{ass:bddnorm} $\sup_k \min\set{\norm{v}_{\Lp{2}(\mu_k)}, \norm{v}_{\Lp{2}(\nu_k)}} < \infty$.
	\end{enumerate}
	Then,~$\sfd(\mu,\nu) = 0$, so that~$\nu = (\mathcal G_T)_\# \mu$ for a~$T \ge 0$. Finally,~$\pi_k \rightharpoonup (\id,\mathcal G_T)_\# \mu \in \Pi_{\mathrm o, \sfd}(\mu,\nu)$.
\end{cor}

\begin{rem}
	This corollary would easily follow from \Cref{prop:deq0}, \Cref{lem:semicontsfc}, and the lower semicontinuity of~$\sfd$ if we assumed convergence of~$(\mu_k)_k$ and~$(\nu_k)_k$ w.r.t.~${\mathrm{W}_2}$. Instead, we assume here only narrow convergence.
\end{rem}

\begin{proof}[Proof of~\Cref{cor:pitoid}]
	Since~$(\mu_k)_k$ and~$(\nu_k)_k$ are convergent, they are tight. Therefore, the same is true for~$(\pi_k)_k$. By Prokhorov's theorem, the sequence~$(\pi_k)_k$ admits at least one narrow limit. If one such limit~$\pi$ satisfies~$\sfc(\pi) = 0$, then~$\pi \in \Pi_{\mathrm o, \sfd}(\mu,\nu)$ and~$\sfd(\mu,\nu) = 0$, which yields, by \Cref{prop:deq0},~$\nu = (\mathcal G_T)_\# \mu$ and~$\pi=(\id,\mathcal G_T)_\# \mu$ for some unique $T \ge 0$ (independent of the limit~$\pi$). If, \emph{every} limit~$\pi$ satisfies~$\sfc(\pi) = 0$, then there exists only one limit of the sequence~$(\pi_k)_k$, namely~$(\id,\mathcal G_T)_\# \mu$.
	
	Let us thus prove that every limit~$\pi$ satisfies~$\sfc(\pi) = 0$. Up to extracting a subsequence,~$\pi_k \rightharpoonup \pi$. To begin with, let us note that, by lower semicontinuity of the norm w.r.t.~narrow convergence,
	\[
		\norm{w-v}_{\Lp{2}(\pi)} \le \liminf_{k \to \infty} \norm{w-v}_{\Lp{2}(\pi_k)} \stackrel{\eqref{eq2.}}{\le} \liminf_{k \to \infty} \sfd(\mu_k,\nu_k) = 0 \fstop
	\]
	Moreover, by the triangle inequality,
	\begin{align*}
		\limsup_{k \to \infty} \norm{v+w}_{\Lp{2}(\pi_k)}
			&\le
		\limsup_{k \to \infty} \left( \norm{w-v}_{\Lp{2}(\pi_k)} + 2\min\set{\norm{v}_{\Lp{2}(\mu_k)}, \norm{v}_{\Lp{2}(\nu_k)}} \right) \\
			&\le
		2 \sup_k \min\set{\norm{v}_{\Lp{2}(\mu_k)}, \norm{v}_{\Lp{2}(\nu_k)}} \comma
	\end{align*}
	and the last term is bounded by Assumption~\ref{ass:bddnorm}. Thus, up to subsequences, we may assume that~$\norm{v+w}_{\Lp{2}(\pi_k)}$ converges to a number~$a \in \R_{\ge 0}$. Up to subsequences, we can also assume that~$\norm{y-x}_{\Lp{2}(\pi_k)}$ converges, to an either real or infinite quantity~$b \in \R_{\ge 0} \cup \set{\infty}$. If~$a = 0$, then, by lower semicontinuity of the norm w.r.t.~narrow convergence,
	\begin{equation*}
		\sfc(\pi)
			\stackrel{\eqref{eq:sfc}}{\le}
		3\norm{v+w}_{\Lp{2}(\pi)}^2 + \norm{w-v}_{\Lp{2}(\pi)}^2 \le
		3\liminf_{k \to \infty} \norm{v+w}_{\Lp{2}(\pi_k)}^2 = 3a^2 = 0 \fstop
	\end{equation*}
	If~$b = 0$, again by lower semicontinuity,~$\norm{y-x}_{\Lp{2}(\pi)} = 0$, hence~$\sfc(\pi) = \norm{w-v}_{\Lp{2}(\pi)}^2 = 0$.
	
	From now on, let us assume~$a,b > 0$ and, possibly up to subsequences, that~$\norm{v+w}_{\Lp{2}(\pi_k)}$ and~$\norm{y-x}_{\Lp{2}(\pi_k)}$ are strictly positive for every~$k$. Define
	\[
		c_k \coloneqq \int \abs{\frac{v+w}{\norm{v+w}_{\Lp{2}(\pi_k)}} - \frac{y-x}{\norm{y-x}_{\Lp{2}(\pi_k)}}  }^2 \, \dif \pi_k \comma \qquad k \in \N \fstop
	\]
	We find that
	\[
		\norm{v+w}^2_{\Lp{2}(\pi)} - \frac{\bigl((y-x,v+w)_\pi\bigr)_+^2}{\norm{y-x}^2_{\Lp{2}(\pi)}}
			=
		\norm{v+w}^2_{\Lp{2}(\pi)} \cdot \left(1-\left(1 - \frac{c_k}{2}\right)_+^2 \right)
	\]
	and, consequently,
	\[
		\limsup_{k \to \infty} \min\set{\frac{c_k}{2},1} \le \limsup_{k \to \infty} \left(1-\left(1 - \frac{c_k}{2}\right)_+^2 \right) \le \limsup_{k \to \infty} \frac{\sfd^2(\mu_k,\nu_k)/3}{\norm{v+w}_{\Lp{2}(\pi)}^2} = \frac{0}{a^2} = 0 \fstop
	\]
	This proves that~$c_k \to 0$.
	Let~$\varphi \in \mathrm{C_c}(\Gamma \times \Gamma)$ be non-negative. The convergence
	\[
		\abs{\frac{v+w}{\norm{v+w}_{\Lp{2}(\pi_k)}} - \frac{y-x}{\norm{y-x}_{\Lp{2}(\pi_k)}}  }^2 \varphi
			\to
		\abs{\frac{v+w}{a} - \frac{y-x}{b}  }^2 \varphi
	\]
	is \emph{uniform}. Thus, the narrow convergence~$\pi_k \rightharpoonup \pi$ yields
	\begin{align*}
		\int \abs{\frac{v+w}{a} - \frac{y-x}{b}  }^2 \varphi \, \dif \pi
			&=
		\lim_{k \to \infty} \int \abs{\frac{v+w}{\norm{v+w}_{\Lp{2}(\pi_k)}} - \frac{y-x}{\norm{y-x}_{\Lp{2}(\pi_k)}}  }^2 \varphi \, \dif \pi_k \\
			&\le
		\norm{\varphi}_{\infty} \liminf_{k \to \infty} c_k = 0 \comma
	\end{align*}
	and, by arbitrariness of~$\varphi$,
	\[
		\frac{v+w}{a} = \frac{y-x}{b} \quad \text{for } \pi \text{-a.e.~} (x,v,y,w) \fstop
	\]
	Using the definition~\eqref{eq:sfc} of~$\sfc$, we infer that~$\sfc(\pi) = \norm{w-v}_{\Lp{2}(\pi)}^2 = 0$.
\end{proof}

\begin{cor}
	In the setting of \Cref{cor:pitoid}, additionally set
	\begin{equation}
		T_k \coloneqq \begin{cases}
			\displaystyle 2\frac{\norm{y-x}_{\Lp{2}(\pi_k)}}{(y-x,v+w)_{\pi_k}} &\text{if } (y-x,v+w)_{\pi_k} > 0 \comma \\
			0 &\text{if } \norm{y-x}_{\Lp{2}(\pi_k)} = 0 \comma \\
			\infty &\text{otherwise.}
		\end{cases}
		\qquad k \in \N \fstop
	\end{equation}
	Then,~$T \coloneqq \lim_{k \to \infty} T_k$ exists, is finite, and~$\nu = (\mathcal G_T)_\# \mu$.
\end{cor}

\begin{proof}
	We know from \Cref{cor:pitoid} that~$\pi_k \rightharpoonup (\id,\mathcal G_{\tilde T})_\# \mu$ for some~$\tilde T \ge 0$. Up to extracting a subsequence, we may assume that~$T \coloneqq \lim_{k \to \infty} T_k$ exists in~$\R_{\ge 0} \cup \set{\infty}$. We shall prove that~$T = \tilde T$.
	
	If~$\norm{y-x}_{\Lp{2}(\pi_k)} = 0$ frequently, then~$T=0$ and, by semicontinuity, we obtain
	\[
		0 = \norm{y-x}_{\Lp{2}\bigl((\id, \mathcal G_{\tilde T})_\# \mu\bigr)} = \tilde T \norm{v}_{\Lp{2}(\mu)} \semicolon
	\]
	hence,~$\tilde T = 0$.
	Up to subsequences, we can from now on assume that~$\norm{y-x}_{\Lp{2}(\pi_k)} > 0$ for every~$k$. By \Cref{prop:thm1.1}, we have
	\[
		\sfd(\mu_k,\nu_k) = \sfc(\pi_k) = \tilde \sfc(\pi_k) \ge \int \abs{\frac{y-x}{T_k} - \frac{v+w}{2}}^2 \, \dif \pi_k \comma \qquad k \in \N \fstop
	\]
	
	Let~$\varphi \in \mathrm{C_c}(\Gamma \times \Gamma)$ be non-negative and assume that~$T > 0$. Then, the convergence 
	\[
		\abs{\frac{y-x}{T_k} - \frac{v+w}{2}}^2 \varphi \to \abs{\frac{y-x}{T} - \frac{v+w}{2}}^2 \varphi
	\]
	is uniform; hence,
	\[
		\int \abs{\frac{\tilde T }{T}v - v}^2 \varphi \, \dif \mu = \lim_{k \to \infty} \int \abs{\frac{y-x}{T_k} - \frac{v+w}{2}}^2 \varphi \, \dif \pi_k \le \norm{\varphi}_\infty \liminf_{k \to \infty} \sfd(\mu_k,\nu_k) = 0 \fstop
	\]
	This proves, by arbitrariness of~$\varphi$, that~$\frac{\tilde T}{T}v = v$ for~$(\proj_v)_\# \mu$-a.e.~$v$. Since, by assumption,~$(\proj_v)_\# \mu \neq \delta_0$, we conclude that~$\tilde T = T$.
	
	Let again~$\varphi \in \mathrm{C_c}(\Gamma \times \Gamma)$ be non-negative and assume that~$T = 0$. Now the convergence
	\[
		\abs{y-x-T_k \frac{v+w}{2}}^2 \varphi \to \abs{y-x}^2 \varphi
	\]
	is uniform; hence,
	\[
		\int \abs{\tilde T \, v}^2 \varphi \, \dif \mu = \lim_{k \to \infty} \int \abs{y-x-T_k \frac{v+w}{2}}^2 \varphi \, \dif \pi_k \le \norm{\varphi}_\infty \liminf_{k \to \infty} T_k \sfd(\mu_k,\nu_k) = 0 \fstop
	\]
	We conclude, as before, that~$\tilde T = 0 = T$.
\end{proof}

\section{Dynamical formulations of kinetic optimal transport}\label{sec:sec4}

This section, devoted to the dynamical formulations of kinetic optimal transport that we introduced in \S\ref{sec:sec1} (see \eqref{eq:dtp1.} and \eqref{eq:vlasov2.}), is organised as follows.
\begin{itemize}
    \item In \S\ref{ss:41}, we explore dynamical transport plans in the kinetic setting and prove the equality of $\tilde\sfd_T$ and $\tilde\sfn_T$, together with the existence of a minimiser for \eqref{eq:dtp1.}.
    \item In \S\ref{ss:42}, we better characterise the \emph{optimal spline interpolations} stemming from Theorem \ref{lifthm} and we discuss injectivity of optimal-spline flows. 
    \item In \S\ref{ss:43}, we study the Vlasov's equation \eqref{eq:vlasov.}, and we conclude the proof of \Cref{thm2} with the kinetic Benamou--Brenier formula of Theorem \ref{thmbb}.
    \item In \S\ref{ss:44}, we show propagation of second-order moments along solutions to \eqref{eq:vlasov.}.
    \end{itemize}

\subsection{Dynamical plans}\label{ss:41}
For a fixed~$T > 0$, a $T$-\emph{dynamical (transport) plan} between~$\mu,\nu \in \cP_2(\Gamma)$ is a probability measure~$\mathbf{m} \in \cP\bigl(\mathrm{H}^2(0,T;\cX)\bigl)$ subject to the endpoint conditions \eqref{eq:bdryConds}. 

The theorem below shows that minimising the acceleration functional 
\[
\alpha \longmapsto T \int_0^T \int_{\mathrm H^2(0,T;\cX)} \abs{\alpha''(t)}^2 \, \dif \mathbf{m}(\alpha) \, \dif t  
\]
along $T$-dynamical plans between $\mu$ and $\nu$ is equivalent to computing $\tilde{\mathsf d}^2_T(\mu,\nu)$ via  \eqref{eq0.}. 

The leading idea is that optimal $T$-dynamical plans between $\mu$ and $\nu$ are supported on $T$-\emph{splines} between points~$(x,v) \in \mathrm{supp} (\mu)$ and $(y,w) \in  \mathrm{supp}(\nu)$. Splines are uniquely determined by their endpoints, and their total squared acceleration equals~$\tilde d_T^2\bigl((x,v),(y,w))$. Endpoints chosen according to an optimal coupling~$\pi$ for~$\tilde{\mathsf d}^2_T(\mu,\nu)$ determine an optimal~$\mathbf m$.

\begin{thm}\label{lifthm}
	For every~$\mu,\nu \in \mathcal{P}_2(\Gamma)$ and~$T > 0$, the problem \eqref{eq:dtp1.} admits a minimiser. Moreover, we have the identity
	\begin{equation}
		\tilde \sfn_T(\mu,\nu)
		=
		\tilde \sfd_T(\mu,\nu) \fstop
	\end{equation}
\end{thm}

\begin{proof}
	Fix $T>0$. We build a correspondence between admissible dynamical transport plans for \eqref{eq:dtp1.} and plans in $\Pi(\mu,\nu)$.
	If $\mathbf m \in \cP\bigl(\mathrm{H}^2(0,T;\cX)\bigr)$ is admissible (i.e.,~it satisfies~\eqref{eq:bdryConds}), we have that \begin{equation} \pi_{\mathbf{m}} \coloneqq \biggl(\proj_{\bigl(\alpha(0),\alpha'(0)\bigr)}, \, \proj_{\bigl(\alpha(T),\alpha'(T)\bigr)}  \biggr)_\# \mathbf{m} \in \Pi(\mu,\nu) \fstop
 \end{equation}
	Conversely, given~$\pi \in \Pi(\mu,\nu)$, we construct a~$T$-dynamical plan as follows:~$\mathbf{m}_{\pi}$ is the push-forward of~$\pi$ through the map~$\bigl((x,v),(y,w) \bigr) \mapsto \alpha_{x,v,y,w}^T (\cdot)$, see~\eqref{eq:gammaT}.
	Note 
	$ \pi_{\mathbf{m}_\pi} = \pi$ for~$\pi \in \Pi(\mu,\nu)$, and $\mathbf{m}_{\pi_{{\mathbf m}}} = \mathbf{m}$ for all~$T$-dynamical plans~$\mathbf{m}$ \emph{concentrated on $T$-splines}. 
	
	Let $\mathbf{m}$ be any admissible dynamical transport plan in \eqref{eq:dtp1.}. For every~$\alpha$, it is clear from \eqref{eq:dynd} that $T\,\int_0^T |\alpha''(t)|^2 \, \dif t \geq 12 \left| \frac{y-x}{T} - \frac{v+w}{2} \right|^2 +|v-w|^2,$ where $(x,v),(y,w)$ are the endpoints of $\alpha$.
	Equality holds if and only if $\alpha$ is minimal in \eqref{eq:dynd}.
	Thus,
	\[
	\begin{aligned}
		&T\,\int_0^T \int_{\mathrm H^2(0,T;\cX)} 
		|\alpha''(t)|^2 \, \dif \mathbf{m}(\alpha) \, \dif t 
		= 
		\int_{\mathrm H^2(0,T;\cX)} T\,\int_0^T |\alpha''(t)|^2 \dif t \, \dif \mathbf{m}(\alpha) 
		\\  &\quad 
		\geq 
		\int_{\mathrm H^2(0,T;\cX)} \left( 12\left| \frac{\alpha(T)-\alpha(0)}{T} - 
		\frac{\alpha'(0)+\alpha'(T)}{2} \right|^2 +|\alpha'(0)-\alpha'(T)|^2 \right) \, \dif \mathbf{m}(\alpha) 
		\\  &\quad
		= \int_{\Gamma \times \Gamma} \left(12 \left| \frac{y-x}{T} - \frac{v+w}{2} \right|^2 +|v-w|^2 \right) \, \dif \pi_{\mathbf m}\bigl((x,v),(y,w)\bigr) 
		\geq 
		\tilde \sfd_T^2(\mu,\nu)
		\fstop    
	\end{aligned}
	\]
	Optimising in $\mathbf{m}$, we find 
	$ \tilde \sfn_T(\mu,\nu)
	\geq 
	\tilde \sfd_T(\mu,\nu)$.
	
	Turning to the converse inequality,
	let $\pi$ be optimal in the definition~\eqref{eq0.} of~$\tilde \sfd_T(\mu,\nu)$.
	By definition of~$\mathbf{m}_\pi$, we have 
	\[
	\begin{aligned}
		\tilde \sfn_T^2(\mu,\nu) 
		&\leq 
			T  \int_0^T \int_{\mathrm H^2(0,T;\cX)} |\alpha''(t)|^2 
				 d\mathbf{m_\pi}(\alpha) \,  \dif t  
		\\ &=
		\int_{\Gamma \times \Gamma} T\int_0^T \abs{(\alpha^T_{x,v,y,w})''(t)}^2 
			\, \dif t\, \dif \pi\bigl((x,v),(y,w)\bigr) 
		 \\ &=  
		 	\int_{\Gamma \times \Gamma} \left(\left| 
					\frac{y-x}{T} - \frac{v+w}{2} \right|^2 
					+|v-w|^2 \right) \dif \pi = \tilde \sfd^2_T(\mu,\nu) 
		\comma
	\end{aligned}
	\]
	thanks to the fact that $\mathbf{m}_\pi$ is supported on $T$-splines $\alpha^T_{x,v,y,w}.$
	As a by-product, we have that $\mathbf{m}_\pi$ is optimal in the minimisation problem for $\tilde \sfn_T(\mu,\nu)$. 
\end{proof}

\begin{rem}
	By optimising in $T$, and then taking the lower semi-continuous relaxation~$\mathrm{sc}^-_{\mathrm{W}_2}$ in 
	$\tilde \sfd_T(\mu,\nu) = \tilde \sfn_T(\mu,\nu)$, we also have that 
	\[
	\begin{aligned}
		\tilde \sfd^2(\mu,\nu) &= \inf_{T>0} \,  \inf_{\mathbf{m} \in \mathcal{P}\bigl(\mathrm H^2(0,T;\cX)\bigl) } \set{  T \int_0^T \int \abs{\alpha''(t)}^2 \, \dif \mathbf{m}(\alpha) \, \dif t \quad \text{ subject to~\eqref{eq:bdryConds}}} \comma \\
		\sfd^2(\mu,\nu) &= \mathrm{sc}^-_{\mathrm{W}_2} \, \inf_{T>0} \,  \inf_{\mathbf{m} \in \mathcal{P}\bigl(\mathrm H^2(0,T;\cX)\bigl) } \set{  T \int_0^T \int \abs{\alpha''(t)}^2 \, \dif \mathbf{m}(\alpha) \, \dif t \quad \text{ subject to~\eqref{eq:bdryConds}}} \fstop
	\end{aligned}
	\]
\end{rem}

\subsection{Spline interpolation and injectivity}\label{ss:42}

As pointed out in \S\ref{sss123}, interpolation of measures based on splines is relevant for various applications. For all $\mu,\nu \in \mathcal P_2(\Gamma)$, all $T>0$, and $\pi \in \Pi(\mu,\nu)$ such that $\pi$ is optimal for $\tilde \sfd_T(\mu,\nu)$, the proof of Theorem \ref{lifthm} provides us with an optimal dynamical transport plan~$\mathbf{m}_{\pi}$. The plan $\mathbf{m}_{\pi}$ is supported on splines (parametrised on $[0,T])$ joining points of $\mathrm{supp}(\mu)$ and $\mathrm{supp}(\nu)$.
Hence, we can interpret the curve
\begin{equation}\label{projsplines} 
	[0,T] \ni t \longmapsto \bar{\mu}_t  \coloneqq \left(\proj_{\bigl(\alpha(t),\alpha'(t)\bigr)} \right)_\# \mathbf{m}_{\pi}
\end{equation}
as an optimal \emph{spline interpolation} between $\mu$ and $\nu$. In \S\ref{ss:43}, we will show that the curve $(\bar\mu_t)_t$ satisfies an optimality criterion and it is a solution to Vlasov's equation \eqref{eq:vlasov.}, for a suitable force field $(F_t)_t.$

We start with proving injectivity of the interpolation $\bar{\mu}_t$, whenever $\mu \ll \dif x \, \dif v$.

\begin{prop}[Injective optimal spline interpolation]\label{injtinterp}
	Fix $T>0$ and~$\mu,\nu \in \mathcal{P}_2(\Gamma)$ such that $\mu$ is absolutely continuous with respect to the Lebesgue measure on $\Gamma$. Then, there exists a unique optimal $T$-dynamical transport plan~$\bar {\mathbf{m}}$ for $\tilde \sfn_T(\mu,\nu)$ and, therefore, a unique spline interpolation~$\bar \mu_{\cdot}$.
	Moreover, for every~$t \in [0,T]$, there exists a~$\mu$-a.e.~injective (and measurable) map~$M_t \colon \Gamma \to \Gamma$ such that
	\begin{equation}
		\bar{\mu}_t = (M_t)_\# \mu \fstop
	\end{equation}
\end{prop}

\begin{proof}
	Existence of~$\bar {\mathbf m}$ is ensured by \Cref{lifthm}. The plan~$\pi_{\bar {\mathbf m}}$ is a minimiser for~$\tilde \sfc_T$ in~$\Pi(\mu,\nu)$. Since~$\mu$ is absolutely continuous, \Cref{prop:thm1.4} shows that~$\pi_{\bar {\mathbf m}}$ equals the \emph{unique}~$\bar \pi$ that is optimal for~$\tilde \sfd_T(\mu,\nu)$. Therefore,
	\[
		\bar {\mathbf m} = \bar {\mathbf m}_{\pi_{\bar {\mathbf m}}} = \bar {\mathbf m}_{\bar \pi} \fstop
	\]
	Once~$(\bar \mu_t)_{t \in [0,T]}$ has been defined as in~\eqref{projsplines}, we use \Cref{prop:thm1.4} to find, for every~$t \in [0,T]$, an optimal map~$M_t$ for~$\tilde \sfd_t(\mu,\bar \mu_t)$. Injectivity of~$M_t$ follows from \Cref{injectiveMap}.
\end{proof}

\begin{lem} \label{injectiveMap}
	Fix $T>0$, and let $\mu,\nu \in \mathcal P_2(\Gamma)$. Assume that, for a Borel map $M_T : \Gamma \to \Gamma$, the transport plan $\pi = (\id,M_T)_\# \mu$ is optimal for $\tilde \sfd_T(\mu,\nu)$. Then there exists a Borel set $A \subseteq \Gamma$ of full $\mu$-measure such that, for every $t \in (0,T)$, the map
	\begin{equation}\label{Mt}
		M_t(x,v) := \left(\alpha^T_{x,v,M_T(x,v)}(t), \bigl(\alpha^T_{x,v,M_T(x,v)}\bigr)'(t) \right), \qquad (x,v) \in A
	\end{equation}
	is injective, where $\alpha^T_{x,y,v,w}$ is the solution of \eqref{eq:dynd}.
	Moreover, the set $B := \bigcup_{t \in (0,T)} \{t\} \times M_t(A)$ and the map $B \ni (t,y,w) \mapsto M_t^{-1}(y,w)$ are Borel measurable.
\end{lem}

\begin{proof}
	Define
	\begin{equation*}
		A := \left\{ (x,v) \in \Gamma \, : \, \bigl((x,v),M_T(x,v)\bigr) \in \supp \pi \right\} = (\id, M_T)^{-1}(\supp \pi) \comma
	\end{equation*}
	and notice that
	\begin{equation*}
		\mu(A) = \mu\bigl((\id,M_T)^{-1}(\supp \pi)\bigr) = \pi(\supp \pi) = 1 \fstop
	\end{equation*}
	By cyclical monotonicity \cite[Theorem 1.38]{santambrogio2015optimal}, we know that
	\begin{multline*}
		\tilde \sfd_T^2\bigl((x_1,v_1),M_T(x_1,v_1)\bigr) + \tilde \sfd_T^2\bigl((x_2,v_2),M_T(x_2,v_2)\bigr) \\ \le \tilde \sfd_T^2\bigl((x_1,v_1),M_T(x_2,v_2)\bigr) + \tilde \sfd_T^2\bigl((x_2,v_2),M_T(x_1,v_1)\bigr)  \comma
	\end{multline*}
	for every $(x_1,v_1), (x_2,v_2) \in A$. Hence, injectivity for every $t \in (0,T)$ comes from \Cref{prop:inj}.
	Consequently, the map
	\begin{equation*}
		(0,T) \times A \ni (t,x,v) \longmapsto \bigl(t,M_t(x,v)\bigr) \in B
	\end{equation*}
	is (Borel and) bijective. By the Lusin--Suslin theorem \cite[Corollary 15.2]{MR1321597}, images of Borel sets through injective maps are Borel, from which the second assertion follows.
\end{proof} 
\subsection{Vlasov's equations and the kinetic Benamou--Brenier formula}\label{ss:43}

\subsubsection*{The class of Vlasov's equations}

\begin{defn}
    Let $a,b \in \mathbb R \cup \{\pm \infty\}$ with~$a<b$. Let~$(\mu_t)_{t \in (a,b)} \subseteq \cP_2(\Gamma)$ be a Borel family of probability measures, and let~$F = (F_t)_{t} \colon (a,b) \times \Gamma \to \R^n$ be a time-dependent measurable vector field. Assume that
    \begin{equation}\label{eq:ic}
		\int_a^b \int_\Gamma \bigl( \, \abs{v}+\abs{F_t}\bigr) \, \dif \mu_t \, \dif t < \infty \fstop
	\end{equation}
    We say that $(\mu_t,F_t)_{t \in (a,b)}$ is a solution to \emph{Vlasov's equation}
	\begin{equation}
		\label{eq:cont}
		\partial_t \mu_t + v \cdot \nabla_x \mu_t + \nabla_v \cdot (F_t \, \mu_t) = 0  
	\end{equation}
	if  \eqref{eq:cont} is solved in the weak sense, namely

	\begin{equation} \forall \varphi \in \mathrm{C}^\infty_\mathrm{c}\bigl((a,b) \times \Gamma\bigr) \quad \int_a^b \int_\Gamma \bigl(\partial_t \varphi + v \cdot  \nabla_x \varphi + F_t \cdot \nabla_v \varphi \bigr) \, \dif \mu_t \, \dif t = 0 \end{equation}
	(or, equivalently, for every~$\varphi \in \mathrm{C}^1_\mathrm{c}\bigl((a,b) \times \Gamma \bigr)$, cf.~\cite[Remark~8.1.1]{MR2401600}).
\end{defn}

\begin{prop}[{\cite[Section~4.1.2]{santambrogio2015optimal}}, {\cite[Section~8.1]{MR2401600}}\textcolor{white}{.}]
	Let $(\mu_t,F_t)_t$ be a solution to~\eqref{eq:cont}. Then, up to changing the representative of $(\mu_t)_t$ (i.e.,~changing~$\mu_t$ for a negligible set of times~$t$), the following hold.
	\begin{itemize}
		\item The curve $(\mu_t)_t$ is continuous w.r.t.~the narrow convergence of measures, and extends continuously to the closure~$[a,b]$.%
		\item For all functions $\psi \in \mathrm{C}^\infty_\mathrm{c}(\Gamma)$, the mapping $t \mapsto \int_\Gamma \psi \, \dif \mu_t(x,v)$ is absolutely continuous and it holds true that
		\begin{equation}
		\frac{\dif }{\dif t} \, \int_\Gamma \psi\, \dif \mu_t(x,v) = \int_\Gamma \nabla_{x,v} \psi \cdot (v,F_t) \, \dif \mu_t(x,v) \comma  \qquad \text{for a.e.~} t \in (a,b) \fstop
		\end{equation}
		\item If $(\mu_t)_t$ has a Lipschitz continuous density (in $t,x,v$) and $F_t$ is Lipschitz continuous in $x,v$, then \eqref{eq:cont} is also solved in the a.e.~sense.
	\end{itemize}
\end{prop}

\noindent Let $(F_t)_t$ be a vector field $(a,b) \times \Gamma \to \mathbb R^{d}$, such that 
\begin{equation}
	\label{eq:ic2}
	\begin{aligned}
		&    \int_a^b \left(\sup_B |F_t| + \mathrm{Lip}_B (F_t) \right) \, \dif t < \infty \comma \\  
	\end{aligned}
\end{equation}
for every compact set $B \Subset \Gamma.$
Then, for every~$(x,v) \in \Gamma$, the associated flow $t \mapsto M_t= (x_t,v_t)$ given by 
\begin{equation}
	\label{eq:char2}
	\begin{cases}
		M_a(x,v) = (x,v) \comma \\
		\partial_t M_t = \bigl(v_t,F_t(x_t,v_t) \bigr) \comma
	\end{cases}
\end{equation}
is well-posed in an interval~$[a,a+\epsilon)$ with~$\epsilon > 0$%
, see \cite[Lemma 8.1.4]{MR2401600}.
 In case
\begin{equation}
    \label{eq:global}
    \int_a^b \left(\sup_\Gamma |F_t| + \mathrm{Lip}_\Gamma (F_t) \right) \, \dif t < \infty \comma
\end{equation}
we have global existence of the flow $M_t$, i.e., \eqref{eq:char2} is well-posed in $[a,b]$.

\begin{prop}[{\cite[Lemma 8.1.6 \& Proposition~8.1.8]{MR2401600}}\textcolor{white}{.}]\label{prop38}
	Let $\mu \in \cP_2(\Gamma)$ and let~$(F_t)_{t \in (a,b)}$ be a vector field satisfying \eqref{eq:ic} and \eqref{eq:ic2}.
	\begin{itemize}
	\item Assume that, for~$\mu$-a.e.~$(x,v) \in \Gamma$, the flow~$t \mapsto M_t(x,v)$ defined by \eqref{eq:char2} is well-posed in the interval~$[a,b)$. Then,~$t \mapsto \mu_t \coloneqq (M_t)_\# \mu$ is narrowly continuous and~$(\mu_t,F_t)_t$ is a weak solution to \eqref{eq:cont} in $(a,b)$. 
    
	\item Conversely, given a narrowly continuous curve $(\mu_t)_{t \in [a,b)}$ such that~$(\mu_t,F_t)_t$ solves $\eqref{eq:cont}$ on $(a,b)$, and~$\mu_a = \mu$, then the flow $t \mapsto M_t(x,v)$ associated with $(F_t)_t$ is well-defined on $(a,b)$ for~$\mu$-a.e.~$(x,v)$, and 
	\begin{equation}
		\mu_t = (M_t)_\# \mu \comma \qquad t \in (a,b) \fstop
	\end{equation}
    \end{itemize}
\end{prop}

We conclude the section by adapting the results of \cite[Section 8.2]{MR2401600} to our framework.

\begin{prop}[{\cite[Theorem~8.2.1]{MR2401600}}\textcolor{white}{.}]\label{prop39}
	Let~$(\mu_t)_{t \in [a,b)} \subseteq \cP_2(\Gamma)$ be a narrowly continuous curve such that $(\mu_t,F_t)_t$ is a solution to \eqref{eq:cont} on $(a,b)$, and \begin{equation}
		\int_a^b \int_\Gamma \bigl( \, \abs{v}^2 + \abs{F_t}^2 \bigr) \dif \mu_t \, \dif t <\infty \fstop
	\end{equation}
	Then, there exists $\boldsymbol{\eta} \in \mathcal{P}\bigl(\cX \times \cV \times \mathrm H^2(a,b;\mathcal X)\bigr)$  
	such that 
	\begin{enumerate}
		\item \label{cond1} the measure $\boldsymbol \eta$ is supported on triples $(x,v,\alpha)$ such that $(\alpha,\alpha')$ is an absolutely continuous curve solving \eqref{eq:char2}, with initial conditions $\bigl(\alpha(a),\alpha'(a)\bigr)=(x,v) \in \supp(\mu_a);$
		\item we have 
		\begin{equation}
			\left(\proj_{\bigl(\alpha(t),\alpha'(t)\bigr)}\right)_\# \boldsymbol \eta = \mu_t \comma  \qquad \text{for all } t \in [a,b] \fstop
		\end{equation}
	\end{enumerate}
	Conversely, any $\boldsymbol \eta \in \mathcal{P}\bigl(\cX \times \cV \times \mathrm{H}^2(a,b;\mathcal X)\bigr)$ satisfying Condition~\ref{cond1} and
	\begin{equation}
		\int_a^b \int_{ \mathrm{H}^2(a,b;\mathcal X)} \bigl(\, \abs{\alpha'(t)}^2 + \abs{F_t(\alpha(t),\alpha'(t))}^2 \bigr) \dif \boldsymbol \eta \, \dif t <\infty
	\end{equation}
	induces a solution to \eqref{eq:cont} via 
	\begin{equation}
		\mu_t \coloneqq  \left(\proj_{\bigl(\alpha(t),\alpha'(t)\bigr)}\right)_\# \boldsymbol \eta \comma \qquad  t \in (a,b) \fstop
	\end{equation}
\end{prop}

The measure $\boldsymbol \eta$ is usually referred to as the $\emph{lift}$ of the curve $(\mu_t,F_t)_t.$

\subsubsection*{Regularising Vlasov's equations}
In various technical passages of the next sections, a suitable \emph{regularisation} of Vlasov's equation \eqref{eq:cont} will be necessary. Namely, given a solution $(\mu_t,F_t)_t$ to \eqref{eq:cont}, we aim at finding a family $\bigl((\mu^\epsilon_t,F^\epsilon_t)_t\bigl)_\epsilon$, for $\epsilon >0$, such that each curve $(\mu^\epsilon_t,F^\epsilon_t)_t$ is a classical solution to \eqref{eq:cont} and $\lim_{\epsilon \to 0} (\mu^\epsilon_t,F^\epsilon_t)_t = (\mu_t,F_t)_t$ in a suitable sense. In particular, a desirable feature is that the approximation is tight enough to ensure that $\lim_{\epsilon \to 0} \, \int_a^b\|F^\epsilon_t\|^2_{\mathrm{L}^2(\mu^\epsilon_t)} \, \dif t = \int_a^b \|F_t\|^2_{\mathrm{L}^2(\mu_t)}\, \dif t$, where the non-trivial inequality is $\le$.

Classically, such arguments are obtained by convolution with \emph{some} regularising kernel (e.g., Gaussian mollifiers). A statement like \cite[Lemma 8.1.10]{MR2401600}---where the action $\int_0^T \|F_t\|_{\Lp{2}(\mu_t)}^2 \, \dif t$ is proved to decrease under any convolution operation---holds true also in our setting, with natural adaptations, see also the proof of the lemma below.

By contrast, we need a novel argument to get a counterpart of~\cite[Lemma 8.1.9]{MR2401600}. There, distributional solutions to the continuity equation~$\partial_t \mu_t + \nabla \cdot (X_t \mu_t) = 0$ are approximated with regular solutions to the same equation, via standard convolution. %
Starting from a solution to Vlasov's equation~$
	\partial_t \mu_t + \nabla_{x,v} \cdot \bigl((v,F_t)\mu_t\bigr) = 0$,
the standard convolution simply yields a solution to~$\partial_t \mu^\epsilon_t + \nabla_{x,v} \cdot (X^\epsilon_t \mu^\epsilon_t) =0,$ without ensuring the structure $X^\epsilon_t = (v,F^\epsilon_t)$.  
Indeed, the operator $v \cdot \nabla_x$ is not preserved under this regularisation.

To overcome this difficulty, we use a natural convolution product for kinetic equations, taken from  \cite{MR4527757}.
Consider the Lie group of Galilean translations of $\mathbb R^{1+2n}:$
\begin{multline} (t,x,v) \diamond (s,y,w)
		= (t+s, x+sv+y,v+w) 
		\comma \\ s,t \in \mathbb R \comma \quad x,y \in \cX = \R^n \comma \quad v,w \in \cV = \R^n \fstop\end{multline}
The Galilean inverse is given by $(t,x,v)^{-1} = (-t,-(x-tv),-v)$.
The Lebesgue measure 
on $\R^{1+2n}$ is invariant under left and right translations, i.e.,~the Galilean group is unimodular.
For finite Borel measures $\mu, \nu$ on $\R^{1 + 2n}$, their Galilean convolution $\mu  \star \nu$ is the Borel measure defined by
\begin{equation}
	\int_{\R^{1+2n}} 
		\varphi \,
	\dif\,(\mu  \star \nu)
	=
	\int_{\R^{1+2n}} \int_{\R^{1+2n}} 
		\varphi( \mathfrak{a} \diamond \mathfrak{b} ) 
	\, \dif \nu (\mathfrak{b})
	\, \dif \mu (\mathfrak{a})
	\comma \quad \varphi \in \mathrm{C_b}(\R^{1+2n})
	\fstop
\end{equation}
Measures that are absolutely continuous with respect to the Lebesgue measure will be identified with their density. In particular, for $f, g \in \Lp{1}(\R^{1+2n})$,
we have
\begin{equation}
	(f \star \nu) (\mathfrak{a})
	= 
	\int f( \mathfrak{a} \diamond \mathfrak{b}^{-1})
	\, \dif \nu (\mathfrak{b})
\comma
\quad
	(\mu \star g) (\mathfrak{b})
	= 
	\int g(\mathfrak{a}^{-1} \diamond \mathfrak{b})
	\, \dif \mu (\mathfrak{a}) 	
\fstop
\end{equation}
For vector-valued measures we apply this definition component-wise.
For 
$\mathfrak{b} \in \R^{1+2n}$, we consider the left shift 
$L^\mathfrak{b}: \mathfrak{a} \mapsto \mathfrak{b} \diamond \mathfrak{a}$,
and the right shift 
$R^\mathfrak{b}: \mathfrak{a} \mapsto \mathfrak{a} \diamond \mathfrak{b}$. 
Then,
\begin{align}
	\label{eq:comm-rel}
	L^\mathfrak{b}_\# (\mu  \star \nu) 
	= 
	(L^\mathfrak{b}_\# \mu) \star \nu  
	\comma \quad 
	R^\mathfrak{b}_\# (\mu  \star \nu) 
	= 
	\mu  \star (R^\mathfrak{b}_\# \nu) 
	\fstop
\end{align}

The relevance of the Galilean group for Vlasov's equation becomes apparent when considering infinitesimal Galilean translations. 
Indeed,
for fixed~$\mathfrak{b} = (t,x,v) \in \R^{1+2n}$ with~$t v = 0$, the left translation operators 
	$\bigl(T^{\mathfrak{b}}_s\bigr)_{s \in \R}$ act on functions~$f \in 
	\Lp{1}(\R^{1+2n})
	$ via
\begin{align*}
	\bigl(T^{\mathfrak{b}}_s f \bigr)(\mathfrak{a})
			= f(\mathfrak{a} \diamond s \mathfrak{b}) \comma
			\quad
			\mathfrak{a} \in \R^{1+2n}
			\fstop
\end{align*}
These operators satisfy the group property 
$T^{\mathfrak{b}}_{r} \circ T^{\mathfrak{b}}_{s} 
= T^{\mathfrak{b}}_{s+r}$ for $r,s \in \R$, since $t v = 0$.
For smooth functions we have
\begin{align*}
	\dds\bigg|_{s = 0}  
		T^{\mathfrak{\mathfrak{b}}}_s f 
		=
		\begin{cases}
			\partial_t f + v \cdot \nabla_x f 
				& {\text{if} } \ \mathfrak{b}  = (1,0,0) 
		\comma \\	
			\partial_{x_i} f
				& {\text{if} } \ \mathfrak{b}  = (0, e_i ,0) 
		\comma \\	
			\partial_{v_i} f
				& {\text{if} } \ \mathfrak{b}  = (0,0, e_i)
				\fstop
		\end{cases} 
\end{align*}
Hence, in view of the commutation relation 
	$T_s^\mathfrak{b} (f \star g) 
	= 
	f \star T_s^\mathfrak{b}  g
	$, 
we infer that
\begin{equation}
	\begin{aligned}
	\label{eq:commute}
	(\partial_t + v \cdot \nabla_x) (f \star g) 
		= 
		f \star (\partial_t + v \cdot \nabla_x) g
	\comma \quad
		\partial_{x_i} (f \star g) 
	&	= 
		f \star (\partial_{x_i} g) 
	\comma
	\\
		\partial_{v_i} (f \star g) 
	&	= 
		f \star (\partial_{v_i} g) 
		\fstop		
	\end{aligned}
\end{equation}

\begin{lem}\label{approxlem}
	Let $(\mu_t,F_t)_t$ be a solution on $(a,b)$ to the Vlasov equation \eqref{eq:cont}, with 
	\begin{equation}
		\int_a^b \int_\Gamma \bigl  (\,\abs{v}^2 + \abs{F_t}^2 \bigr) \, \dif \mu_t \, \dif t < \infty \fstop
	\end{equation}
	Then, there exists an approximating sequence $\bigl((\mu^\epsilon_t,F^\epsilon_t)_t\bigr)_{\epsilon >0}$, such that 
	\begin{enumerate}
		\item For all $\epsilon>0,$ the function  $(\mu^\epsilon_t,F^\epsilon_t)_t$ is smooth in $(t,x,v)$, satisfies \eqref{eq:ic}-\eqref{eq:ic2}, and solves the Vlasov equation \eqref{eq:cont} on $(a,b)$ in the classical sense.
		\item The following bounds hold true:
		\begin{align} \int_a^b \int_\Gamma |F^\epsilon_t|^2 \, \dif \mu^\epsilon_t \, \dif t &\leq \int_a^b \int_\Gamma |F_t|^2 \, \dif \mu_t\, \dif t \comma \\
			\int_a^b \int_\Gamma |v|^2 \, \dif \mu^\epsilon_t \, \dif t &\leq \int_a^b \int_\Gamma |v|^2 \, \dif \mu_t \, \dif t + C \, \epsilon^2 \comma
		\end{align}
		for some constant $C>0$, and for all $\epsilon>0$.
		\item The sequence $\bigl((\mu^\epsilon_t,F^\epsilon_t)_t\bigr)_\epsilon$ converges to $(\mu_t,F_t)_t$ as $\epsilon \downarrow 0$ in the following sense:
		\begin{equation} \forall t \in (a,b) \qquad \mu^\epsilon_t  \rightharpoonup \mu_t \quad \text{and} \quad F^\epsilon_t \mu^\epsilon_t \rightharpoonup F_t \mu_t \comma
		\end{equation}
		narrowly, and 
		\begin{equation} \|(v,F^\epsilon_t)\|^2_{\mathrm{L}^2(\mu^\epsilon_t \dif t)} \to \|(v,F_t)\|^2_{\mathrm{L}^2(\mu_t \dif t)} \fstop
		\end{equation}
	\end{enumerate}
\end{lem}

\begin{proof}
	Let us take a smooth function $\eta = \eta(t,x,v) : \mathbb R^{1+2n} \to \R_{\ge 0},$ with globally bounded derivatives, unitary integral, and the moment bound $$\int_{-1}^1\int_\Gamma |v|^2 \eta \, \dif x\,\dif v \, \dif t <\infty \comma$$ that is also symmetric w.r.t.~the variable~$v$ (i.e.,~$\eta(\cdot,\cdot,-v)= \eta(\cdot,\cdot,v)$, for all $v \in \cV$), strictly positive when the variable~$t$ lies in~$(-1,1)$, and equal to~$0$ otherwise. 
	We introduce the mollifiers 
	$$\eta_\epsilon(t,x,v) := \epsilon^{-2-4n} \eta(\epsilon^{-2} t, \epsilon^{-3} x, \epsilon^{-1}v) \comma \qquad  \epsilon>0 \fstop $$
	
	Given $(\mu_t,F_t)_{t \in (a,b)}$, we consider their trivial extension to curves defined on~$\mathbb R$: 
	$$F_t = 0 \quad \text{and} \quad \partial_t \mu_t + v\cdot \nabla_x \mu_t =0 \quad \text{for } t \notin [a,b] \fstop$$
	We define $E \coloneqq F \,\mu$ and consider the regularised measures
	$$\mu^\epsilon 
		\coloneqq 
	\eta_\epsilon \star \mu \comma 
	\quad E^{\epsilon} 
		\coloneqq 
	\eta_\epsilon \star E \comma 
		\qquad  \epsilon>0 \comma$$
	Smoothness of $\mu^\epsilon$ and $E^\epsilon$ are indeed a consequence of the last display in \cite[Page 6]{MR4527757}: 
    \[
    (\nabla_v + t\nabla_x, \nabla_{t,x}) 
		\, (\mu^\epsilon,E^\epsilon) 
	= (\nabla_v + t\nabla_x,\nabla_{t,x}) 
		\eta_\epsilon \star (\mu, E)
	\fstop
    \]

	Let now 
		$F^\epsilon \coloneqq \frac{E^\epsilon}{\mu^\epsilon}$, 
	where we identify regular measures with their densities. 
	Following the proof of \cite[Lemma 4.2]{MR4527757}, 
	we will show that $(\mu^\eps, F^\eps)$
	solves the Vlasov equation
	$$ \partial_t \mu^\epsilon_t + v \cdot \nabla_x \mu^\epsilon_t + \nabla_v \cdot (F_t^\eps \mu_t^\eps ) =0 \fstop $$
	Indeed, write 
	$\tilde   \eta^\epsilon(t,x,v) 
			= \eta^\epsilon((t,x,v)^{-1})$ 
	where $(t,x,v)^{-1}$ denotes the Galilean inverse.
	Using 
	\eqref{eq:commute}
	and the fact that 
	$(\mu, F)$ solves the Vlasov equation,
	we obtain for any test function
	$\varphi \in \mathrm{C}^\infty_\mathrm{c}\bigl((a,b) \times \Gamma\bigr)$,
	\begin{align*}
		&\int_\R \int_\Gamma 
			(\partial_t + v \cdot \nabla_x) 
				\varphi 
		\dd \mu^\eps
		= 
		\int_\R \int_\Gamma 
			\tilde \eta_\epsilon 
			\star
			(\partial_t + v \cdot \nabla_x) 
				\varphi 
		\dd \mu
		= 
		\int_\R \int_\Gamma 
			(\partial_t + v \cdot \nabla_x) 	
			(\tilde \eta_\epsilon 
			\star \varphi) 
		\dd \mu			
		\\&= 
		- \int_\R \int_\Gamma 
			\nabla_v
			(\tilde \eta_\epsilon 
			\star \varphi) 
		\dd (F\mu)		
		=	
		- \int_\R \int_\Gamma 
			\tilde \eta_\epsilon 
			\star \nabla_v \varphi
		\dd (F\mu)		
		=	
		- \int_\R \int_\Gamma 
			\nabla_v \varphi
			\cdot F^\eps
		\dd \mu^\eps \comma
	\end{align*}
	which proves the claim.
	Since $\eta^\eps$ is an approximation of the identity, it holds true that
    \[
    \mu^\epsilon \rightharpoonup \mu_t\,\dif t \comma \quad E^\epsilon \rightharpoonup F_t\,\mu_t \comma \qquad \text{as } \epsilon \to 0 \fstop
    \]
    Using Jensen's inequality for the jointly convex function $(E,\mu) \mapsto \frac{|E|^2}{\mu}$ as in \cite[Lemma 8.1.10]{MR2401600}, we obtain the pointwise inequality
	\begin{align*}
		|F^\epsilon|^2 \, \mu^\epsilon
		= 
			\frac{|\eta^\epsilon \star E|^2}
				 {\eta^\epsilon \star  \mu}
		\leq 
		\eta^\epsilon \star (|F|^2 \, \mu) \fstop
	\end{align*}
    Integration over $\R \times \Gamma$ yields
    \[    \|F_t^\epsilon\|^2_{\mathrm{L^2}(\mu^\epsilon_t\,\dif t)} \leq \|F_t\|^2_{\mathrm{L}^2(\mu_t\,\dif t)} \comma \qquad  \epsilon >0 \fstop
    \]
    By \cite[Proposition 5.18]{santambrogio2015optimal}, we have 
    \[
\|(v,F_t)\|^2_{\mathrm{L}^2(\mu_t\,\dif t)} \leq \liminf_{\epsilon \downarrow 0} \, \|(v,F^\epsilon_t)|^2_{\mathrm{L}^2(\mu^\epsilon_t\,\dif t)} \fstop
    \]
    Finally, using that$\int v \, \eta(\cdot,\cdot,v) \, \dif v = 0$, we obtain 
	\begin{align*}
	& \int_{\mathbb R \times \Gamma} 
		|v|^2 \, 
		\mu^\epsilon_t(x,v) \, \dif x \, \dif v \, \dif t
	=
	\int_{\mathbb R \times \Gamma} 
		|v|^2 \, 
		\dd( \eta^\epsilon \star \mu)	
	=
	\int_{\mathbb R \times \Gamma} 
		\eta^\epsilon  \star |v|^2 \, 
		\dd \mu
	\\&= 	
	\int_{\R \times \Gamma} 
	\int_{\R \times \Gamma}
		\eta_\epsilon
			(s,y,w)\,
		|v - w|^2  \,
		 \dif s \, \dif y \, \dif w \, 
		  \dif \mu_{t}(x,v)\,\dif t
	\\&= 	
	\int_{\R \times \Gamma} 
	\int_{\R \times \Gamma}
		\eta_\epsilon
			(s,y,w)\,
		\big(|v|^2 + |w|^2\bigr)  \,
		 \dif s \, \dif y \, \dif w \, 
		  \dif \mu_{t}(x,v)\,\dif t
	\\& = 
	\|v\|^2_{\mathrm{L}^2(\mu_t\,\dif t)}
	+
	\int_{\R \times \Gamma}
		\eta_\epsilon
			(s,y,w)\,
		 |w|^2  
		  \dif s \, \dif y \, \dif w \,
	\\&\leq \|v\|^2_{\mathrm{L}^2(\mu_t\,\dif t)} + C \epsilon^2,
    \end{align*}
   with an explicit constant $C>0$ independent of $\epsilon$.		  
\end{proof}

\subsubsection*{Proof of the kinetic Benamou--Brenier formula}

In classical optimal transport, the Kantorovich problem admits an equivalent fluid-dynamics formulation, as was shown by J.-D.~Benamou and Y.~Brenier \cite{benamou2000computational}.
The idea is that optimally transporting $\rho_0$ to $\rho_1$ is equivalent to finding the minimal velocity field~$(V_t)_t$ one should apply to make particles \emph{flow} from one measure to the other. This velocity field induces an evolution of measures~$t \mapsto \rho_t$ that satisfies the continuity equation
$$\partial_t \rho_t + \nabla \cdot (V_t \rho _t) =0 \fstop$$
Here, we recover a similar interpretation for the second-order discrepancy $\sfd.$
The \emph{kinetic optimal transport} between $\mu,\nu$ is given by the minimal force field $(F_t)_t$ required to \emph{push} particles from $\mu$ to $\nu$. In this case,~$t \mapsto \mu_t$ evolves according to the Vlasov equation \eqref{eq:cont}.

\begin{thm}[Kinetic Benamou--Brenier formula]
	\label{thmbb}
	For every~$\mu,\nu \in \mathcal{P}_2(\Gamma)$ and~$T > 0$, the problem \eqref{eq:vlasov2.} admits a minimiser. Moreover, we have the identities
	\begin{equation}
		\tilde \sfn_T(\mu,\nu)
		=
		\tilde \sfd_T(\mu,\nu) 
		= \widetilde{\mathcal{MA}}_T(\mu,\nu) \fstop
	\end{equation}
\end{thm}

\begin{proof}
	Fix $T>0.$ We say that a curve $(\mu_t)_t :[0,T] \to \mathcal{P}_2(\Gamma)$ is admissible and belongs to the class~$\mathcal N_T(\mu,\nu)$ if $(\mu_t,F_t)_t$ solves \eqref{eq:cont} for a vector field $(F_t)_t$ satisfying 
	\begin{equation}
		\label{ic3}
		\int_0^T \int_\Gamma \bigl(\, \abs{v}^2 + \abs{F_t}^2 \bigr) \, \dif \mu_t \,\dif t < \infty
	\end{equation}
	and  
	\begin{equation}
		\label{constr}
		\begin{aligned}
			&\mu_0 = \mu \comma \qquad \mu_T = \nu \fstop    
		\end{aligned} 
	\end{equation} 
	We shall prove that  $\tilde{\sfd}_T \geq \widetilde{ \mathcal{MA}}_T \geq \tilde{\mathsf{n}}_T$, which is sufficient, since $\tilde{\mathsf{n}}_T = \tilde{\mathsf{d}}_T$, in view of Theorem \ref{lifthm}. 
	To this end, fix $\mu, \nu \in \mathcal{P}_2(\Gamma)$ and, for now, 
	assume that $\mu$ is absolutely continuous with respect to the Lebesgue measure. 
	We shall prove that 
	$$\inf_{\pi \in \Pi(\mu,\nu)}  \tilde \sfc_T (\pi) \geq \inf_{(\mu_t)_t \in \mathcal{N}_T(\mu,\nu)} \int_0^T \int_\Gamma |F_t|^2 \, \dif \mu_t \, \dif t \fstop    $$
	Notice first that the value $\inf_{\pi \in \Pi(\mu,\nu)} \tilde{\sfc}_T(\pi)$ is attained at a transport plan of the form $\pi =(\id,M_T)_\# \mu$, for a  map $M_T = (Y_T,W_T) \colon \Gamma \to \Gamma$, see Theorem \ref{thm1}.
	Now define the flow $(M_t)_t,$ for $t\in[0,T]$, via 
	$$M_t(x,v) = (x_t,v_t), \qquad (x_t,v_t) = \left(\alpha^T_{x,v,M_T(x,v)}(t), (\alpha^T_{x,v,M_T(x,v)})'(t) \right) \comma \qquad t\in [0,T] \comma $$
	using the same notation as Lemma \ref{injectiveMap}.
	For every~$t \in (0,T)$, the map $M_t$ is injective on a full $\mu$-measure set, as shown in Lemma \ref{injectiveMap}.  
	Let now
	\begin{equation}\label{sFt}
		F_t(y,w) \coloneqq 
		\frac{\mathrm{d}^2}{\mathrm{d}s^2}\bigg|_{s = t}
		\alpha^T_{M_t^{-1}(y,w), M_T \circ M_t^{-1}(y,w)}(s) \comma \qquad t \in (0,T) \fstop
	\end{equation}
	It is clear that $(M_T)_\# \mu = \nu$, and, setting $\bar\mu_t = (M_t)_\# \mu$, we claim that $(\bar\mu_t)_t$ is a narrowly-continuous curve such that \eqref{eq:cont} holds, with the vector field $(F_t)_t$ given above. To prove this, we fix a smooth test function $\varphi \in \mathrm{C}^\infty_\mathrm{c} \bigl([0,T] \times \Gamma\bigr)$ and compute
	\[
	\begin{aligned}
		&\int_0^T \int_\Gamma \partial_t \varphi(t,x,v) \dif \bar\mu_t(x,v)  \dif t 
		= \int_0^T \int_\Gamma \partial_t \varphi(t,x_t,v_t) \, \dif \mu_0(x,v) \dif t  
		\\& = \int_0^T \int_\Gamma \left(\frac{\dif }{\dif t} \varphi(t,x_t,v_t) - v_t \cdot \nabla_x \varphi(t,x_t,v_t) -F_t(x_t,v_t) \cdot \nabla_v \varphi(t,x_t,v_t) \right) \dif \mu_0(x,v) \dif t 
		\\ & =  \int_\Gamma \varphi(0,x,v) \, \dif \mu_0  
		- \int_\Gamma \varphi(T,x_T,v_T) \dif \mu_0  
		- \int_0^T \int_\Gamma \left( v \cdot \nabla_x \varphi  + \nabla_v \varphi \cdot \nabla_v F_t \right) \, \dif \bar\mu_t \dif t \comma
	\end{aligned}
	\]
	which is the weak formulation of \eqref{eq:cont} with fixed endpoints, as
	$$\int_\Gamma \varphi(T,x_T,v_T) \, \dif \mu_0(x,v) 
	= \int_\Gamma \varphi(T,\cdot,\cdot) \, \dif \mu_T
	= \int_\Gamma \varphi(T,\cdot,\cdot) \, \dif \nu \fstop$$
	It it easy to check---using \eqref{eq:gammaT}---that 
	$$ \int_0^T \int_\Gamma |v_t|^2 \, d\bar\mu_t \dif t \lesssim_T \iint \bigl( |x|^2+ |v|^2 + |y|^2 +|w|^2 \bigr) \, \dif \mu(x,v) \, \dif \nu(y,w) < \infty \fstop$$
	In addition,
	\[ 
	\begin{aligned}
		T \int_0^T \int_\Gamma |F_t|^2 d\bar\mu_t \dif t 
		&= T \int_0^T \int_\Gamma |F_t(x_t,v_t)|^2 \dif \mu \dif t
		=  
		\int T \int_0^T |F_t(x_t,v_t)|^2 \dif t \dif \mu 
		\\ &= \int_\Gamma \left(  12  
			\left|  \frac{Y_T(x,v) -x }{T} - \frac{W_T(x,v)+v}{2}\right|^2 
					+ |v-W_T(x,v)|^2 \right) \, \dif \mu(x,v)  
		\\ &= \tilde \sfc_T(\pi) = \tilde{\sfd}^2_T(\mu,\nu) \comma
	\end{aligned}
	\]
	since $\pi$ is optimal.
	Then, $(\bar\mu_t)$ solves \eqref{eq:cont}---in particular it belongs to the class $\mathcal{N}_T(\mu,\nu)$---and $\widetilde{\mathcal{MA}_T}^2(\mu,\nu) \leq  \tilde \sfd^2_T(\mu,\nu),$ every time $\mu$ is absolutely continuous.  
    
	We get rid of this additional assumption. 
	Let $(\mu^k)_k$ be an approximation of $\mu$ in the Wasserstein metric, such that $\mu^k$ is absolutely continuous for all $k \in \mathbb N$, and let $M^k_T$ be the optimal transport map for $\tilde \sfd_T(\mu^k,\nu).$  Define the splines flow $(M^k_t)_t$ associated with $M^k_T$ via \eqref{Mt}, let $\mu^k_t := (M^k_t)_\# \mu^k$, and let $(F^k_t)_t$ be given by \eqref{sFt}. Note that $(\mu^k_t,F^k_t)_t$ solves \eqref{eq:cont}, for all $k\in \mathbb N$.
	Using the explicit expressions of \eqref{eq:gammaT}, and indicating with $(x_t,v_t)_t=M^k_t$ the solution of \eqref{eq:char2}, we find 
	\[
	\begin{aligned}
		&\int_0^T \int \bigl(\, \abs{v}^2 + |F_t^k|^2 \bigr) \, \dif \mu^k_t \dif t = \int_0^T \int \bigl( \, \abs{v_t}^2 + |F^k(x_t,v_t)|^2 \bigr) \dif \mu_0^k \, \dif t \\ &=
		\int \int_0^T \bigl( \, \abs{v_t}^2 + |F^k(x_t,v_t)|^2 \bigr) \dif t \, \dif \mu_0^k \lesssim_T \int \bigl( \, |x|^2 + |v|^2 + |M_T^k(x,v)|^2 \bigr) \dif \mu_0^k \\
		&\leq \int \bigl( \, |x|^2+|v|^2+|y|^2+|w|^2 \bigr) \, \dif \mu^k_0(x,v) \, d\nu(y,w) \\
		&\lesssim 1+ \int \bigl( \, |x|^2+|v|^2+|y|^2+|w|^2 \bigr) \, \dif \mu_0(x,v) \, \dif \nu(y,w) \leq C < \infty \fstop
	\end{aligned}
	\]
	In addition,
	\[
	\int (\, |x|^2+|v|^2) \, \dif \mu^k_t(x,v) = \int (|x_t|^2 + |v_t|^2) \, \dif \mu_0^k \leq C' < \infty, \,\, \text{uniformly in} \, t\in[0,T], \, k\in \mathbb N \fstop
	\]
	Then, following \cite[Lemma 4.5]{dolbeault2009new} we have that, up to a subsequence, $\mu^k_t \rightharpoonup \bar\mu_t$ for all $t \in [0,T]$, and $(v,F^k_t) \mu^k_t \dif t \rightharpoonup Z$ narrowly, for some measures $\bar\mu_t \in \mathcal{P}_2(\Gamma)$ and $Z\in \mathcal{M}([0,T]\times\Gamma;\mathbb R^{2n})$. By uniform integrability of $t\mapsto \int |(v,F^k_t)|\,\dif \mu_t$ with respect to $k$, we have that 
	$Z = \Xi_t \, \dif t$, for a vector-valued measure $\Xi_t$ satisfying 
	\[
	\int \int_\Gamma \frac{|\Xi_t|^2}{\bar\mu_t} \, \dif t \leq \liminf_{k \to \infty} \int_0^T \int_\Gamma \bigl( \, \abs{v}^2 + |F^k_t|^2 \bigr) \, \dif \mu^k_t \, \dif t \fstop
	\]
	Finally, by \cite[Proposition 5.18]{santambrogio2015optimal}, we have that $\Xi_t = X_t \, \bar\mu_t$ for a vector field $X_t = (X^{(1)}_t,X^{(2)}_t) \in \mathrm{L}^2(\bar\mu_t;\mathbb R^{2n})$ and a.e.~$t\in[0,T]$. By weak convergence, $\nabla_x \cdot (X^{(1)}_t \mu_t) = v \cdot \nabla_x \bar\mu_t$. Let $F_t \coloneqq X^{(2)}_t$.
	Passing to the limit in the weak formulation of \eqref{eq:cont}, we have that $(\bar\mu_t,F_t)_t$ is a solution to \eqref{eq:cont}, such that $(\bar\mu_t)_t$ is admissible for $\widetilde{\mathcal{MA}_T}(\mu,\nu)$.
	
	Using lower semi-continuity (see again \cite[Proposition 5.18]{santambrogio2015optimal}), we achieve
	\begin{multline*}
	    \widetilde{\mathcal{MA}_T^2}(\mu,\nu) \leq \int_0^T \int |F_t|^2 \, \dif \bar\mu_t \, \dif t \leq \liminf_{k \to \infty} \int_0^T \int |F^k_t|^2 \, \dif \mu^k_t \dif t = \liminf_{k \to \infty} \tilde \sfd^2_T(\mu^k,\nu) \\ = \tilde \sfd_T^2(\mu,\nu) \comma
	\end{multline*}
	where the second to last equality holds because $(\mu^k_t,F^k_t)_t$ are optimal spline interpolations, and the last equality is a consequence of the Wasserstein convergence $\mu^k \to \mu$ and of the sequential Wasserstein continuity of $\tilde \sfd_T$.

	For the inequality $ \tilde \sfn_T \leq \widetilde{\mathcal{MA}}_T$, let $(\mu_t,F_t)_t$ be any admissible curve in \eqref{eq:vlasov2.}. By the smoothing procedure of Lemma \ref{approxlem}, we can find a sequence $\bigl((\mu^\epsilon_t, F^\epsilon_t)_t\bigr)_\epsilon$ of classical solutions to \eqref{eq:cont} such that
	$$\int_0^T \int |F^\epsilon_t|^2 \dif \mu^\epsilon_t \,\dif t \leq \int_0^T \int |F_t|^2 \dif \mu_t \,\dif t \comma \quad \int_0^T \int |v|^2 \, \dif \mu^\epsilon_t \,\dif t \leq 1+ \int_0^T \int |v|^2 \, \dif \mu_t \, \dif t \fstop$$
	By Proposition~\ref{prop38}---more precisely following \cite[Proposition 8.1.8]{MR2401600}---for all $\epsilon>0$, we have~$\mu^\epsilon_t = (M^\epsilon_t)_\# \mu^\epsilon_0$ for all~$t \in [0,T]$, where $M^\epsilon$ is the flow generated by the vector field~$(v,F^\epsilon_t)_t.$ %
	Let $\mathbf{m}^\epsilon \in \mathcal{P}\bigl(\mathrm{H}^2(0,T;\mathcal{X})\bigr)$ be defined via $$\mathbf m^\epsilon := \int \delta_{M^\epsilon(x,v)} \, \dif \mu_0^\epsilon(x,v) \fstop$$
	For all $\epsilon>0$ and~$t \in [0,T]$, it holds true that \[ (\proj_{(\alpha(t),\alpha'(t))})_\# \mathbf{m}^\epsilon = \mu^\epsilon_t\quad \text{and} \quad (\proj_{( \alpha(t),\alpha'(t))})_\# \bigl(\, \abs{\alpha''(t)}^2\mathbf{m}^\epsilon\bigr) = \abs{F_t^\epsilon}^2 \mu^\epsilon_t \fstop \]
	Then, as in \cite{MR2401600}, we have that the sequence $(\mathbf{m}^\epsilon)_\epsilon$ is tight, and we call $\mathbf{m}$ any narrow limit point of $(\mathbf{m}^\epsilon)_\epsilon$. 
	Narrow convergence, together with Lemma \ref{approxlem}, ensures that $(\proj_{(\alpha(t),\alpha’(t))})_\# \mathbf m = \mu_t$, for all $t\in[0,T]$, and, in particular, $(\proj_{(\alpha(0),\alpha’(0))})_\# \mathbf m = \mu,$ and $(\proj_{(\alpha(T),\alpha’(T))})_\# \mathbf m = \nu$. %
	By semicontinuity, 
	\[
	\begin{aligned}
		&\tilde \sfn_T^2(\mu,\nu) \le \int_0^T \int_{\mathrm H^2(0,T;\mathcal{X})} |\alpha''(t)|^2 \, \dif  \mathbf{m}(\alpha) \, \dif t \leq \liminf_{\epsilon \downarrow 0} \int_0^T \int_{\mathrm H^2(0,T;\mathcal{X})} |\alpha''(t)|^2 \dif  \mathbf{m^\epsilon} \dif t \\ &= \liminf_{\epsilon \downarrow 0} \int_0^T \int |F^\epsilon_t|^2 \dif \mu^\epsilon_t \, \dif t = \int_0^T \int |F_t|^2 \, \dif \mu_t \, \dif t,
	\end{aligned}
	\]
	where we used the strong convergence induced by Lemma \ref{approxlem}. 
	This concludes the equivalence, by taking the infimum over $(\mu_t,F_t)_t$.
    As a by-product, the curve $(\bar \mu_t)_t$ built above is a minimiser in \eqref{eq:vlasov2.}.
\end{proof}

\begin{rem}
	\emph{A posteriori}, the proof shows that optimal curves in \eqref{eq:vlasov2.} are given by injective interpolation along splines, when $\mu$ is absolutely continuous, see also \Cref{injtinterp}.
	Indeed, in this case, when $\mu \ll \dif x \, \dif v$, the curve $\mu_t = (M_t)_\#\mu$ is optimal in \eqref{eq:vlasov2.}, where $M_t$ is the flow of \eqref{Mt}. 
	The general case is a mixture of spline interpolations. 
\end{rem}

\begin{rem}\label{rem:rem}
	Our result proves\footnote{in case only two measures are considered} a conjecture of \cite{chen2018measure}, i.e.,~the equivalence of \cite[Formula~(14)]{chen2018measure} and \cite[Formula~(3)]{chen2018measure}.  
    Indeed, in our language \cite[Formula~(14)]{chen2018measure} reads
	\[ \inf_{\mu_0,\mu_1} \set{ \, \tilde{ \mathsf{n}}^2_1 (\mu_0,\mu_1) \, :\, (\proj_x)_\# \mu_i = \rho_i, \, i=0,1 },
    \] 
    while \cite[Formula~(3)]{chen2018measure} corresponds to 
    \[\inf_{\mu_0,\mu_1} \, \set{\widetilde{\mathcal{MA}}^2_1 (\mu_0,\mu_1) \, :\, (\proj_x)_\# \mu_i = \rho_i, \, i=0,1 },\]
    and equality between the two is a straightforward consequence of Theorem \ref{thmbb}.

    Y.~Chen, G.~Conforti, and T.~T.~Georgiou conjecture such an equivalence in \cite[Claim 4.1]{chen2018measure}, and provide a formal argument in favour of it. At the same time, the authors remark the lack of a rigorous proof. Our Theorem \ref{thmbb} fills the gap and completes the proof, by building on the argument of~\cite{chen2018measure}  with the crucial addition of two new ingredients:
    the injectivity of the map $M_t$ (allowing for the definition of $F_t$) and the Galilean approximation of solutions to \eqref{eq:cont} via Lemma \ref{approxlem}.
\end{rem}

\subsection{Moment estimates for Vlasov's equations}\label{ss:44}
In this section we prove propagation estimates for moments along solutions to \eqref{eq:cont}. In particular, the following results show that a solution $(\mu_t)_{t\in[a,b]}$ of \eqref{eq:cont} stays in $\mathcal{P}_2(\Gamma)$, provided the initial datum $\mu_a \in \mathcal{P}_2(\Gamma).$

Before turning to the rigorous estimates, let us give a heuristic argument. Let $(\mu_t,F_t)_t$ be a solution to \eqref{eq:cont}. Then, formally,
\[
	\frac{\dif}{\dif t} \norm{v}_{\Lp{2}(\mu_t)}^2 = -\int |v|^2  \nabla_{x,v} \cdot ((v,F_t )\mu_t) = 2 \int v \cdot F_t \, \dif \mu_t \leq 2\norm{v}_{\Lp{2}(\mu_t)} \norm{F_t}_{\mathrm{L}^2(\mu_t)} \comma
\]
from which we obtain~$	\frac{\dif}{\dif t} \norm{v}_{\Lp{2}(\mu_t)} \le \norm{F_t}_{\mathrm{L}^2(\mu_t)}$. Similarly,
\[
	\frac{\dif}{\dif t} \norm{x}_{\Lp{2}(\mu_t)}^2 = -\int |x|^2 \nabla_x \cdot (v \, \mu_t) =  2\norm{x}_{\Lp{2}(\mu_t)} \norm{v}_{\mathrm{L}^2(\mu_t)} \comma
\]
and, therefore,~$\frac{\dif}{\dif t} \norm{x}_{\Lp{2}(\mu_t)} \le \norm{v}_{\mathrm{L}^2(\mu_t)}$.

\begin{lem}[Moment estimate] \label{lem:momentestimate}
	Let $(\mu_t,F_t)_t$ be a narrowly continuous solution on $[a,b]$ to the Vlasov equation \eqref{eq:cont} with $\mu_a \in \mathcal P_2(\Gamma)$ and $\int_a^b \int |F_t|^2 \, \dif \mu_t \, \dif t < \infty$.
	Then, for every $t \in (a,b)$:
	\begin{equation}
	\|v\|_{\mathrm{L}^2(\mu_t)} \le \|v\|_{\mathrm{L}^2(\mu_a)} + \int_a^t \|F_s\|_{\mathrm{L}^2(\mu_s)} \, \dif s
	\end{equation}
	and
	\begin{align} \label{eq:momentestimate2}
	\|x\|_{\mathrm{L}^2(\mu_t)} &\le \|x\|_{\mathrm{L}^2(\mu_a)} + \int_a^t \|v\|_{\mathrm{L}^2(\mu_s)} \, \dif s \\ &\le \|x\|_{\mathrm{L}^2(\mu_a)} + (t-a)\|v\|_{\mathrm{L}^2(\mu_a)}+ \int_a^t (t-s) \|F_s\|_{\mathrm{L}^2(\mu_s)} \, \dif s \fstop
	\end{align}
\end{lem}

\begin{proof}
	Let~$\psi \in \mathrm{C}_\mathrm{c}^\infty((a,b) \times \cV)$ and~$\zeta \in \mathrm{C}_\mathrm{c}^\infty(\cX)$ with~$\zeta(0)=1$. For every~$\epsilon > 0$, the definition of solution to the Vlasov equation implies
	$$
	\int_a^b \int \left( \zeta(\epsilon x) \partial_t \psi + \epsilon \psi \, v \cdot  (\nabla_x \zeta)(\epsilon x) + \zeta(\epsilon x) F_t \cdot \nabla_v \psi \right) \, \dif \mu_t \, \dif t = 0 \fstop
	$$
	Note that~$v\psi$ is compactly supported, hence bounded. The dominated convergence theorem (for $\epsilon \to 0$) yields
	$$
	\int_a^b \int \left(\partial_t \psi + F_t \cdot \nabla_v \psi \right) \, \dif \mu_t \, \dif t = 0 \fstop
	$$
	For every $t$, consider the disintegration $\mu_t = \int \mu_t^v \, \dif \, (\proj_v)_\# \mu_t$, which gives
	\begin{align}
		\label{eq:weak-cont-v}
		\int_a^b \int \left( \partial_t \psi + \left(\int F_t \, \dif  \mu_t^v \right) \cdot \nabla_v \psi \right) \, \dif \, (\proj_v)_\# \mu_t \, \dif t = 0 \fstop
	\end{align}
	By arbitrariness of~$\psi$, this argument shows that $\bigl((\proj_v)_\# \mu_t, \int F_t \, d \mu_t^v \bigr)$ satisfies the classical continuity equation. Note that the vector field satisfies \cite[(8.1.21)]{MR2401600}: by assumption
	$$
	\int_a^b \int \left| \int F_t \, \dif  \mu_t^v \right|^2 \, \dif \, (e_v)_\# \mu_t \, \dif t \le \int_a^b \int  \int |F_t|^2 \, \dif \mu_t^v \, \dif \, (\proj_v)_\# \mu_t \, \dif t < \infty \fstop
	$$
	Therefore, by \cite[Theorem 8.2.1]{MR2401600}, there exists a probability measure $\boldsymbol \eta$ such that:
	\begin{enumerate}[label=(\roman*)]
		\item $\boldsymbol \eta$ is concentrated on the set of pairs $(v,\beta)\in \cV \times \mathrm H^1(a,b;\cV)$ such that~$\dot \beta(t) = \int F_t(x,\beta(t)) \, \dif \mu_t^{\beta(t)}$ for a.e.~$t \in (a,b)$, with~$\beta(0)=v$;
		\item \label{item:propertyeta} for every~$t \in [a,b]$, $(\proj_v)_\# \mu_t$ equals the push-forward of~$\boldsymbol \eta$ via the map $(v,\beta) \mapsto \beta(t)$.
	\end{enumerate}
	For~$t \in (a,b)$, using Property~\ref{item:propertyeta}, and the Minkowski and Cauchy--Schwarz inequalities,
	\begin{align*}
		\|v\|_{\mathrm{L}^2(\mu_t)} &= \sqrt{\int \abs{v}^2 \, \dif \, (\proj_v)_\# \mu_t} = \sqrt{\int \abs{\beta(t)}^2 \, \dif \boldsymbol \eta} \\
		&\le \sqrt{\int \abs{\beta(a)}^2 \, \dif \boldsymbol \eta} + \sqrt{\int \abs{ 
			\int_a^t \abs{\dot \beta(s)} \, \dif  s }^2 \, \dif \boldsymbol \eta} \\
		&= \|v\|_{\mathrm{L}^2(\mu_a)} + \sqrt{\int \abs{
			\int_a^t \int F_s(x,\beta(s)) \, \dif \mu_s^{\beta(s)} \, \dif s }^2 \, \dif \boldsymbol \eta} \\
		&\le \|v\|_{\mathrm{L}^2(\mu_a)} + \int_a^t \sqrt{\int \abs{ 
			\int F_s(x,\beta(s)) \, \dif \mu_s^{\beta(s)} }^2 \, \dif \boldsymbol \eta} \, \dif s \\
		&\le \|v\|_{\mathrm{L}^2(\mu_a)} + \int_a^t \sqrt{\int 
			\int \abs{F_s(x,\beta(s))}^2 \, \dif \mu_s^{\beta(s)} \, \dif \boldsymbol \eta} \, \dif s \\
		&= \|v\|_{\mathrm{L}^2(\mu_a)} + \int_a^t \sqrt{ 
			\int \abs{F_s(x,v)}^2 \, \dif \mu_s } \, \dif  s \fstop
	\end{align*}
	
	Let us focus on the other inequality we need to prove. The Vlasov equation can be seen as a classical continuity equation with vector field $(v,F_t)$. It follows from the previous estimates that this vector field satisfies \cite[(8.1.21)]{MR2401600} and, by \cite[Theorem 8.2.1]{MR2401600}, there exists a probability measure $\boldsymbol \xi$ such that:
	\begin{enumerate}[label=(\roman*)]
		\item $\boldsymbol \xi$ is concentrated on the set of triples $(x,v,\gamma)\in \cX \times \cV \times \mathrm H^1(a,b;\cX \times \cV)$ such that~$\dot \gamma_x(t) = \gamma_v(t)$ and~$\dot \gamma_v(t) = F_t\bigl(\gamma_x(t),\gamma_v(t)\bigr)$ for a.e.~$t \in (a,b)$, with~$\gamma(0)=(x,v)$;
		\item for every~$t \in [a,b]$, $\mu_t$ equals the push-forward of~$\boldsymbol \xi$ via the map $(x,v,\gamma) \mapsto \gamma(t)$.
	\end{enumerate}
	Hence, for every $t \in (a,b)$, we have:
	\begin{align*}
		\|x\|_{\mathrm{L}^2(\mu_t)} &= \sqrt{\int \abs{x}^2 \, \dif \mu_t} = \sqrt{\int \abs{\gamma_x(t)}^2 \, \dif \boldsymbol \xi} \\
		&\le \sqrt{\int \abs{\gamma_x(a)}^2 \, \dif  \boldsymbol \xi} + \sqrt{\int \abs{\int_a^t \dot \gamma_x(s) \, d s}^2 \, \dif \boldsymbol \xi} \\
		&= \|x\|_{\mathrm{L}^2(\mu_a)} + \sqrt{\int \abs{\int_a^t \gamma_v(s) \, \dif s}^2 \, \dif \boldsymbol \xi} \\
		&\le \|x\|_{\mathrm{L}^2(\mu_a)} + \int_a^b \sqrt{\int |\gamma_v(s)|^2 \, \dif \boldsymbol \xi} \, \dif s \\
		&=\|x\|_{\mathrm{L}^2(\mu_a)} + \int_a^b \sqrt{\int |v|^2 \, \dif \mu_s} \, \dif s\fstop \qedhere
	\end{align*}
\end{proof}

\section{Hypoelliptic Riemannian structure}
\label{sec:sec5}
In this section, we develop a differential calculus induced by $\mathsf d$, with the main contributions organised as follows.
\begin{itemize}
    \item In \S\ref{ss:51}, we show that solutions $(\mu_t,F_t)_t$ to \eqref{eq:cont} are physical and $\sfd$-absolutely continuous, with the optimal time for $\sfd(\mu_t,\mu_{t+h})$ being asymptotically $h$, for $h \downarrow 0$.  
    \item In \S\ref{ss:52}, we prove the converse: $\sfd$-absolutely continuous curves of measures $(\mu_t)_t$ can be represented as solutions to \eqref{eq:cont}, provided the optimal time for $\sfd(\mu_t,\mu_{t+h})$ is asymptotically $h$ as $h \downarrow 0$. Similarly, we show that physical curves solve~\eqref{eq:cont}.
    \item In \S\ref{ss:53}, we show that the minimal $\Lp{2}(\mu_t)$-norm of a force field~$(F_t)_t$ such that~$(\mu_t,F_t)_t$ solves~\eqref{eq:cont} can be interpreted as a \emph{metric derivative}, namely, it is, for a.e.~$t$, the limit of~$\frac{\sfd(\mu_t,\mu_{t+h})}{h}$ and~$\frac{\tilde \sfd_h(\mu_t,\mu_{t+h})}{h}$ as~$h \downarrow 0$.
    \item In \S\ref{sec:repar}, we extend these results to reparametrisations of \eqref{eq:cont} and complete the proof of \Cref{thm:main:new:2}. 
\end{itemize}

\noindent Henceforth, we assume that~$(\mu_t)_{t \in (a,b)} \subseteq \mathcal P_2(\Gamma)$ is a narrowly continuous curve. We set
\begin{equation} \label{eq:Omega}
	\Omega \coloneqq \set{t \in (a,b) \, : \, \norm{v}_{\Lp{2}(\mu_t)} > 0 }
\end{equation}
and define the \emph{spatial density}
\begin{equation}
	\rho_t \coloneqq (\proj_x)_\# \mu_t
	\in \cP(\cX) \comma \qquad t \in (a,b) \fstop
\end{equation}
Using the disintegration theorem we write
$\dif\mu_t(x,v) = \dif\mu_{t,x}(v) \, \dif \rho_t(x)$,
where $\mu_{t,x} \in \cP(\cV)$ denotes the distribution of velocities at~$x \in \cX$, defined $\rho_t$-a.e.

For every~$t \in (a,b)$, let~$\overline V_t$ be the closure of the space~$V \coloneqq \set{ \nabla \phi \, : \, \phi \in \mathrm{C}^\infty_\mathrm{c}(\cX)}$ in~$\Lp{2}(\rho_t;\R^d)$. Additionally let~$\proj_{\overline V_t} \colon \Lp{2}(\rho_t; \R^d) \to \overline V_t$ be the corresponding projection operator, and define the \emph{flow velocity}
\begin{equation}
	j_t(x) 
		\coloneqq 
	\int_\cV v \, \dif \mu_{t,x} 
	\comma \qquad   
	(t,x) \in (a,b) \times \cX \comma
\end{equation}
and the \emph{total momentum}
\begin{equation}
	\langle v \rangle_t \coloneqq 
	\int_\Gamma v \, \dif \mu_t 
	= 
	\int_\cX j_t \, \dif \rho_t 
	\comma \qquad t \in (a,b) \fstop
\end{equation}

\begin{rem} \label{rem:constants}
	For any~$\rho \in \cP(\cX)$, the closure of~$V$ in~$\Lp{2}(\rho; \R^d)$ contains all constant vector fields. Indeed, fix~$u_0 \in \R^d$ and a~$\mathrm{C}^\infty_{\mathrm{c}}(\R^d)$ function~$\zeta$ with support contained in the unit ball, and such that~$\zeta \equiv 1$ in a neighbourhood of~$0$. For every~$\epsilon > 0$, set
	\[
		\psi_\epsilon \coloneqq \zeta(\epsilon x)\, x \cdot u_0 \comma \qquad x \in \cX \fstop
	\]
	We have
	\[
		\nabla \psi_\epsilon(x) = \zeta(\epsilon x) u_0 + \epsilon (x \cdot u_0) \nabla \zeta(\epsilon x) \in V \comma \qquad x \in \cX \fstop
	\]
	As~$\epsilon \to 0$, the dominated convergence theorem gives
	\[
		\zeta(\epsilon x) u_0 \stackrel{\Lp{2}(\rho;\R^d)}{\longrightarrow} {\zeta(0) u_0} = u_0 \comma
	\]
	as well as
	\begin{multline*}
		\int_\cX \abs{\epsilon (x \cdot u_0) \nabla \zeta(\epsilon x)}^2 \, \dif \rho 
		\le 
		\abs{u_0}^2 \int_{\set{\abs{x}<\frac{1}{\epsilon}}} \epsilon^2 \abs{x}^2 \abs{\nabla \zeta(\epsilon x)} \, \dif \rho 
		\le 
		\abs{u_0}^2 \int_\cX \abs{\nabla \zeta(\epsilon x)} \, \dif \rho 
		\\ \to \abs{u_0}^2 \abs{\nabla \zeta(0)} = 0 \fstop
	\end{multline*}
\end{rem}
Let~$(s,t) \mapsto \pi_{s,t} \in \Pi_{\mathrm o, \sfd}(\mu_s,\mu_t)$ be a measurable selection of~$\sfd$-optimal transport plans, and~$T_{s,t}$ the corresponding optimal times,\footnote{All times~$T>0$ are optimal when~$\norm{y-x}_{\Lp{2}(\pi_{s,t})} = 0$. In this case, we conventionally choose~$T_{s,t}=0$.} i.e.,
\begin{equation} \label{eq:defTbis}
	T_{s,t} = \begin{cases}
		2 \displaystyle \frac{\norm{y-x}_{\Lp{2}(\pi_{s,t})}^2}{(y-x,v+w)_{\pi_{s,t}}} &\text{if } (y-x,v+w)_{\pi_{s,t}} > 0 \comma \\
		0 &\text{if } \norm{y-x}_{\Lp{2}(\pi_{s,t})} = 0 \comma \\
		\infty &\text{otherwise.}
	\end{cases}
\end{equation}

\subsection{$\sfd$-regularity of solutions to Vlasov's equations}\label{ss:51}

The results of this subsection are given under the following.

\begin{ass}[Solution to Vlasov's equation]\label{ass:vlasov}
	The curve~$(\mu_t)_{t \in (a,b)}$ in $\cP_2(\Gamma)$ is a distributional solution to Vlasov's equation~\eqref{eq:cont} for a field~$(F_t)_{t \in (a,b)}$ such that
	\begin{equation}
		\int_a^b \left(\norm{v}_{\Lp{2}(\mu_t)}^2 + \norm{F_t}_{\Lp{2}(\mu_t)}^2 \right) \, \dif t
		<
		\infty \fstop
	\end{equation}
\end{ass}

Under this assumption, the curve~$t \mapsto \mu_t$ is ${\mathrm{W}_2}$-$2$-absolutely continuous by~\cite[Theorem~8.3.1]{MR2401600}. 
It is readily shown that the maps~$t \mapsto \norm{x}_{\Lp{2}(\mu_t)}$ and~$t \mapsto \norm{v}_{\Lp{2}(\mu_t)}$ are continuous. Indeed, for any $\mathrm{W}_2$-optimal plan $\pi_{s,t} \in \Pi(\mu_s,\mu_t)$ we have
\begin{align*}
	\big| \norm{v}_{\Lp{2}(\mu_t)}^2 - \norm{v}_{\Lp{2}(\mu_s)}^2 \big|
	= \bigg| \int_{\Gamma \times \Gamma} (v + w)\cdot(v-w) \dif \pi_{s,t}  \bigg|
	\leq \mathrm{W}_2(\mu_s, \mu_t) \Bigl( \norm{v}_{\Lp{2}(\mu_t)} + \norm{v}_{\Lp{2}(\mu_s)}   \Bigr).
\end{align*} 
Since $t\mapsto \norm{x}_{\Lp{2}(\mu_t)}^2 +  \norm{v}_{\Lp{2}(\mu_t)}^2 =  \mathrm{W}_2^2(\mu_t, \delta_{(0,0)})$ is continuous and thus locally bounded, the continuity of $t \mapsto \norm{v}_{\Lp{2}(\mu_t)}$ follows. In particular, the set~$\Omega$ is open in~$(a,b)$. The continuity of $t \mapsto \norm{x}_{\Lp{2}(\mu_t)}$ is proved analogously.

\begin{lem} \label{lem:derW2}
	Under \Cref{ass:vlasov}, the space-marginal curve~$t \mapsto \rho_t$ is~${\mathrm{W}_2}$-$2$-a.c.~with
	\begin{equation} \label{eq:derW2}
		\abs{\langle v \rangle_t}
            \le
		\abs{\rho_t'}_{{\mathrm{W}_2}}
		=
		\norm{\proj_{\overline V_t}(j_{t})}_{\Lp{2}(\rho_t)}
			\le
		\norm{v}_{\Lp{2}(\mu_t)} \quad \text{for a.e.~} t \in (a,b) \fstop
	\end{equation}
\end{lem}

\begin{proof}
	Fix~$\psi \in \mathrm{C}^\infty_\mathrm{c}\bigl((a,b) \times \cX \bigr)$. With the same argument as in the proof of \Cref{lem:momentestimate}:
	\[
		\int_a^b \int_\Gamma (\partial_t \psi + v \cdot \nabla_x \psi) \, \dif \mu_t \, \dif t = 0 \comma
	\]
	from which we get
	\begin{equation} \label{eq:continuityrho}
		0
		=
		\int_a^b \int_\cX \left( \partial_t \psi + j_t(x) \cdot \nabla_x \psi  \right) \, \dif \rho_t \, \dif t 
		= 
		\int_a^b \int_\cX \left( \partial_t \psi + \proj_{\overline V_t} (j_t) \cdot \nabla_x \psi  \right) \, \dif \rho_t \, \dif t \comma
	\end{equation}
	where we used that~$\nabla_x \psi(t,\cdot) \in \overline V_t$ in the last equality. Since $\psi$ is arbitrary, we deduce that~$t \mapsto \rho_t$ is a solution to the continuity equation with vector field~$\bigl(\proj_{\overline V_t}(j_t)\bigr)_{t \in (a,b)}$. 
	The identity~$\abs{\rho_t'}_{\mathrm{W}_2} = \norm{\proj_{\overline V_t}(j_t)}_{\Lp{2}(\rho_t)}$ thus follows from~\cite[Proposition~8.4.5]{MR2401600}.
	
	The inequality~$\norm{\proj_{\overline V_t}(j_t)}_{\Lp{2}(\rho_t)} \le \norm{v}_{\Lp{2}(\mu_t)}$ follows from the definition of~$j_t$ using Jensen's equality.
        Finally, by definition of~$\langle v \rangle_t$ and \Cref{rem:constants}, we write
        \[
	\langle v \rangle_t^2 = \int_\cX j_t \cdot \langle v \rangle_t \, \dif \rho_t
	=
	\int_\cX \proj_{\overline V_t}(j_t) \cdot \langle v \rangle_t \, \dif \rho_t
	\le
	\abs{\langle v \rangle_t} \norm{\proj_{\overline V_t}(j_t)}_{\Lp{2}(\rho_t)} \comma
	\]
    which yields~$\abs{\langle v \rangle_t} \le \norm{\proj_{\overline V_t}(j_t)}_{\Lp{2}(\rho_t)}$.
\end{proof}

Our goal is to prove the following three propositions.

\begin{prop} \label{prop:vlasovtoac1}
	Under \Cref{ass:vlasov}, for every~$s,t$ with~$a<s<t<b$, we have
	\begin{equation} \label{eq:vlasovtoac11}
		\sfd(\mu_s,\mu_t) \le \tilde \sfd_{t-s}(\mu_s,\mu_t) \le 2 \int_{s}^t \norm{F_r}_{\Lp{2}(\mu_r)} \, \dif r \fstop
	\end{equation}
\end{prop}

\begin{prop} \label{prop:vlasovtoac2}
	Under \Cref{ass:vlasov}, for almost every~$t \in (a,b)$, we have
	\begin{equation} \label{eq:vlasovtoac2}
		\limsup_{h \downarrow 0} \frac{\sfd(\mu_t,\mu_{t+h})}{h}
		\le
		\limsup_{h \downarrow 0} \frac{\tilde \sfd_h(\mu_t,\mu_{t+h})}{h}
		\le
		\norm{F_{t}}_{\Lp{2}(\mu_{t})} \fstop
	\end{equation}
\end{prop}

\begin{prop} \label{prop:vlasovtoac3}
	Under \Cref{ass:vlasov}, the following assertions hold.
	\begin{enumerate}
		\item For a.e.~$t \in (a,b)$ such that~$\abs{\rho_t'}_{\mathrm{W_2}} > 0$, we have
	\begin{equation}
		\lim_{h \downarrow 0} \frac{T_{t,t+h}}{ h} = 1 \quad \text{for a.e.~} t \in (a,b) \fstop
	\end{equation}
		\item For every~$[a', b'] \subseteq \Omega$, there exist~$\bar h > 0$ and~$g \in \Lp{2}(a',b')$ such that
	\begin{equation} \label{eq:L2dom}
		\sup_{h \in (0,\bar h)}\frac{T_{t,t+ h}}{ h} \le g(t) \quad \text{for all } t \in [a',  b'] \fstop
	\end{equation}
	\end{enumerate}
	
	Consequently, if~$\abs{\rho_t'}_{\mathrm{W}_2} > 0$ for a.e.~$t \in \Omega$, we have~$\frac{T_{\cdot,\cdot+h}}{h} \to 1$ in~$\Lp{2}_\mathrm{loc}(\Omega)$ as~$h \downarrow 0$.
\end{prop}

\Cref{prop:vlasovtoac1} and \Cref{prop:vlasovtoac2} provide upper bounds for the kinetic discrepancies between successive states along~$(\mu_t)_t$. The first one applies to any two times~$s,t$ with~$s<t$, while the second one concerns the infinitesimal change, i.e.,~it provides an upper bound on the~$\sfd$-{derivative}. \Cref{prop:vlasovtoac3} shows that the optimal time for~$\sfd$ between two successive nearby states along a solution to Vlasov's equation is comparable to the \emph{physical} time. In \Cref{prop:tangent2} below, the convergence~$\frac{T_{t,t+h}}{h} \to 1$ will be improved to~$\frac{T_{t,t+h}-h}{h^2} \to 0$ at a.e.~times~$t \in (a,b)$ where~$\langle v \rangle_t \neq 0$.

\begin{rem}
	Comparing \Cref{prop:vlasovtoac1} and \Cref{prop:vlasovtoac2}, we see the presence of an extra factor~$2$ in the former. Note that a version of \Cref{prop:vlasovtoac2} with the extra factor~$2$ follows immediately from \Cref{prop:vlasovtoac1}. Also notice that, if~$\sfd$ were a distance, these two propositions, together, would allow dropping the constant~$2$ in~\eqref{eq:vlasovtoac11}, see \cite[Theorem~1.1.2]{MR2401600}. However, this factor is \emph{sharp}, as demonstrated by the following example.
\end{rem}

\begin{ex}
	Let $\cX = \cV = \R$.
	For~$\epsilon \in (0,1)$ we define
	\[
		\alpha(t) \coloneqq \begin{cases}
			\epsilon t^2 \comma &\text{for } t \in \bigl[0,1\bigr ] \comma \\
			-\epsilon +2\epsilon t - (t-1)^2 &\text{for } t \in \bigl [1,1+\sqrt{\epsilon} \bigr] \comma
		\end{cases}
	\]
	and, by means of~$\alpha$,
	\[
		\mu_t \coloneqq \delta_{\bigl(\alpha(t),\alpha'(t)\bigr)} \comma \qquad t \in \bigl[0,1+\sqrt{\epsilon}\bigl] \fstop
	\]
	This curve solves Vlasov's equation with~$F_t(x,v) \coloneqq \alpha''(t)$. In particular,
	\[
		\int_0^{1+\sqrt{\epsilon}} \norm{F_t}_{\Lp{2}(\mu_t)} \, \dif t = 2\epsilon + 2 \sqrt{\epsilon} \fstop
	\]
	On the other hand, recalling the definition~\eqref{eq:st2.} of~$d$,
	\begin{align*}
		&d^2\bigl((\alpha(0),\alpha'(0)),(\alpha(1+\sqrt{\epsilon}),\alpha'(1+\sqrt{\epsilon})\bigr) = d^2\bigl((0,0), (2\epsilon \sqrt{\epsilon}, 2\epsilon-2\sqrt{\epsilon})\bigr) \\
		&\quad = 3 \abs{2\epsilon-2\sqrt{\epsilon}}^2 - 3\left(\frac{2\epsilon \sqrt{\epsilon}}{\abs{2\epsilon \sqrt{\epsilon}}}\cdot (2\epsilon-2\sqrt{\epsilon})\right)_+^2 + \abs{2\epsilon-2\sqrt{\epsilon}}^2 \\
		&\quad = 4 \abs{2\epsilon-2\sqrt{\epsilon}}^2 - 3(2\epsilon-2\sqrt{\epsilon})_+^2 = 4 \abs{2\epsilon-2\sqrt{\epsilon}}^2 \comma
	\end{align*}
	where the last equality is true because~$\epsilon < 1$. Hence,
	\[
		\frac{d\bigl((\alpha(0),\alpha'(0)),(\alpha(1+\sqrt{\epsilon}),\alpha'(1+\sqrt{\epsilon})\bigr)}{\int_0^{1+\sqrt{\epsilon}} \norm{F_t}_{\Lp{2}(\mu_t)} \, \dif t} = 2 \frac{\abs{2\epsilon - 2\sqrt{\epsilon}}}{2\epsilon + 2\sqrt{\epsilon}} \comma
	\]
	and the latter tends to~$2$ as~$\epsilon \to 0$.
\end{ex}

\begin{proof}[Proof of~\Cref{prop:vlasovtoac1}]
	Let us fix~$s,t \in (a,b)$ with~$s<t$. By~\cite[Theorem~8.2.1]{MR2401600}, there exists a measure~$\boldsymbol{\eta} \in \cP\bigl(\Gamma \times \mathrm{H}^1(s,t;\Gamma)\bigr)$ supported on tuples~$(x,v,\gamma_x,\gamma_v)$ such that:
	\begin{enumerate}
		\item $\gamma_x(s)=x$ and~$\gamma_v(s) = v$;
		\item $\dot \gamma_x(r) = \gamma_v(r)$ and~$\dot \gamma_v(r) = F_r\bigl(\gamma_x(r),\gamma_v(r)\bigr)$ for a.e.~$r \in (s,t)$;
		\item $\left(\proj_{\gamma(r)}\right)_\# \boldsymbol \eta = \mu_{r}$ for every~$r \in (s,t)$.
	\end{enumerate}
	By definition of~$\tilde \sfd_{t-s}$ and by the properties of~$\boldsymbol \eta$, we write
	\begin{align*}
		&\tilde \sfd_{t-s}^2(\mu_s,\mu_t)
			\le
		\tilde \sfc_{t-s}\left(\left(\proj_{\gamma(s),\gamma(t)}\right)_\# \boldsymbol{\eta} \right) \\
		&\quad=
		\int \left( 12\abs{\frac{\gamma_x(t)-\gamma_x(s)}{t-s}-\frac{\gamma_v(t)+\gamma_v(s)}{2}}^2 + \abs{\gamma_v(t)-\gamma_v(s)}^2 \right) \, \dif \boldsymbol \eta \\
		&\quad=
		3\int \abs{\int_s^t \frac{t+s-2r}{t-s} F_r\bigl(\gamma_x(r),\gamma_v(r)\bigr) \, \dif r}^2 \, \dif \boldsymbol \eta 
		+ \int \abs{\int_s^t F_r\bigl(\gamma_x(r),\gamma_v(r)\bigr) \, \dif r }^2  \, \dif \boldsymbol \eta \comma
	\end{align*}
	which yields, by Minkowski's integral inequality,
	\begin{align} \label{eq:estdtilde}
		\tilde \sfd_{t-s}^2(\mu_s,\mu_t)
			&\le
		3 \left(\int_s^t \frac{\abs{t+s-2r}}{t-s} \sqrt{\int \abs{F_r\bigl(\gamma_x(r),\gamma_v(r)\bigr)}^2 \, \dif \boldsymbol \eta } \, \dif r \right)^2 \nonumber \\
		&\quad 
		+ \left(\int_s^t \sqrt{\int \abs{F_r\bigl(\gamma_x(r),\gamma_v(r)\bigr)}^2 \, \dif \boldsymbol \eta } \, \dif r \right)^2 \nonumber \\
		&=
		3 \left(\int_s^t \frac{\abs{t+s-2r}}{t-s} \norm{F_r}_{\Lp{2}(\mu_r)} \, \dif r \right)^2 + \left(\int_s^t \norm{F_r}_{\Lp{2}(\mu_r)} \, \dif r \right)^2 \fstop
	\end{align}
	The conclusion follows by estimating~$\frac{\abs{t+s-2r}}{t-s} \le 1$.
\end{proof}

\begin{proof}[Proof of \Cref{prop:vlasovtoac2}]
	Let~$t \in (a,b)$ be a Lebesgue point for~$\tilde t \mapsto \norm{F_{\tilde t}}_{\Lp{2}(\mu_{\tilde t})}$. By the kinetic Benamou--Brenier formula of \Cref{thmbb} we have, for every~$h > 0$,
    \[
        \frac{\tilde \sfd_h^2(\mu_t,\mu_{t+h})}{h^2}
            \le
        \frac{\widetilde{\mathcal{MA}}_{h}^2(\mu_t,\mu_{t+h})}{h^2} \le \fint_t^{t+h} \norm{F_s}_{\Lp{2}(\mu_s)}^2 \, \dif s \fstop
    \]
    We conclude by letting~$h \downarrow 0$.
\end{proof}

\subsubsection*{Proof of \Cref{prop:vlasovtoac3}}

The core idea in the proof of \Cref{prop:vlasovtoac3} is to equate two interpretations of~$\mu_t$ and~$\mu_{t+h}$, as marginals of two different plans in $\Pi(\mu_t,\mu_{t+h})$. One is the~$\sfd$-optimal plan~$\pi_{t,t+h}$, i.e.,~an evolution along a~$T_{t,t+h}$-long curve; while the other is the dynamical transport plan induced by Vlasov's equation, hence an evolution taking time~$h$. 
One of the lemmas we prove after this idea---namely, \Cref{lem:general}---will also be used to compute the $\sfd$-derivative, see~\Cref{prop:metricder}.

Another key passage in the proof below is the derivation of the local~$\Lp{2}$-domination~\eqref{eq:L2dom} by means of the upper bound~\eqref{eq:vlasovtoac11}.

\begin{lem} \label{lem:boundT}
	Assume \Cref{ass:vlasov}. Fix~$[a',b'] \subseteq \Omega$. Then, there exist~$\bar h > 0$ and a function~$g \in \Lp{2}(a',b')$ such that
	\begin{equation} \label{eq:boundT1}
		\frac{T_{t,t+h}}{h} \le g(t) \quad \text{for all } t \in [a',b'] \text{ and every } h \in (0,\bar h) \fstop
	\end{equation}
	In particular, for a.e.~$t \in \Omega$ (hence, for a.e.~$t$ such that~$\abs{\rho_t'}_{{\mathrm{W}_2}}>0$), we have
	\begin{equation} \label{eq:boundT2}
		\limsup_{h \downarrow 0} \frac{T_{t,t+h}}{h} < \infty \fstop
	\end{equation}
\end{lem}

\begin{proof}
	Recall that the functions~$\tilde t \mapsto \norm{x}_{\Lp{2}(\mu_{\tilde t})}$ and~$\tilde t \mapsto \norm{v}_{\Lp{2}(\mu_{\tilde t})}$ are continuous on~$(a,b)$. Let~$c > 0$ be an upper bound on the restriction of these functions to~$\left[a',\frac{b'+b}{2}\right]$, and~$\epsilon > 0$ be the minimum of~$\tilde t \mapsto \norm{v}_{\Lp{2}(\mu_{\tilde t})}$ on~$[a',b']$. 
	By \Cref{ass:vlasov} and \Cref{prop:vlasovtoac1}, we can find~$\bar h \in \left(0, \frac{b-b'}{2}\right)$ depending on $\epsilon,c$ and 
	$\int_a^b \norm{F_t}_{\Lp{2}(\mu_t)}^2 \, \dif t$, 
	such that
	\begin{equation} \label{eq:barh}
		h \in (0,\bar h)
			\quad \Longrightarrow \quad 
		\norm{w-v}_{\Lp{2}(\pi_{t,t+h})}\le \sfd(\mu_t,\mu_{t+h})
			\le
		\tilde \sfd_h(\mu_t,\mu_{t+h})
			\le
		\frac{\epsilon^2}{6  c} \fstop
	\end{equation}
	Fix~$t \in [a',b']$ and~$h \in (0,\bar h)$. Let~$\boldsymbol \eta \in \cP\bigl(\Gamma \times \mathrm{H}^1(t,t+h;\Gamma)\bigr)$ be as in the proof of \Cref{prop:vlasovtoac1} (after replacing~$(s,t)$ with~$(t,t+h)$) and let~$\mathbf P$ be a probability measure on~$\G \times \G \times \mathrm{H}^1(t,t+h;\G)$ constructed by gluing $\pi_{t,t+h}$ and~$\boldsymbol \eta$ at~$\mu_{t}$, i.e.,~such that
	\[
	\left(\proj_{x,v,y,w}\right)_\# \mathbf P = \pi_{t,t+h}
	\quad \text{and} \quad
	\left(\proj_{x,v,\gamma}\right)_\# \mathbf P = \boldsymbol \eta \fstop
	\]
	Using the properties of~$\boldsymbol \eta$ and~$\mathbf P$, we write
	\begin{align*}
		&\int y \cdot w \, \dif \mathbf P - \int x \cdot v \, \dif \mathbf P\\
		&\quad=
		\int \gamma_x(t+h) \cdot \gamma_v(t+h) \, \dif \mathbf P - \int \gamma_x(t) \cdot \gamma_v(t) \, \dif \mathbf P \\
		&\quad=
		\int \int_{t}^{t+h} \left(\dot \gamma_x(s) \cdot \gamma_v(t+h) + \dot \gamma_v(s) \cdot \gamma_x(t) \right) \, \dif s \, \dif \mathbf P \\
		&\quad=
		\int \int_t^{t+h} \left(\gamma_v(s) \cdot \gamma_v(t+h)  +  F_s\bigl(\gamma_x(s),\gamma_v(s) \bigr) \cdot \gamma_x(t) \right) \, \dif s \, \dif \mathbf P \fstop
	\end{align*}
	Rearranging terms, this identity writes, for every~$\theta > 0$, as
	\begin{multline*}
		\int \left( \theta \abs{v}^2 + \theta (w-v) \cdot v + (y-x-\theta v) \cdot w \right) \, \dif \mathbf P \\
		=
		\int x \cdot (v-w) \, \dif \mathbf P + \int \int_t^{t+h} \left(\gamma_v(s) \cdot \gamma_v(t+h)  +  F_s\bigl(\gamma_x(s),\gamma_v(s) \bigr) \cdot \gamma_x(t) \right) \, \dif s \, \dif \mathbf P \fstop
	\end{multline*}
	By the triangle and Cauchy--Schwarz inequalities, we obtain
	\begin{multline*}
		\theta \cdot \left(\norm{v}_{\Lp{2}(\mu_{t})}^2 - \norm{v}_{\Lp{2}(\mu_t)}\norm{w-v}_{\Lp{2}(\pi_{t,t+h})} - \norm{v}_{\Lp{2}(\mu_{t+h})} \norm{\frac{y-x}{\theta}-v}_{\Lp{2}(\pi_{t,t+h})} \right) \\
		\le
		\norm{x}_{\Lp{2}(\mu_t)} \norm{w-v}_{\Lp{2}(\pi_{t,t+h})} 
		+
		\norm{v}_{\Lp{2}(\mu_{t+h})} \int_t^{t+h} \norm{v}_{\Lp{2}(\mu_s)} \, \dif s
		+
		\norm{x}_{\Lp{2}(\mu_t)} \int_t^{t+h} \norm{F_s}_{\Lp{2}(\mu_s)} \, \dif s \semicolon
	\end{multline*}
	hence, 
	since~$\norm{v}_{\Lp{2}(\mu_t)} > \epsilon$
	and $\max\set{\sup_{\tilde t} \norm{x}_{\Lp{2}(\mu_{\tilde t})}, \sup_{\tilde t} \norm{v}_{\Lp{2}(\mu_{\tilde t})}} \leq c$, we have
	\begin{multline} \label{eq:theta}
		\theta \cdot \left( \epsilon^2 - c \, \norm{w-v}_{\Lp{2}(\pi_{t,t+h})} - c \, \norm{\frac{y-x}{\theta}-v}_{\Lp{2}(\pi_{t,t+h})} \right) \\
		\le
		c \, \norm{w-v}_{\Lp{2}(\pi_{t,t+h})} + c^2 h + c \int_{t}^{t+h} \norm{F_s} _{\Lp{2}(\mu_s)} \, \dif s \fstop
	\end{multline}
	It remains to bound the term $ \norm{\frac{y-x}{\theta}-v}_{\Lp{2}(\pi_{t,t+h})}$ for a suitable choice of $\theta$.
	
	If~$T_{t,t+h} \in (0,\infty)$, we choose~$\theta \coloneqq T_{t,t+h}$.
	Using the triangle inequality, the definition of $\tilde \sfc$, the fact that~$\norm{y-x}_{\Lp{2}(\pi_{t,t+h})} > 0$, and the optimality of $\pi_{t,t+h}$, we obtain
	\begin{multline} \label{eq:boundtild}
		\norm{\frac{y-x}{T_{t,t+h}}-v}_{\Lp{2}(\pi_{t,t+h})} + \norm{w-v}_{\Lp{2}(\pi_{t,t+h})}\le  \norm{\frac{y-x}{T_{t,t+h}}-\frac{v+w}{2}}_{\Lp{2}(\pi_{t,t+h})} + \frac{3}{2} \norm{w-v}_{\Lp{2}(\pi_{t,t+h})} \\
		\le 3\sqrt{\tilde \sfc(\pi_{t,t+h})} =  3 \sqrt{\sfc(\pi_{t,t+h})} = 3\, \sfd(\mu_t,\mu_{t+h}) \stackrel{\eqref{eq:barh}}{\le} \frac{\epsilon^2}{ 2  c} \fstop
	\end{multline}
	We arrive to the same conclusion in the case where~$T_{t,t+h} = \infty$ by letting~$\theta \to \infty$. 
	In all cases (trivially when~$T_{t,t+h} = 0$), we get the inequality
	\[
	\frac{\epsilon^2 \, T_{t,t+h}}{2}
	\le
	c \,  \sfd(\mu_t,\mu_{t+h}) + c^2 h + c \int_t^{t+h} \norm{F_s}_{\Lp{2}(\mu_s)} \, \dif s \comma
	\]
	proving in particular that~$T_{t,t+h} < \infty$.
	Bounding~$ \sfd(\mu_t,\mu_{t+h})$ with \Cref{prop:vlasovtoac1}, we find
	\[
	\frac{T_{t,t+h}}{h} \le \frac{1}{\epsilon^2} \fint_t^{t+h} \left( 6 c \norm{F_s}_{\Lp{2}(\mu_s)}+2 c^2 \right) \, \dif s \le \underbrace{\frac{1}{\epsilon^2} \sup_{\tilde h \in (0,\bar h)} \fint_t^{t+\tilde h} \left( 6 c \norm{F_s}_{\Lp{2}(\mu_s)}+2 c^2 \right) \, \dif s}_{\eqqcolon g(t)} \fstop
	\]
	Since~$\tilde t \mapsto \norm{F_{\tilde t}}_{\Lp{2}(\tilde t)}$ is~$\Lp{2}$, so is~$g$ by the strong Hardy--Littlewood maximal inequality.
\end{proof}

\begin{lem} \label{lem:general}
	Fix~$\varphi \in \mathrm{C}^{1,1}_\mathrm{b}(\Gamma)$, i.e.,~$\varphi$ is bounded and continuously differentiable, with bounded and Lipschitz gradient. Under \Cref{ass:vlasov}, for a.e.\footnote{We allow the negligible set of times where~\eqref{eq:general} does not hold to possibly depend on~$\varphi$.}~$t \in \Omega$, we have
	\begin{equation} \label{eq:general}
		\lim_{h \downarrow 0} \left( \int_{\Gamma \times \Gamma} \frac{w-v}{h} \cdot \nabla_v \varphi \, \dif \pi_{t,t+h} + \frac{T_{t,t+h}}{h} \int_\Gamma v \cdot \nabla_x \varphi \, \dif \mu_t \right) = \int_\Gamma (v,F_t) \cdot \nabla_{x,v} \varphi \, \dif \mu_t \fstop
	\end{equation}
\end{lem}

\begin{proof}
	Fix~$t \in \Omega$ satisfying~\eqref{eq:boundT2}. We also assume that~$t$ is a Lebesgue point of
	\begin{equation} \label{eq:lebPoint}
		s \longmapsto \int_{\Gamma} 
			\bigl( \tilde v, F_s \bigr) 
			\cdot 
			\nabla_{x,v} \varphi
			\, \dif \mu_{s}
			 \quad \text{and} \quad s \longmapsto \norm{F_s}_{\Lp{2}(\mu_s)} \fstop
	\end{equation}
	Fix~$h \in (0,b-t)$ such that~$T_{t,t+h}<\infty$, and let~$\boldsymbol \eta, \mathbf P$ be as in \Cref{lem:boundT}. In particular
	\[
	\int \varphi(y,w) \, \dif \mathbf P - \int \varphi(x,v) \, \dif \mathbf P
	=
	\int \varphi \bigl(\gamma(t+h)\bigr) \, \dif \mathbf P  - \int \varphi \bigl(\gamma(t)\bigr) \, \dif \mathbf P \fstop
	\]
	By the fundamental theorem of calculus and the properties of~$\boldsymbol \eta$ and~$\mathbf{P}$, we deduce that
	\begin{multline} \label{eq:general1}
		\frac{1}{h}\int_{\Gamma \times \Gamma} \int_0^1  (y-x,w-v) \cdot \nabla_{x,v} \varphi\bigl(x+r(y-x),v+r(w-v)\bigr)  \, \dif r \, \dif \pi_{t,t+h} \\
			=
		\fint_t^{t+h} \int \dot \gamma(s) \nabla_{x,v} \varphi\bigl(\gamma(s)\bigr) \, \dif \boldsymbol \eta \, \dif s
			=
		\fint_t^{t+h} \int_\Gamma \bigl(\tilde v,F_s(\tilde x, \tilde v)\bigr) \cdot \nabla_{x,v} \varphi(\tilde x, \tilde v) \, \dif \mu_s \, \dif s \fstop
	\end{multline}
	The right-hand side in the latter equality converges to the right-hand side of~\eqref{eq:general} as~$h \downarrow 0$, since~$t$ is a Lebesgue point for~\eqref{eq:lebPoint}.
	
	Focusing on the left-hand side of~\eqref{eq:general1}, we observe that
	\begin{align} \label{eq:negl}
		&\abs{\frac{1}{h}\int_{\Gamma^2} \int_0^1  (y-x,w-v) \cdot \left(\nabla_{x,v} \varphi\bigl(x+r(y-x),v+r(w-v)\bigr)-\nabla_{x,v} \varphi(x,v) \right) \, \dif r \, \dif \pi_{t,t+h}} \nonumber \\
			&\quad \lesssim \frac{\norm{y-x}_{\Lp{2}(\pi_{t,t+h})}^2 + \norm{w-v}_{\Lp{2}(\pi_{t,t+h})}^2}{h} \nonumber \\
			&\quad \lesssim \frac{1}{h} \left(
		\norm{y-x-T_{t,t+h}\frac{w+v}{2}}_{\Lp{2}(\pi_{t,t+h})}^2 + \bigl(1+T_{t,t+h}^2\bigr) \norm{w-v}_{\Lp{2}(\pi_{t,t+h})}^2 + T_{t,t+h}^2 \norm{v}_{\Lp{2}(\mu_t)}^2 \right) \nonumber \\
			&\quad \lesssim
		\frac{\bigl(1+T_{t,t+h}^2\bigr) \tilde \sfd(\mu_t,\mu_{t+h})^2 + T_{t,t+h}^2 \norm{v}_{\Lp{2}(\mu_t)}^2}{h} \comma
	\end{align}
	where the constants hidden in~$\lesssim$ do not depend on~$h$. Combining the latter with~\Cref{prop:vlasovtoac1},~\eqref{eq:boundT2}, and the fact that~$t$ is a Lebesgue point for~$s \mapsto \norm{F_s}_{\Lp{2}(\mu_s)}$, we infer that~\eqref{eq:negl} is~$o(1)$ for~$h \downarrow 0$. Finally, we notice that
	\[
		\abs{\int_\Gamma \frac{y-x - T_{t,t+h}v}{h} \cdot \nabla_x \varphi \, \dif \mu_t} \lesssim \frac{\norm{y-x-T_{t,t+h}v}_{\Lp{2}(\pi_{t,t+h})}}{h} \stackrel{(*)}{\lesssim} \frac{T_{t,t+h}}{h}  \tilde \sfd(\mu_t,\mu_{t+h}) = o(1) \comma
	\]
	where~$(*)$ can be proved as in~\eqref{eq:boundtild} if~$T_{t,t+h}> 0$, and is trivial otherwise. From~\eqref{eq:general1} and these observations, the conclusion follows.
\end{proof}

\begin{cor} \label{cor:limT}
	Under \Cref{ass:vlasov}, for a.e.~$t \in (a,b)$ such that~$\abs{\rho_t'}_{{\mathrm{W}_2}} > 0$, we have
	\begin{equation} \label{eq:limT}
		\lim_{h \downarrow 0} \frac{T_{t,t+h}}{h}
		=
		1 \fstop
	\end{equation}
\end{cor}

\begin{proof}
	Let~$\set{\phi_k}_{k \in \N}$ be a~$\mathrm{C}^1$-dense set of~$\mathrm{C}_\mathrm{c}^\infty(\mathcal X)$. We apply \Cref{lem:general} with~$\varphi \colon (x,v) \mapsto \phi_k(x)$ for every~$k \in \N$ to deduce that, for a.e.~$t \in (a,b)$ such that~$\abs{\rho_t'}_{{\mathrm{W}_2}} > 0$, we have
	\begin{equation} \label{eq:convThh}
		\lim_{h \downarrow 0} \frac{T_{t,t+h}}{h} \int v \cdot \nabla_x \phi_k \, \dif \mu_t = \int v \cdot \nabla_x \phi_k \, \dif \mu_t \comma \qquad k \in \N \fstop
	\end{equation}
	Let us take any such~$t$ for which, additionally,~\eqref{eq:derW2} holds. By \Cref{lem:derW2}, there exists~$\bar \phi \in \mathrm{C}^\infty_\mathrm{c}(\cX)$ such that~$\nabla \bar \phi$ is sufficiently close to~$\proj_{\overline V_t} (j_{ t})$ in~$\Lp{2}( \rho_t ; \R^d)$, in the sense that
	\[
	\abs{\int \nabla \bar \phi(x) \cdot v \, \dif \mu_{t}} = \abs{\int \nabla \bar \phi \cdot \proj_{\overline V_t}(j_{t})  \, \dif \rho_t } >  \frac{\abs{\rho_t'}_{{\mathrm{W}_2}}}{2} > 0 \fstop
	\]
	To conclude, it suffices to choose~$k$ in~\eqref{eq:convThh} such that~$\norm{\bar \phi - \phi_k}_{\mathrm{C}^1}$ is sufficiently small.
\end{proof}

\begin{proof}[Proof of \Cref{prop:vlasovtoac3}]
	The result is immediate after \Cref{lem:boundT} and \Cref{cor:limT}.
\end{proof}

\subsection{Physical curves solve Vlasov's equations}\label{ss:52}

Let $(\mu_t)_{t \in (a,b)}$ be a narrowly continuous curve in $\mathcal{P}_2(\Gamma)$, let~$(s,t) \mapsto \pi_{s,t}$ be a measurable selection of optimal transport plans for~$\sfd$, and let~$T_{s,t}$ be the corresponding optimal times, see~\eqref{eq:defTbis}. Recall~$
	\rho_t \coloneqq (\proj_x)_\# \mu_t$
and~$\Omega \coloneqq \set{t \in (a,b) \, : \, \norm{v}_{\Lp{2}(\mu_t)} > 0 }$.

Choose measurable functions~$(s,t) \mapsto \mathfrak \theta_{s,t} \in [0,\infty]$ and~$(s,t) \mapsto \pi_{s,t}^\theta \in \Pi(\mu_s,\mu_t)$. Let~$\widehat \Omega \subseteq (a,b)$ be an \emph{open} set of times.

\begin{prop} \label{prop:actovlasov_new}
    Assume the following:
	\begin{enumerate}[label=(\alph*)]
		\item \label{ass:cont_new} The curve~$(\mu_t)_{t \in (a,b)}$ is~${\mathrm{W}_2}$-$2$-absolutely continuous. (Consequently, $t \mapsto \rho_t$ is~$\mathrm{W}_2$-2-a.c.,~and~$t \mapsto \norm{v}_{\Lp{2}(\mu_t)}$ is~$2$-a.c.)
		\item \label{ass:theta} For every~$s<t$ such that~$\theta_{s,t} = 0$, we have~$y=x$ $\pi_{s,t}^\theta$-a.e.
		\item \label{ass:ac_new} There exists a function~$\ell \in \Lp{2}(a,b)$ such that
		\begin{equation} \label{eq:metderupper1}
		\norm{w-v}_{\Lp{2}(\pi_{s,t}^\theta)} \le \int_s^t \ell(r) \, \dif r \comma \qquad a < s < t < b \fstop
		\end{equation}
		\item \label{ass:conv_new} The set~$\widehat \Omega$ has full measure in~$(a,b)$. For every~$[a',b'] \subseteq \widehat \Omega$, we have the limits
		\begin{equation} \label{eq:L1conv}
			\lim_{h \downarrow 0} \int_{a'}^{b'} \abs{\frac{\theta_{t,t+h}}{h} - 1} \, \dif t
			=
			0
		\end{equation}
		and
		\begin{equation} \label{eq:bounddtilde}
			\lim_{h \downarrow 0} \int_{\set{t \in (a',b') \, : \, \theta_{t,t+h} \in (0,\infty)}} \frac{\theta_{t,t+h}}{h} \sqrt{\tilde \sfc_{\theta_{t,t+h}}(\pi_{t,t+h}^\theta)} \, \dif t = 0 \fstop
		\end{equation}
	\end{enumerate}
	
	Then, there exists a force field $(F_t)_t$ such that $(\mu_t,F_t)_t$ solves Vlasov's equation~\eqref{eq:cont}, and~$(F_t)_t$ belongs to the~$\Lp{2}(\mu_t \dif t)$-closure of~$\set{\nabla_v \varphi \, : \, \varphi \in \mathrm{C}^\infty_\mathrm{c}\bigl((a,b) \times \G \bigr)}$.
\end{prop}

We discuss the assumptions of \Cref{prop:actovlasov_new} below, and give a proof at the end of this section. Choosing~$\theta_{s,t} \coloneqq t-s$ and~$\theta_{s,t} \coloneqq T_{s,t}$, we immediately obtain two corollaries.

\begin{cor} \label{cor:actovlasov_physical}
	Assume~\ref{ass:cont_new} in \Cref{prop:actovlasov_new} holds. If there exists~$\ell \in \Lp{2}(a,b)$ with
	\begin{equation} \label{eq:metderupper2}
		\tilde \sfd_{t-s}(\mu_s,\mu_t) \le \int_s^t \ell(r) \, \dif r \comma \qquad a<s<t<b \comma
	\end{equation}
	then the conclusion of \Cref{prop:actovlasov_new} holds.
\end{cor}

\begin{proof}
	Set~$\theta_{s,t} \coloneqq t-s$, and let~$(s,t) \mapsto \pi_{s,t}^{\theta} \in \Pi(\mu_s,\mu_t)$ be a measurable selection of~$\tilde \sfd_{t-s}$-optimal plans. We set~$\widehat \Omega \coloneqq (a,b)$. Then, Assumption~\ref{ass:theta} in \Cref{prop:actovlasov_new} is vacuously true,~\eqref{eq:metderupper1} follows from~\eqref{eq:metderupper2} because~$\tilde \sfd_{t-s}(\mu_s,\mu_t) \ge \norm{w-v}_{\Lp{2}(\pi_{s,t}^\theta)}$,~\eqref{eq:L1conv} is obvious, and~\eqref{eq:bounddtilde} follows from~\eqref{eq:metderupper2}.
\end{proof}

\begin{cor} \label{cor:actovlasov}
	Assume~\ref{ass:cont_new} in \Cref{prop:actovlasov_new} and, in addition, the following.
	\begin{enumerate}[label=(\alph*')] \setcounter{enumi}{2}
		\item \label{ass:ac_bis} There exists a function~$\ell \in \Lp{2}(a,b)$ such that
		\begin{equation} \label{eq:metderupper3}
			\sfd(\mu_s,\mu_t) \le \int_s^t \ell(r) \, \dif r \comma \qquad a < s < t < b \fstop
		\end{equation}
		\item \label{ass:conv_bis}  The set~$\widehat \Omega$ has full measure in~$(a,b)$. For every~$[a',b'] \subseteq \widehat \Omega$, we have the limit
		\begin{equation}
			\label{eq:Thh1}
			\lim_{h \downarrow 0} \int_{a'}^{b'} \abs{\frac{T_{t,t+h}}{h} - 1} \, \dif t
			=
			0 \fstop
		\end{equation}
	\end{enumerate}
	Then the conclusion of \Cref{prop:actovlasov_new} holds.
\end{cor}

\begin{proof}
	Set~$\theta_{s,t} \coloneqq T_{s,t}$ and choose~$\pi_{s,t}^\theta \coloneqq \pi_{s,t}$, which has the following property:
	\begin{equation} \label{eq:d_dtilde}
		T_{s,t} \in (0,\infty) \quad \Longrightarrow \quad \sfd^2(\mu_s,\mu_t) = \sfc(\pi_{s,t}) = \tilde \sfc_{T_{s,t}}(\pi_{s,t}) \fstop
	\end{equation}
	Then, Assumption~\ref{ass:theta} in \Cref{prop:actovlasov_new} follows from~\eqref{eq:defTbis}, and~\eqref{eq:metderupper1} follows from~\eqref{eq:metderupper3} because~$ \sfd(\mu_s,\mu_t) \ge \norm{w-v}_{\Lp{2}(\pi_{s,t})}$. Finally,~\eqref{eq:bounddtilde} follows from~\eqref{eq:d_dtilde},~\eqref{eq:metderupper3}, and~\eqref{eq:Thh1}; indeed,
	\begin{align*}
		\int_{\set{t \in (a',b') \, : \, T_{t,t+h} \in (0,\infty)}} \frac{T_{t,t+h}}{h} \sqrt{\tilde \sfc_{T_{s,t}}(\pi_{s,t})} \, \dif t
			&\le
		\int_{a'}^{b'} \frac{T_{t,t+h}}{h} \int_t^{t+h} \ell(r) \, \dif r \, \dif t \comma \\
			&\le \norm{\ell}_{\Lp{1}} \int_{a'}^{b'} \abs{\frac{T_{t,t+h}}{h} - 1} \, \dif t \\
			&\quad+ h \int_{a'}^{b'+h} \ell(r) \, \dif r \comma
	\end{align*}
	which tends to~$0$ as~$h \downarrow 0$ by~\eqref{eq:Thh1}.
\end{proof}

\begin{rem}
	The combination of \Cref{cor:actovlasov} and \Cref{prop:vlasovtoac3} yields a self-improvement result for the convergence~$\frac{T_{t,t+h}}{h} \to 1$. Indeed, let us make the assumptions of \Cref{cor:actovlasov} with~$\widehat \Omega = \Omega$. In particular, we assume the~$\Lp{1}_\mathrm{loc}(\Omega)$ convergence of~$\frac{T_{t,t+h}}{h}$. From \Cref{cor:actovlasov}, we get \Cref{ass:vlasov} and therefore, by \Cref{prop:vlasovtoac3}, the $\Lp{2}_\mathrm{loc}(\Omega)$- and almost everywhere convergence of~$\frac{T_{t,t+h}}{h}$.
\end{rem}

\Cref{prop:actovlasov_new} and its corollaries reproduce \cite[Theorem~8.3.1]{MR2401600} from the classical OT theory. We will also adopt a similar proof strategy, namely we prove that a certain linear functional is bounded, so as to apply the Riesz representation theorem. Naively, one could try to work with the same functional as in~\cite[Theorem~8.3.1]{MR2401600}, i.e.
\[
	\varphi \longmapsto \int_a^b \int_\Gamma \partial_t \varphi \, \dif \mu_t \, \dif t \comma
\]
and prove that the function representing it is of the form~$(v,F)$. However, it turns out being more natural to treat~$\partial_t + v \cdot \nabla_x$ as a single differential operator, in the spirit of hypoellipticity~\cite{hormander1967hypoelliptic}. Then, we work with the linear functional~$L = L(\varphi)$ defined via
\[
	\varphi \stackrel{L}{\longmapsto} \int_a^b \int_\Gamma (\partial_t \varphi + v \cdot \nabla_x \varphi) \, \dif \mu_t \, \dif t = \int_a^b \int_\Gamma \lim_{h \downarrow 0} \frac{\varphi(t,x,v)-\varphi(t-h,x-hv,v)}{h} \, \dif \mu_t \, \dif t \fstop
\]
We shall prove that, in fact, $L = L(\nabla_v \varphi)$, i.e.,~$L$ is a linear and bounded functional of~$\nabla_v \varphi$, with operator norm~$\norm{L} \le \norm{\ell}_{\Lp{2}}$.
One key ingredient in the proof is that~$x$ and~$y-hw$ coincide to the first order on the support of~$\pi_{t,t+h}^\theta$. More precisely, by means of Assumption~\ref{ass:conv_new}, we show that
\[
	\lim_{h \downarrow 0} \int_{a'}^{b'} \frac{\abs{y-x-hw}}{h} \, \dif \pi_{t,t+h}^\theta \, \dif t
		=
	0 \comma \qquad [a',b'] \subseteq \widehat\Omega \fstop
\]

Let us briefly comment on the assumptions of \Cref{prop:actovlasov_new}. %
	Assumption~\ref{ass:cont_new} is certainly true for any solution to Vlasov's equation with moment bounds by~\cite[Theorem~8.3.1]{MR2401600}. Furthermore, it is independent of the other assumptions, and not even replaceable by Wasserstein BV-continuity of~$(\mu_t)_{t \in (a,b)}$. The following example produces a bounded variation curve satisfying Assumptions~\ref{ass:theta} to \ref{ass:conv_new} which does not solve~\eqref{eq:cont}.

\begin{ex}
	Let~$x_0 \colon (0,1) \to (0,1)$ be the Cantor function, namely a continuous, surjective, nondecreasing function of bounded variation, having a Cantor measure---concentrated on the Cantor set~$C$---as derivative. In dimension~$n=1$, set
	\begin{align*}
		v(t) &\coloneqq \inf \set{\abs{t-c} \, : \, c \in C} \comma  &&t \in (0,1) \comma \\
		x(t) &\coloneqq x_0(t) + \int_0^t v(s) \, \dif s \comma  && t \in (0,1) \comma
	\end{align*}
	and define the curve~$\mu_\cdot \coloneqq \delta_{\bigl(x(\cdot),v(\cdot)\bigr)}$. Choose~$\widehat \Omega \coloneqq (0,1) \setminus C$ and
	\begin{equation} \label{eq:deftheta}
		\theta_{s,t} \coloneqq T_{s,t}  \stackrel{\eqref{eq:defT}}{=} 2 \frac{\abs{x(t)-x(s)}^2}{\bigl(x(t)-x(s)\bigr)\bigl(v(t)+v(s)\bigr)} = \frac{x_0(t)-x_0(s)+ \int_s^{t}v(r) \, \dif r}{v(t)+v(s)}  \fstop
	\end{equation}
	In this case, there is only one admissible plan for every~$s,t$, namely~$\pi_{s,t}^\theta = \mu_s \otimes \mu_t$.
	
	Assumption~\ref{ass:theta} is vacuously true, because~$x$ is strictly increasing. Assumption~\ref{ass:ac_new} holds because
	\[
		\abs{v(t)-v(s)} \le t-s \comma \qquad 0<s<t<1 \fstop
	\]
	The function~$v$ is uniformly bounded away from~$0$ on any compact subset of~$\widehat \Omega$, and~$x_0$ is constant on any interval in~$\widehat \Omega$. Therefore, the convergence~$\frac{\theta_{t,t+h}}{h} \to 1$ holds locally uniformly on~$\widehat \Omega$. When~$\theta_{s,t} \in (0,\infty)$, we have
	\[
		\sqrt{\tilde \sfc_{\theta_{s,t}}(\pi_{s,t}^\theta)} = \abs{v(t) - v(s)} \le t-s \comma
	\]
	and with this we verify Assumption~\ref{ass:conv_new}.

	Nevertheless, this curve does not solve Vlasov's equation for any force field~$(F_t)_t$ such that~$\int_0^1 \abs{F_t\bigl(x(t),v(t)\bigr)}^2 \, \dif t < \infty$. If it did, then, by \Cref{lem:momentestimate}, we would have
	\[
		\abs{x(t)} \stackrel{\eqref{eq:momentestimate2}}{\le} \abs{x(0)} + \int_0^t \abs{v(s)} \, \dif s \comma \qquad t \in (0,1) \comma
	\]
	which would imply~$x_0 \equiv 0$.
\end{ex}

Assumption~\ref{ass:theta}, Assumption~\ref{ass:ac_new}, and~\eqref{eq:bounddtilde}, together, are a weakened version of the natural absolute continuity condition
\[
	\sqrt{\tilde \sfc_{\theta_{t,t+h}}(\pi^\theta_{t,t+h})} \le \int_s^t \ell(r) \, \dif r
\]
(calling~$\tilde \sfc_0(\pi^\theta_{t,t+h})$ the limit of~$\tilde \sfc_\epsilon(\pi^\theta_{t,t+h})$ for~$\epsilon \to 0$), as can be easily checked (using~\eqref{eq:L1conv}).

Assuming~$\frac{\theta_{t,t+h}}{h} \to 1$ is needed to select a solution to Vlasov's equation~\eqref{eq:cont} among all its possible reparametrisations. Recall also the hypothesis that~$\widehat \Omega$ has full measure in~$(a,b)$. This ensures that~$(\mu_t)_t$ solves Vlasov's equation \emph{on the whole~$(a,b)$}. In the next lemma, we show that there are conditions under which~$\Omega$ has full measure and, therefore, it may be a viable choice for~$\widehat \Omega$. More precisely, we establish a connection between~$\abs{\rho'}_{\mathrm{W}_2}$ and~$\norm{v}_{\Lp{2}(\mu_t)}$ in terms of~$\liminf_{h \downarrow 0} \frac{\theta_{t,t+h}}{h}$ \emph{a priori}, i.e.,~without knowing that~$(\mu_t)_t$ solves Vlasov's equation. For solutions to Vlasov's equation, the analogous statement is \Cref{lem:derW2}.

\begin{lem} \label{lem:derrhov_new}
	Let~$t \in (a,b)$ be a ${\mathrm{W}_2}$-differentiability point for~$\tilde t \to \rho_{\tilde t}$, and assume that there exists a sequence~$h_k \downarrow 0$ such that
	\begin{equation}
		l \coloneqq \lim_{k \to \infty} \frac{\theta_{t,t+h_k}}{h_k} < \infty\comma \quad \theta_{t,t+h_k} \in (0,\infty) \, \forall \, k \comma \quad \text{and} \quad
		\lim_{k \to \infty} \tilde \sfc_{\theta_{t,t+h_k}}(\pi_{t,t+h_k}^\theta) = 0 \fstop
	\end{equation}
	Then,
	\begin{equation}
		\abs{\rho_t'}_{{\mathrm{W}_2}} \le l \, \norm{v}_{\Lp{2}(\mu_t)} \fstop
	\end{equation}
	As a consequence, if
	\begin{equation} \label{eq:derrhovT}
		\liminf_{h \downarrow 0} \frac{T_{t,t+h}}{h} < \infty \comma \quad \lim_{h \downarrow 0} \sfd(\mu_t,\mu_{t+h}) = 0 \comma \quad \text{and} \quad \abs{\rho_t'}_{\mathrm{W_2}} > 0
	\end{equation}
	for a.e.~$t \in (a,b)$, then~$\Omega$ has full measure in~$(a,b)$.
\end{lem}

\begin{rem} \label{rem:thetapos}
	If~$\abs{\rho_t'}_{\mathrm{W}_2} > 0$ and Assumption~\ref{ass:theta} holds, then~$\theta_{t,t+h} > 0$ for small~$h$.
\end{rem}

\begin{proof}[Proof of \Cref{lem:derrhov_new}]
	We write
	\begin{align*}
		\frac{\norm{y-x}_{\Lp{2}(\pi_{t,t+h_k}^\theta)}}{\theta_{t,t+h_k}}
		&\le
		\norm{\frac{y-x}{\theta_{t,t+h_k}} - \frac{v+w}{2}}_{\Lp{2}(\pi_{t,t+h_k}^\theta)} + \norm{\frac{w-v}{2}}_{\Lp{2}(\pi_{t,t+h_k}^\theta)} + \norm{v}_{\Lp{2}(\mu_t)} \\
		&\le \sqrt{\tilde \sfc_{\theta_{t,t+h_k}}(\pi_{t,t+h_k}^\theta)} + \norm{v}_{\Lp{2}(\mu_t)} \fstop
	\end{align*}
	Therefore, by hypothesis,
	\[
	\abs{\rho_t'}_{\mathrm{W}_2} \le l \, \liminf_{k \to \infty} \frac{\norm{y-x}_{\Lp{2}(\pi_{t,t+h_k}^\theta)}}{\theta_{t,t+h_k}} = l \, \norm{v}_{\Lp{2}(\mu_t)} \fstop
	\]
	
	The proof that~$\Omega$ has full measure under~\eqref{eq:derrhovT} is consequence of \Cref{rem:thetapos} and the fact that~$\sfd(\mu_t,\mu_{t+h}) = \tilde \sfc_{T_{t,t+h}}(\pi_{t,t+h})$ if~$T_{t,t+h} \in (0,\infty)$.
\end{proof}

\begin{proof}[Proof of \Cref{prop:actovlasov_new}]
	Let us fix~$[a',b'] \subseteq \widehat \Omega$ and~$\varphi \in \mathrm{C}^\infty_\mathrm{c}\bigl((a',b') \times \cX \times \cV \bigr)$. We write
	\begin{align*}
		L(\varphi)
		&\coloneqq
		\int_{a}^{b}  \int_\Gamma (\partial_t \varphi + v \cdot \nabla_x \varphi ) \, \dif \mu_t \, \dif t \\
		&\le
		\abs{\int_{a'}^{b'} \int_\Gamma \lim_{h \downarrow 0} \frac{\varphi(t,x,v)-\varphi(t-h,x-hv,v)}{h} \, \dif \mu_t \, \dif t}
	\end{align*}
	and, by the dominated convergence theorem and a change of time variable,
	\begin{align*}
		L(\varphi)
		&\le
		\lim_{h \downarrow 0} \abs{\int_{a'}^{b'} \int \frac{\varphi(t,x,v)-\varphi(t-h,x-hv,v)}{h} \, \dif \mu_t \, \dif t} \\
		&=
		\lim_{h \downarrow 0} \abs{\frac{\int_{a'}^{b'} \int \varphi(t,x,v) \, \dif \mu_{t} \, \dif t - \int_{a'}^{b'} \int \varphi(t,x-hv,v) \, \dif \mu_{t+h} \, \dif t }{h} } \fstop%
	\end{align*}
	Hence, we have
	\begin{align*}
		L(\varphi)
		&\le
		\lim_{h \downarrow 0} \abs{\int_{a'}^{b'} \int_\Gamma \frac{\varphi(t,x,v)-\varphi(t,y-hw,w)}{h} \, \dif \pi_{t,t+h}^\theta \, \dif t} \\
		&\le
		\underbrace{\liminf_{h \downarrow 0} \int_{a'}^{b'} \int_{\Gamma \times \Gamma} \frac{\abs{\varphi(t,x,v)-\varphi(t,x,w)}}{h} \, \dif \pi_{t,t+h}^\theta \, \dif t}_{\eqqcolon I_{\varphi,1}} \\
		&\quad
		+ \underbrace{\limsup_{h \downarrow 0} \int_{a'}^{b'} \int_{\Gamma \times \Gamma} \frac{\abs{\varphi(t,x,w)-\varphi(t,y-hw,w)}}{h} \, \dif \pi_{t,t+h}^\theta \, \dif t}_{\eqqcolon I_{\varphi,2}} \fstop
	\end{align*}
	
	We start by estimating~$I_{\varphi,1}$. By the Cauchy--Schwarz inequality:
	\begin{multline} \label{eq:estI1}
		I_{\varphi,1}
		\le
		\liminf_{h \downarrow 0} \sqrt{\int_{a'}^{b'} \int_{\Gamma \times \Gamma} \frac{\abs{\varphi(t,x,v)-\varphi(t,x,w)}^2}{\abs{v-w}^2} \, \dif  \pi_{t,t+h}^\theta \, \dif t} \\
		\cdot \sqrt{\int_{a'}^{b'} \int_{\Gamma \times \Gamma} \frac{\abs{v-w}^2}{h^2} \, \dif \pi_{t,t+h}^\theta \, \dif t} \fstop
	\end{multline}
	By Assumption~\ref{ass:ac_new}, we have~$\norm{w-v}_{\Lp{2}(\pi_{t,t+h}^\theta)} \to 0$ for a.e.~$t$. Therefore, the first square root in~\eqref{eq:estI1} converges to~$\norm{\nabla_v \varphi}_{\Lp{2}(\mu_t \dif t)}$ by the dominated convergence theorem. As for the second one, again by Assumption~\ref{ass:ac_new}, we write
	\begin{equation} \label{eq:boundF}
	\int_{a'}^{b'} \int_{\Gamma \times \Gamma} \frac{\abs{v-w}^2}{h^2} \, \dif \pi_{t,t+h}^\theta \, \dif t
		\le
	\int_{a'}^{b'} \left(\fint_t^{t+h} \ell(r) \, \dif r\right)^2 \, \dif t
		\le
	\int_{a'}^{b'+h} \ell^2(r) \, \dif r
	\end{equation}
	from which we conclude that~$I_{\varphi,1} \le \norm{\nabla_v \varphi}_{\Lp{2}(\mu_t \dif t)} \norm{\ell}_{\Lp{2}}$.

	We claim that~$I_{\varphi,2} = 0$. To prove it, we estimate
	\[
		I_{\varphi,2}
			\le
		\norm{\varphi}_{C^1} \limsup_{h \downarrow 0} \int_{a'}^{b'} \int_{\Gamma \times \Gamma}  \frac{\abs{y-x-h\, w}}{h} \, \dif \pi_{t,t+h}^\theta \, \dif t \fstop
	\]
	Momentarily fix~$t$ and~$h$. If~$\theta_{t,t+h} \in (0,\infty)$, then the triangle inequality gives
	\[
		\frac{\abs{y-x-h\, w}}{h} 
		\le
		\abs{\frac{y-x}{h}-\frac{\theta_{t,t+h}}{h}\frac{v+w}{2}}
		+ \frac{\abs{w-v}}{2} 
		+  \abs{\frac{\theta_{t,t+h}}{h}-1}  \frac{\abs{v+w}}{2} \comma
	\]
	and, therefore,
	\begin{multline*}
		\int_{\Gamma \times \Gamma} \frac{\abs{y-x-h\, w}}{h} \, \dif \pi_{t,t+h}^\theta
			\le
		\frac{\theta_{t,t+h}}{h} \sqrt{\tilde \sfc_{\theta_{t,t+h}}(\pi_{t,t+h}^\theta)} + \int_t^{t+h} \ell(r) \, \dif r \\
			+ \abs{\frac{\theta_{t,t+h}}{h}-1}  \frac{\norm{v}_{\Lp{2}(\mu_t)} + \norm{v}_{\Lp{2}(\mu_{t+h})}}{2}
	\end{multline*}
	If~$\theta_{t,t+h}=0$, then by Assumption~\ref{ass:theta},
	\[
		\int_{\Gamma \times \Gamma} \frac{\abs{y-x-h\, w}}{h} \, \dif \pi_{t,t+h}^\theta
			=
		\norm{v}_{\Lp{2}(\mu_{t+h})} = \abs{\frac{\theta_{t,t+h}}{h}-1} \norm{v}_{\Lp{2}(\mu_{t+h})} \fstop
	\]
	If~$\theta_{t,t+h} = \infty$, then, trivially,
	\[
		\int_{\Gamma \times \Gamma} \frac{\abs{y-x-h\, w}}{h} \, \dif \pi_{t,t+h}^\theta
		\le
	 \abs{\frac{\theta_{t,t+h}}{h}-1} \fstop
	\]
	Hence, we find
	\begin{align*}
		\int_{a'}^{b'} \int_{\Gamma \times \Gamma}  \frac{\abs{y-x-h\, w}}{h} \, \dif \pi_{t,t+h}^\theta \, \dif t
			&\le
		\int_{\set{t \in (a',b') \, : \, \theta_{t,t+h} \in (0,\infty)}} \frac{\theta_{t,t+h}}{h} \sqrt{\tilde \sfc_{\theta_{t,t+h}}(\pi_{t,t+h}^\theta)} \, \dif t \\
			&\quad+
		h \int_{a'}^{b'+h} \ell(r) \, \dif r \\
			&\quad+
		\left(\sup_{\frac{a+a'}{2} \le t \le \frac{b+b'}{2}} \norm{v}_{\Lp{2}(\mu_t)} +1 \right)\int_{a'}^{b'} \abs{\frac{\theta_{t,t+h}}{h}-1} \, \dif t \comma
	\end{align*}
	and Assumption~\ref{ass:conv_new} allows to conclude that~$I_{\varphi,2} = 0$.

	We established that~$L(\varphi)  \le \norm{\nabla_v \varphi}_{\Lp{2}(\mu_t \dif t)} \norm{\ell}_{\Lp{2}}$ for every~$\varphi \in \mathrm{C}_\mathrm{c}^\infty\bigl((a',b') \times \cX \times \cV \bigr)$ with~$[a',b'] \subseteq \Omega$. By linearity of~$L$ and arbitrariness of~$a',b'$, we have, in fact, that
	\begin{equation} \label{eq:boundL2}
		L(\varphi) \le \norm{\nabla_v \varphi}_{\Lp{2}(\mu_t \dif t)} \norm{\ell}_{\Lp{2}} \quad \text{for every } \varphi \in \mathrm{C}_\mathrm{c}^\infty(\widehat \Omega \times \cX \times \cV ) \fstop
	\end{equation}
	We claim that the same inequality holds for every~$\varphi \in \mathrm{C}_\mathrm{c}^\infty \bigl((a,b) \times \cX \times \cV\bigr)$. Given one such~$\varphi$, and a function~$\eta \in \mathrm{C}^\infty_{\mathrm{c}}(\widehat \Omega)$, we write
	\begin{align*}
		L(\varphi)
			&=
		\int_a^b \int_\Gamma (\partial_t \varphi + v \cdot \nabla_x \varphi) \, \dif \mu_t \, \dif t \\
			&=
		\int_a^b \bigl(1-\eta(t)\bigr) \int_\Gamma (\partial_t \varphi + v \cdot \nabla_x \varphi) \, \dif \mu_t \, \dif t \\
			&\quad +
		\int_{a}^b \left( \eta(t) \int_\Gamma (\partial_t  \varphi 
		+ v \cdot \nabla_x  \varphi) \, \dif \mu_t + \partial_t \eta(t) \int_\Gamma  \varphi  \, \dif \mu_t \right) \, \dif t - \int_a^b \partial_t \eta(t) \int_\Gamma \varphi \, \dif \mu_t \, \dif t \fstop
	\end{align*}
	Since~$(\mu_t)_{t \in (a,b)}$ is~${\mathrm{W}_2}$-$2$-a.c.,~and~$\varphi$ is smooth and compactly supported, the function~$t \mapsto \int \varphi(t,\cdot) \, \dif \mu_t$ is~$2$-a.c. Therefore, an integration by parts yields
	\begin{align*}
		L(\varphi)
		&=
		\int_a^b \bigl(1-\eta(t)\bigr) \int_\Gamma (\partial_t \varphi + v \cdot \nabla_x \varphi) \, \dif \mu_t \, \dif t \\
		&\quad +
		\int_{a}^b \int_\Gamma \bigl(\partial_t ( \eta \varphi) 
		+ v \cdot \nabla_x  (\eta \varphi) \bigr) \, \dif \mu_t  \, \dif t 
		+
		\int_a^b \eta(t) \frac{\dif}{\dif t} \int_\Gamma \varphi \, \dif \mu_t  \, \dif t \fstop
	\end{align*}
	Since~$\eta \varphi \in \mathrm{C}_\mathrm{c}^\infty(\widehat \Omega \times \cX \times \cV)$, we apply~\eqref{eq:boundL2} to write
	\begin{multline*}
		L(\varphi) \le \int_a^b \bigl(1-\eta(t)\bigr) \int (\partial_t \varphi + v \cdot \nabla_x \varphi) \, \dif \mu_t \, \dif t + \norm{\eta \nabla_v \varphi}_{\Lp{2}(\mu_t \dif t)} \norm{\ell}_{\Lp{2}(a,b)} \\
		+
		\int_a^b \eta(t) \frac{\dif}{\dif t} \int_\Gamma \varphi \, \dif \mu_t  \, \dif t \fstop
	\end{multline*}
	By Assumption~\ref{ass:conv_new}, the complement of~$\widehat\Omega$ is Lebesgue negligible. Thus,~$\eta$ can approximate the constant function~$1$ in~$\Lp{2}(\widehat \Omega) = \Lp{2}(a,b)$. This gives
	\[
		L(\varphi) \le \norm{ \nabla_v \varphi}_{\Lp{2}(\mu_t \dif t)} \norm{\ell}_{\Lp{2}(a,b)}
		+
		\int_a^b \frac{\dif}{\dif t} \int_\Gamma \varphi \, \dif \mu_t  \, \dif t = \norm{ \nabla_v \varphi}_{\Lp{2}(\mu_t \dif t)} \norm{\ell}_{\Lp{2}} \comma
	\]
	which was our claim.
	
	Finally, we apply the Riesz representation theorem on the closure in~$\Lp{2}(\mu_t \dif t)$ of the set~$\set{\nabla_v \varphi \, : \, \varphi \in \mathrm{C}^\infty_\mathrm{c}\bigl((a,b) \times \cX \times \cV \bigr)}$ to find~$(F_t)_t$
such that~$(\mu_t,F_t)_{t \in (a,b)}$ solves~\eqref{eq:cont}.
\end{proof}

\subsection{First-order differential calculus} \label{ss:53}

Let $(\mu_t)_{t \in (a,b)}$ be a narrowly continuous curve in $\mathcal{P}_2(\Gamma)$, let~$(s,t) \mapsto \pi_{s,t}$ be a measurable selection of optimal transport plans for~$\sfd$, and let~$T_{s,t}$ be the corresponding optimal times, see~\eqref{eq:defTbis}. Recall that~$\rho_t \coloneqq (\proj_x)_\# \mu_t$ and~$\langle v \rangle_t = \int v \, \dif \mu_t$.

\subsubsection{The tangent}\label{ss:tangent}
In the light of the previous sections, we can now give a more rigorous description of the geometric intuitions of \Cref{rem:tgc1}, \Cref{rem:tgc1bis}, and \Cref{rem:tgc2}. Given a solution~$(\mu_t,F_t)_{t \in (a,b)}$ to Vlasov's equation~\eqref{eq:cont}, we can see~$v$ as the infinitesimal $x$-variation (i.e.~the \emph{velocity}), and the field~$(F_t)_t$ driving~$(\mu_t)_t$ as the infinitesimal $v$-variation (i.e.,~the \emph{acceleration} or \emph{force}).
In the case of the particle model---Section~\ref{sec:sec2}---we have that, along a \emph{regular} solution to Newton's equations $\dot {x}_t=v_t, \quad \dot{v}_t =F_t(x_t,v_t)$, it holds true:
\[
\begin{aligned}
&x_t = x_0 + t \, v_0 + o(t),\quad t \to 0 \comma \\
&v_t = v_0 + t \, F_0(x_0,v_0) + o(t), \quad t \to 0 \comma \\
&x_t = x_0 + t v_0 + \frac{1}{2} \, t^2 \, F_0(x_0,v_0) + o (t^2) \comma \quad t \to 0 \fstop
\end{aligned}
\]
In the next propositions, we recover analogous formulae in the case of evolutions of measures along Vlasov's equations. The heuristic argument---given in \Cref{rem:tgc1bis}---is the following. Along a solution $(\mu_t)_t$ to \eqref{eq:cont}, the optimal plan $\pi_{t,t+h}$ for $\mathsf{d}(\mu_t,\mu_{t+h})$ 
is \emph{close} to the projection 
\[
\left(\mathrm{pr}_{(\alpha(t),\alpha'(t)),(\alpha(t+h),\alpha'(t+h))}\right)_\# \, \mathbf m
\] 
of the dynamical transport plan $\mathbf{m}$ induced by Vlasov's equation itself (cf.~\cite[Theorem~8.2.1]{MR2401600}).
Quantitative statements are given below.

\begin{prop} \label{prop:tangent}
	Suppose that \Cref{ass:vlasov} holds (i.e.,~$(\mu_t,F_t)_t$ is a solution to Vlasov's equation for a field~$(F_t)_t$), with~$(F_t)_t$ belonging to the~$\Lp{2}(\mu_t \dif t)$-closure of the set~$\set{\nabla_v \varphi \, : \, \varphi \in \mathrm{C}^\infty_\mathrm{c}\bigl((a,b) \times \G \bigr)}$. Then, for a.e.~$t \in (a,b)$ such that~$\abs{\rho_t'}_{{\mathrm{W}_2}} > 0$, we have
	\begin{align} %
		\label{eq:tangent2}
		&\lim_{h \downarrow 0} \frac{1}{h} \norm{w-v-hF_t(x,v)}_{\Lp{2}(\pi_{t,t+h})}
			=
		0 \comma \\ \label{eq:tangent3}
		&\lim_{h \downarrow 0} \frac{1}{h^2} \norm{y-x-T_{t,t+h}\frac{v+w}{2}}_{\Lp{2}(\pi_{t,t+h})}
			=
		0 \fstop
	\end{align}
\end{prop}
\begin{prop} \label{prop:tangent2}
	In the setting of \Cref{prop:tangent}, for a.e.~$t$ such that~$\langle v \rangle_t \neq 0$,
	we have
	\begin{equation} \label{eq:tangent4}
		\lim_{h \downarrow 0} \frac{T_{t,t+h}-h}{h^2} = 0
	\end{equation}
	and
	\begin{equation} \label{eq:tangent5}
		\lim_{h \downarrow 0} \frac{1}{h^2} \norm{y-x-hv-\frac{h^2}{2} F_t(x,v)}_{\Lp{2}(\pi_{t,t+h})}
			=
		0 \fstop
	\end{equation}
\end{prop}

\begin{proof}[Proof of \Cref{prop:tangent}]
	Let~$\set{\varphi_k}_{k \in \N}$ be a~$\mathrm{C}^1$-dense set of~$\mathrm{C}_\mathrm{c}^\infty(\Gamma)$ functions. For almost every~$t \in (a,b)$, we have
	\begin{enumerate}
		\item $\abs{\rho_t'}_{\mathrm{W_2}} > 0$,
		\item $\frac{T_{t,t+h}}{h} \to 1$, cf.~\Cref{cor:limT},
		\item \Cref{eq:general} holds with~$\varphi \coloneqq \varphi_k$ for every~$k \in \N$,
		\item $\limsup_{h \downarrow 0} \frac{\sfd(\mu_t,\mu_{t+h})}{h} \le \norm{F_t}_{\Lp{2}(\mu_t)} < \infty$, cf.~\Cref{prop:vlasovtoac2},
		\item $F_t$ belongs to the~$\Lp{2}(\mu_t)$-closure of~$\set{\nabla_v \varphi \, : \, \varphi \in \mathrm{C}_\mathrm{c}^\infty(\Gamma)}$.
	\end{enumerate}
	Let~$t$ be one such time. %
	To prove~\eqref{eq:tangent2}, we compute
	\begin{multline*}
		\norm{w-v-hF_t(x,v)}_{\Lp{2}(\pi_{t,t+h})}^2 = \norm{w-v}_{\Lp{2}(\pi_{t,t+h})}^2 + h^2 \norm{F_t}_{\Lp{2}(\mu_t)}^2 \\ - 2h\int (w-v) \cdot F_t(x,v) \, \dif \pi_{t,t+h} \comma
	\end{multline*}
	hence,  resorting to \Cref{prop:vlasovtoac2},
	\begin{equation} \label{eq:tangent1bis}
		\limsup_{h \downarrow 0} \frac{\norm{w-v-hF_t}_{\Lp{2}(\pi_{t,t+h})}^2}{h^2}
			\stackrel{\eqref{eq:vlasovtoac2}}{\le}
		2 \norm{F_t}_{\Lp{2}(\mu_t)}^2 - 2 \liminf_{h \downarrow 0} \int \frac{w-v}{h} \cdot F_t(x,v) \, \dif \pi_{t,t+h} \fstop
	\end{equation}
	We estimate the last integral. For every~$k$, we have
	\begin{align*}
		&\int F_t \cdot \nabla_v \varphi_k \, \dif \mu_t \\
		&\quad\stackrel{\eqref{eq:general}}{=}
		\lim_{h \downarrow 0} \left( \int \frac{w-v}{h} \cdot \nabla_v \varphi_k(x,v) \, \dif \pi_{t,t+h} + \left(\frac{T_{t,t+h}}{h}-1\right) \int v \cdot \nabla_x \varphi_k \, \dif \mu_t \right) \\
		&\quad\stackrel{\eqref{eq:limT}}{=}
		\lim_{h \downarrow 0} \int \frac{w-v}{h} \cdot \nabla_v \varphi_k(x,v) \, \dif \pi_{t,t+h} \\
		&\quad \le
		\liminf_{h \downarrow 0} \int \frac{w-v}{h} \cdot F_t \, \dif \pi_{t,t+h} + \norm{\nabla_v \varphi_k - F_t}_{\Lp{2}(\mu_t)} \limsup_{h \downarrow 0} \frac{\norm{w-v}_{\Lp{2}(\pi_{t,t+h})}}{h} \\
		&\quad \stackrel{\eqref{eq:vlasovtoac2}}{\le} \liminf_{h \downarrow 0} \int \frac{w-v}{h} \cdot F_t \, \dif \pi_{t,t+h} + \norm{\nabla_v \varphi_k - F_t}_{\Lp{2}(\mu_t)} \norm{F_t}_{\Lp{2}(\mu_t)} \fstop
	\end{align*}
	By arbitrariness of~$k$,
	\[
		\int F_t \cdot \nabla_v \varphi \, \dif \mu_t
			\le
		\liminf_{h \downarrow 0} \int \frac{w-v}{h} \cdot F_t \, \dif \pi_{t,t+h} + \norm{\nabla_v \varphi - F_t}_{\Lp{2}(\mu_t)} \norm{F_t}_{\Lp{2}(\mu_t)}
	\]
	for every~$\varphi \in \mathrm{C}_\mathrm{c}^\infty(\Gamma)$. Since~$F_t$ belongs to the~$\Lp{2}(\mu_t)$-closure of~$\set{\nabla_v \varphi \, : \, \varphi \in \mathrm{C}_\mathrm{c}^\infty(\Gamma)}$,
	\[
		\norm{F_t}_{\Lp{2}(\mu_t)}^2 \le \liminf_{h \downarrow 0} \int \frac{w-v}{h} \cdot F_t \, \dif \pi_{t,t+h} \comma
	\]
	which, together with~\eqref{eq:tangent1bis}, yields~\eqref{eq:tangent2}.
	
	Let us now prove~\eqref{eq:tangent3}. By definition of~$\tilde \sfd$, we write
	\[
		\frac{12}{h^4}\norm{y-x-T_{t,t+h}\frac{v+w}{2}}_{\Lp{2}(\pi_{t,t+h})}^2 = \frac{T_{t,t+h}^2}{h^2} \, \frac{ \tilde \sfd^2(\mu_t,\mu_{t+h}) - \norm{w-v}_{\Lp{2}(\pi_{t,t+h})}^2 }{h^2} \comma
	\]
	therefore, it suffices that
	\[
		\lim_{h \downarrow 0} \frac{ \tilde \sfd^2(\mu_t,\mu_{t+h}) - \norm{w-v}_{\Lp{2}(\pi_{t,t+h})}^2 }{h^2} = 0 \comma
	\]
	which follows from~\eqref{eq:vlasovtoac2} and~\eqref{eq:tangent2}.
\end{proof}

\begin{proof}[Proof of \Cref{prop:tangent2}]
	Let~$t \in (a,b)$ be such that~$\langle v \rangle_t \neq 0$. In the light of \Cref{lem:derW2}, \Cref{cor:limT}, \Cref{prop:tangent}, we may assume that \eqref{eq:limT},~\eqref{eq:tangent2},~\eqref{eq:tangent3} hold, and, additionally,~$t$ is a Lebesgue point for
	\[
		s \longmapsto \int F_s \, \dif \mu_s \fstop
	\]
	Fix~$h \in (0,b-t)$ such that~$T_{t,t+h}<\infty$, and let~$\boldsymbol \eta, \mathbf P$ be as in \Cref{lem:boundT}. In particular,
	\[
		\int y \cdot \langle v \rangle_t \, \dif \mathbf{P} - \int x \cdot \langle v \rangle_t \, \dif \mathbf{P}
			=
		\int \gamma_x(t+h) \cdot \langle v \rangle_t \, \dif \mathbf{P} - \int \gamma_x(t) \cdot\langle v \rangle_t \, \dif \mathbf{P} \comma
	\]
	thus,
	\begin{align*}
		&\int (y-x) \cdot \langle v \rangle_t \, \dif \pi_{t,t+h} = \int \int_t^{t+h} \dot \gamma_x(s) \cdot \langle v \rangle_t \, \dif s \, \dif \boldsymbol \eta = \int \int_t^{t+h} \gamma_v(s) \cdot \langle v \rangle_t \, \dif s \, \dif \boldsymbol \eta \\
		&\quad= \int \int_t^{t+h} \gamma_v(s) \cdot \langle v \rangle_t \, \dif s \, \dif \boldsymbol \eta = \int \left(h\gamma_v(t) + \int_t^{t+h} (t+h-s) \dot \gamma_v(s) \, \dif s \right) \cdot \langle v \rangle_t \, \dif \boldsymbol{\eta} \\
		&\quad = \int \left(h\gamma_v(t) + \int_t^{t+h} (t+h-s) F_s\bigl(\gamma_x(r),\gamma_v(r)\bigr) \, \dif s \right) \cdot \langle v \rangle_t \, \dif \boldsymbol{\eta} \\
		&\quad= h \int v \cdot \langle v \rangle_t \, \dif \mu_t + \int_t^{t+h} (t+h-s) \int F_s \cdot \langle v \rangle_t \, \dif \mu_s \, \dif s \fstop
	\end{align*}
	We infer that
	\begin{align*}
		&\frac{T_{t,t+h}-h}{h^2} \langle v \rangle_t^2 = \fint_t^{t+h} \frac{t+h-s}{h} \left( \int F_s \cdot \langle v \rangle_t \, \dif \mu_s - \int F_t \cdot \langle v \rangle_t \, \dif \mu_t \right) \, \dif s \\
		&\quad + \int F_t \cdot \langle v \rangle_t \, \dif \mu_t \fint_t^{t+h} \frac{t+h-s}{h} \, \dif s \\
		&\quad+ \frac{T_{t,t+h}}{h} \int \frac{v-w}{2h} \cdot \langle v \rangle_t \, \dif \pi_{t,t+h} -  \frac{1}{h^2} \int \left(y-x-T_{t,t+h}\frac{v+w}{2}\right) \cdot \langle v \rangle_t \, \dif \pi_{t,t+h}  \fstop
	\end{align*}
	Let us analyse the four terms at the right-hand side one by one. The first one is negligible, because we can bound~$\frac{t+h-s}{h} \le 1$ and use the Lebesgue differentiation theorem. The second one is equal, for every~$h>0$, to~$\frac{1}{2} \int F_t \cdot \langle v \rangle_t \, \dif \mu_t$. The third one converges to this same quantity with inverse sign (i.e.,~$-\frac{1}{2} \int F_t \cdot \langle v \rangle_t \, \dif \mu_t$) by~\eqref{eq:limT} and~\eqref{eq:tangent2}. The fourth one is negligible by~\eqref{eq:tangent3}.
	
	Since~$\langle v \rangle_t \neq 0$, the proof of~\eqref{eq:tangent4} is complete, and~\eqref{eq:tangent5} follows from~\eqref{eq:tangent2},~\eqref{eq:tangent3},~\eqref{eq:tangent4}.
\end{proof}

\subsubsection{$\sfd$-derivative}

For a solution~$(\mu_t)_t$ to Vlasov's equation, we are going to prove that the limits of the incremental ratios~$\frac{\sfd(\mu_t,\mu_{t+h})}{h}$ and~$\frac{\tilde \sfd_h(\mu_t,\mu_{t+h})}{h}$ as~$h \downarrow 0$ exist and are equal to the smallest norm of a force field~$(F_t)_t$ driving~$(\mu_t)_t$. This is similar to a consequence of \cite[Theorem~1.1.2 \& Theorem~8.3.1]{MR2401600} in classical OT. A major obstacle in replicating these results is that~$\sfd$ is not a distance. 

The classical way to show that~$\frac{\sfd(\mu_t,\mu_{t+h})}{h}$ has a limit for almost every~$t$---without extracting a subsequence---is \cite[Thereom~1.1.2]{MR2401600}, which relies on the triangle inequality. The proof we give here is, instead, based on \Cref{lem:general}, which, for a.e.~$t \in (a,b)$ with~$\abs{\rho_t'}_{\mathrm{W}_2} > 0$, identifies~$F_t$ with the direction of infinitesimal change of the velocities, which in turn bounds the incremental ratio of~$\sfd$ from below.

\begin{rem}
	\Cref{prop:metricder} below is a generalised version of \Cref{prop:ddiff2}, which provided the $d$-derivative in the particle-model case.
\end{rem}

\begin{prop} \label{prop:metricder}
	Under \Cref{ass:vlasov} with~$(F_t)_t$ belonging to the~$\Lp{2}(\mu_t \dif t)$-closure of the set~$\set{\nabla_v \varphi \, : \, \varphi \in \mathrm{C}^\infty_\mathrm{c}\bigl((a,b) \times \G \bigr)}$, for a.e.~$t$ such that~$\abs{\rho_t'}_{{\mathrm{W}_2}} > 0$, we have the limits
	\begin{equation} \label{eq:metricder}
		\lim_{h \downarrow 0} \frac{\sfd(\mu_t,\mu_{t+h})}{h}
			=
		\lim_{h \downarrow 0} \frac{\tilde \sfd_h(\mu_t,\mu_{t+h})}{h}
			=
		\norm{F_t}_{\Lp{2}(\mu_t)} \fstop
	\end{equation}
\end{prop}

\begin{proof}
	The inequality~$\le$ is given by \Cref{prop:vlasovtoac2}. The inequality~$\ge$ follows from \Cref{prop:tangent}:
	\[
		\liminf_{h \downarrow 0} \frac{\tilde \sfd_h(\mu_t,\mu_{t+h})}{h}
			\ge
		\liminf_{h \downarrow 0} \frac{\sfd(\mu_t,\mu_{t+h})}{h} \ge \liminf_{h \downarrow 0} \frac{\norm{w-v}_{\Lp{2}(\pi_{t,t+h})}}{h} \stackrel{\eqref{eq:tangent2}}{=} \norm{F_t}_{\Lp{2}(\mu_t)}
	\]
	for a.e.~$t \in (a,b)$ such that~$\abs{\rho_t'}_{\mathrm{W}_2} > 0$.
\end{proof}

The limit
\begin{equation}
	\label{eq:lim_dhh}
	\lim_{h \downarrow 0} \frac{\tilde \sfd_h(\mu_t,\mu_{t+h})}{h}
	=
	\norm{F_t}_{\Lp{2}(\mu_t)} \comma
\end{equation}
can be obtained without the assumption~$\abs{\rho_t'}_{\mathrm{W}_2} > 0$.

\begin{prop} \label{prop:metricder_bis}
	Under \Cref{ass:vlasov} with~$(F_t)_t$ belonging to the~$\Lp{2}(\mu_t \dif t)$-closure of the set~$\set{\nabla_v \varphi \, : \, \varphi \in \mathrm{C}^\infty_\mathrm{c}\bigl((a,b) \times \G \bigr)}$, for a.e.~$t$, we have~\eqref{eq:lim_dhh}.
\end{prop}

\begin{proof}
	The inequality~$\le$ is given by \Cref{prop:vlasovtoac2}. To prove~$\ge$, we adopt a similar strategy as in the proof of \Cref{lem:general} and \Cref{prop:tangent2}. Let us fix a~$\mathrm{C}^1$-dense set of functions~$\set{\varphi_k}_{k \in \N} \subseteq \mathrm{C}_\mathrm{c}^\infty(\Gamma)$, and~$t$ such that:
	\begin{enumerate}
		\item $t$ is a Lebesgue point of the functions
		\[
			s \longmapsto \int_\Gamma (\tilde v,F_s) \cdot \nabla_{x,v} \varphi_k \, \dif \mu_s \quad \text{ for every } k \in \N \comma \quad \text{and} \quad s \longmapsto \norm{F_s}_{\Lp{2}(\mu_s)} \fstop
		\]
		\item $F_t$ belongs to the~$\Lp{2}(\mu_t)$-closure of~$\set{\nabla_v \varphi \, : \, \varphi \in \mathrm{C}_\mathrm{c}^\infty(\Gamma)}$.
	\end{enumerate}
	The points~$t$ satisfying the previous conditions form a full-measure set in~$(a,b)$.
	
	Fix~$h \in (0,b-t)$, let~$\bar \pi \in \Pi(\mu_t,\mu_{t+h})$  be~$\tilde \sfd_h$-optimal, let~$\boldsymbol \eta \in \cP\bigl(\Gamma \times \mathrm{H}^1(t,t+h;\Gamma)\bigr)$ be as in the proof of \Cref{prop:vlasovtoac1} (after replacing~$(s,t)$ with~$(t,t+h)$). By gluing~$\bar \pi$ and~$\boldsymbol \eta$ at~$\mu_{t}$, construct~$\mathbf P \in \cP\bigl(\G \times \G \times \mathrm{H}^1(t,t+h;\G) \bigr)$. For every~$k$, we have
	\[
	\int \varphi_k(y,w) \, \dif \mathbf P - \int \varphi_k(x,v) \, \dif \mathbf P
	=
	\int \varphi_k \bigl(\gamma(t+h)\bigr) \, \dif \mathbf P  - \int \varphi_k \bigl(\gamma(t)\bigr) \, \dif \mathbf P \comma
	\]
	which yields, by the fundamental theorem of calculus and the properties of~$\boldsymbol \eta$ and~$\mathbf{P}$,
	\begin{multline*}
		\frac{1}{h}\int_{\Gamma \times \Gamma} \int_0^1  (y-x,w-v) \cdot \nabla_{x,v} \varphi_k \bigl(x+r(y-x),v+r(w-v)\bigr)  \, \dif r \, \dif \bar \pi \\
		=
		\fint_t^{t+h} \int_\Gamma \bigl(\tilde v,F_s(\tilde x, \tilde v)\bigr) \cdot \nabla_{x,v} \varphi_k(\tilde x, \tilde v) \, \dif \mu_s \, \dif s \fstop
	\end{multline*}
	Observe that
	\begin{align*}
		&\abs{\frac{1}{h}\int_{\Gamma^2} \int_0^1  (y-x,w-v) \cdot \left(\nabla_{x,v} \varphi_k\bigl(x+r(y-x),v+r(w-v)\bigr)-\nabla_{x,v} \varphi_k(x,v) \right) \, \dif r \, \dif \bar \pi} \nonumber \\
		&\quad \lesssim \frac{\norm{y-x}_{\Lp{2}(\bar \pi)}^2 + \norm{w-v}_{\Lp{2}(\bar \pi)}^2}{h} \nonumber \\
		&\quad \lesssim \frac{1}{h} \left(
		\norm{y-x-h\frac{w+v}{2}}_{\Lp{2}(\bar \pi)}^2 + \bigl(1+h^2\bigr) \norm{w-v}_{\Lp{2}(\bar \pi)}^2 + h^2 \norm{v}_{\Lp{2}(\mu_t)}^2 \right) \nonumber \\
		&\quad \lesssim
		\frac{\bigl(1+h^2\bigr) \tilde \sfd_h(\mu_t,\mu_{t+h})^2 + h^2 \norm{v}_{\Lp{2}(\mu_t)}^2}{h} \comma
	\end{align*}
	and the last contribution is negligible by \Cref{prop:vlasovtoac1}. Furthermore,
	\[
	\abs{\int_\Gamma \frac{y-x - hv}{h} \cdot \nabla_x \varphi_k \, \dif \bar \pi} \lesssim \frac{\norm{y-x-h v}_{\Lp{2}(\bar \pi)}}{h} \lesssim  \tilde \sfd_h(\mu_t,\mu_{t+h}) \to 0 \comma
	\]
	as~$h \downarrow 0$. We deduce that
	\[
		\lim_{h \downarrow 0} \int_{\Gamma} \frac{w-v}{h} \cdot \nabla_v \varphi_k \, \dif \bar \pi = \int_\Gamma F_t \cdot \nabla_v \varphi_k \, \dif \mu_t \fstop
	\]
	Consequently,
	\[
		\int_\Gamma F_t \cdot \nabla_v \varphi_k \, \dif \mu_t \le \norm{\nabla_v \varphi_k}_{\Lp{2}(\mu_t)} \liminf_{h \downarrow 0} \frac{\norm{w-v}_{\Lp{2}(\bar \pi)}}{h} \le \norm{\nabla_v \varphi_k}_{\Lp{2}(\mu_t)} \liminf_{h \downarrow 0} \frac{\tilde \sfd_h(\mu_t,\mu_{t+h})}{h} \fstop
	\]
	The conclusion follows, as~$F_t$ can be approximated by~$\nabla_v \varphi_k$.
\end{proof}

\subsection{Reparametrisations} \label{sec:repar}

Let~$(\tilde \mu_s)_{s \in (\tilde a,\tilde b)} \subseteq \mathcal P_2(\Gamma)$ be a~$\mathrm{W}_2$-$2$-absolutely continuous curve, and~$(s,t) \mapsto \tilde \pi_{s,t}$ a measurable selection of~$\sfd$-optimal transport plans. Define~$\tilde \Omega$,~$\tilde \rho_s$,~$\tilde T_{s,t}$ in the same way as in the introduction to~\S\ref{sec:sec5}.

\begin{thm} \label{thm:rep}
	Let~$\tilde \lambda \colon (\tilde a,\tilde b) \to \R_{>0}$ be measurable, bounded, and bounded away from zero. Assume that~$\abs{\tilde \rho_s'}_{\mathrm{W}_2} > 0$ for a.e.~$s \in (\tilde a,\tilde b)$. Then, the following are equivalent:
	\begin{enumerate}
		\item \label{rep1} The curve~$(\tilde \mu_s)_{s \in (\tilde a, \tilde b)}$ is a distributional solution to
		\begin{equation}\label{eq:rep1}
			\partial_s \tilde \mu_s + \tilde \lambda (s) \, v \cdot \nabla_x \tilde \mu_s + \nabla_v \cdot (\tilde F_s \tilde \mu_s) = 0
		\end{equation}
		for some force field~$(\tilde F_s)_{s\in (\tilde a, \tilde b)}$ with
		\begin{equation} \label{eq:rep2}
			\int_{\tilde a}^{\tilde b} \left( \norm{v}_{\Lp{2}(\tilde \mu_s)}^2 +  \norm{\tilde F_s}_{\Lp{2}(\tilde \mu_s)}^2 \right) \, \dif s < \infty \fstop
		\end{equation}
		\item \label{rep2} There exists~$\tilde \ell \in \Lp{2}(\tilde a, \tilde b)$ such that
		\begin{equation} \label{eq:rep3}
			\sfd(\tilde \mu_s,\tilde \mu_t)
				\le
			\int_s^t \tilde \ell(r) \, \dif r \comma \qquad \tilde a<s<t<\tilde b \fstop
		\end{equation}
		Moreover,
		\begin{equation} \label{eq:rep4bis}
			\lim_{\tilde h \downarrow 0} \frac{\tilde T_{s,s+\tilde h}}{\tilde h} = \tilde \lambda(s) \quad \text{for a.e.~} s \in (\tilde a, \tilde b) \fstop
		\end{equation}
		and, for every~$[\tilde a', \tilde b'] \subseteq \tilde \Omega$, there exist~$\bar h > 0$ and a function~$\tilde g \in \Lp{1}(\tilde a', \tilde b')$ such that
		\begin{equation} \label{eq:rep5bis}
			\sup_{\tilde h \in (0,\bar h)}\frac{\tilde T_{s,s+\tilde h}}{\tilde h} \le \tilde g(s) \quad \text{for all } s \in [\tilde a', \tilde b'] \fstop
		\end{equation}

	\end{enumerate}
	When the first statement holds for some~$(\tilde F_s)_{s \in (\tilde a, \tilde b)}$, we can choose~$\tilde l \coloneqq 2 \norm{\tilde F_\cdot}_{\Lp{2}(\mu_\cdot)}$ in~\eqref{eq:rep3}.

	When either of the two statements is true, a force field~$(\tilde F_s)_{s \in (\tilde a,\tilde b)}$ for which~\eqref{eq:rep1} and~\eqref{eq:rep2} hold exists in the~$\Lp{2}(\tilde \mu_s \, \dif s)$-closure of~$\set{\nabla_v \tilde \varphi \, : \, \tilde \varphi \in \mathrm{C}^\infty_\mathrm{c}\bigl((\tilde a,\tilde b) \times \Gamma \bigr)}$. Given such a force field, for a.e.~$s \in (\tilde a, \tilde b)$ we have
	\begin{equation} \label{eq:metricder2}
		\lim_{\tilde h \downarrow 0} \frac{\sfd(\tilde \mu_s,\tilde \mu_{s+\tilde h})}{\tilde h} = \norm{\tilde F_s}_{\Lp{2}(\mu_s)} \fstop
	\end{equation}
\end{thm}

\begin{proof}
	Set~$a \coloneqq 0$,~$b \coloneqq \int_{\tilde a}^{\tilde b} \tilde \lambda(s) \, \dif s$, and define the bi-Lipschitz continuous function
	\[
		\tau(s)
			\coloneqq
		\int_0^s \tilde \lambda(r) \, \dif r \comma \qquad s \in (\tilde a,\tilde b) \fstop
	\]
	Define
	\[
		\mu_t \coloneqq \tilde \mu_{\tau^{-1}(t)} \comma \qquad t \in (a,b) \comma
	\]
	as well as
	\[
		\pi_{s,t} \coloneqq \tilde \pi_{\tau^{-1}(s),\tau^{-1}(t)} \in \Pi_{\mathrm o, \sfd}(\mu_s,\mu_t) \comma \qquad a<s<t<b \comma
	\]
	so that
	\[
		T_{s,t} = \tilde T_{\tau^{-1}(s), \tau^{-1}(t)} \comma \qquad a< s<t<b \fstop
	\]
	Note that~$(\mu_t)_{t \in (a,b)}$ is $\mathrm{W}_2$-$2$-absolutely continuous. Indeed, for all~$a<s<t<b$, we have
	\[
		\mathrm{W}_2(\mu_s,\mu_t)
			=
		\mathrm{W}_2\bigl(\tilde \mu_{\tau^{-1}(s)},\tilde \mu_{\tau^{-1}(t)}\bigr)
			\le
		\int_{\tau^{-1}(s)}^{\tau^{-1}(t)} \abs{\tilde \mu'_{\tilde r}}_{\mathrm{W}_2} \, \dif \tilde r
			=
		\int_{s}^{t} \frac{\abs{\tilde \mu'_{\tau^{-1}( r)}}_{\mathrm{W}_2}}{\tilde \lambda\bigl(\tau^{-1}(r)\bigr)} \, \dif r
	\]
	and
	\[
		\int_{a}^{b} \frac{\abs{\tilde \mu'_{\tau^{-1}(r)}}_{\mathrm{W}_2}^2}{\tilde \lambda\bigl(\tau^{-1}(r)\bigr)^2} \, \dif r
			=
		\int_{\tilde a}^{\tilde b} \frac{\abs{\tilde \mu'_{\tilde r}}_{\mathrm{W}_2}^2}{ \tilde \lambda(\tilde r)} \, \dif \tilde r
			\le
		\norm{\frac{1}{\tilde \lambda}}_{\Lp{\infty}} \int_{\tilde a}^{\tilde b} \abs{\tilde \mu'_{\tilde r}}_{\mathrm{W}_2}^2 \, \dif \tilde r < \infty \fstop
	\]
	Moreover, at every differentiability point~$s$ for~$\tilde s \mapsto \tilde \rho_{\tilde s}$ and~$\tau$, for which~$\tau(s)$ is a $W_2$-differentiability point for~$\tilde t \mapsto \rho_{\tilde t}$ and~$\tau'(s) = \tilde \lambda(s) > 0$, we have
	\begin{align}
		\label{eq:derivRho}
		\begin{split}
		0 &< \abs{\tilde \rho_s'}_{\mathrm{W}_2}
			=
		\lim_{h \to 0} \frac{\mathrm{W}_2(\tilde \rho_s,\tilde \rho_{s+h})}{\abs{h}} \\
			&\le \lim_{h \to 0} \frac{\mathrm{W}_2\bigl( \rho_{\tau(s)}, \rho_{\tau(s+h)} \bigr)}{\abs{\tau(s+h) - \tau(s)}}
			\lim_{h \to 0} \frac{\abs{\tau(s+h) - \tau(s)}}{\abs{h}}
				=
			\abs{\rho'_{\tau(s)}} \tau'(s) \fstop
			\end{split}
	\end{align}
	This proves that~$\abs{\rho_t'}_{\mathrm{W}_2} > 0$ for a.e.~$t \in (a,b)$.
	Observe that
		\[
	\Omega = \set{t \in (a,b) \, : \, \norm{v}_{\Lp{2}(\mu_t)} > 0} = \set{ \tau(s) \, : \, s \in (\tilde a, \tilde b) \comma \norm{v}_{\Lp{2}(\tilde \mu_s)} > 0} = \tau(\tilde \Omega) \fstop
			\]
	
	\textbf{Proof of $\mathbf {\ref{rep1} \Rightarrow \ref{rep2}}$.} Define
	\[
		F_t \coloneqq \frac{\tilde F_{\tau^{-1}(t)}}{\tilde \lambda\bigl(\tau^{-1}(t)\bigr)} \comma \qquad t \in (a,b) \fstop
	\]
	We have
	\begin{equation} \label{eq:integrChVar}
		\int_a^b \left( \norm{v}_{\Lp{2}(\mu_t)}^2 + \norm{F_t}_{\Lp{2}(\mu_t)}^2 \right) \, \dif t
			=
		\int_{\tilde a}^{\tilde b} \left( \tilde \lambda(s) \norm{v}_{\Lp{2}(\tilde \mu_s)}^2 + \frac{\norm{\tilde F_s}_{\Lp{2}(\tilde \mu_s )}^2}{\tilde \lambda(s)} \right) \, \dif s < \infty \fstop
	\end{equation}
	Fix~$\varphi \in \mathrm{C}_\mathrm{c}^\infty\bigl( (a,b) \times \Gamma \bigr)$ and define~$\tilde \varphi(s,\cdot) \coloneqq \varphi\bigl(\tau(s), \cdot \bigr)$. Let~$(\tau_k)_{k \in \N}$ be a sequence of~$\mathrm{C}^\infty$ functions converging to~$\tau$ in~$\mathrm{H}^1(\tilde a,\tilde b)$ (hence uniformly), and let~$\tilde \varphi_k(s,\cdot) \coloneqq \varphi\bigl(\tau_k(s),\cdot \bigr)$ for~$s \in (\tilde a, \tilde b)$ and~$k \in \N$. At least when~$k$ is large, we have~$\tilde \varphi_k \in \mathrm C^\infty_\mathrm{c}\bigl((\tilde a, \tilde b) \times \Gamma\bigr)$, so that
	\begin{align}
		0
		 	&\stackrel{\eqref{eq:rep1}}{=}
		 \lim_{k \to \infty} \int_{\tilde a}^{\tilde b} \int (\partial_s \tilde \varphi_k + \tilde \lambda\,  v \cdot \nabla_x \tilde \varphi_k + \tilde F_s \nabla_v \tilde \varphi_k) \, \dif \tilde \mu_s \, \dif s \nonumber \\
		 \label{eq:vlasovlambda}
		 	&=
		 \int_{\tilde a}^{\tilde b} \int (\partial_s \tilde \varphi + \tilde \lambda\,  v \cdot \nabla_x \tilde \varphi + \tilde F_s \nabla_v \tilde \varphi) \, \dif \tilde \mu_s \, \dif s 
		 	=
		 \int_a^b \int (\partial_t \varphi + v \cdot \nabla_x \varphi + F_t \cdot \nabla_v \varphi) \, \dif \mu_t \, \dif t \fstop
	\end{align}
	This proves that~$(\mu_t,F_t)_{t \in (a,b)}$ satisfies \Cref{ass:vlasov}. We apply \Cref{prop:vlasovtoac1} to write
	\begin{equation} \label{eq:boundd2}
		\sfd(\tilde \mu_s,\tilde \mu_t)
			=
		\sfd\bigl( \mu_{\tau(s)},  \mu_{\tau(t)} \bigr)
			\stackrel{\eqref{eq:vlasovtoac11}}{\le} 2 \int_{\tau(s)}^{\tau(t)} \norm{F_r}_{\Lp{2}(\mu_r)} \, \dif r = 2 \int_s^t \norm{\tilde F_{r}}_{\Lp{2}(\tilde \mu_{\tilde r})}  \, \dif \tilde r
	\end{equation}
	and deduce~\eqref{eq:rep3}.
	
	Let~$[\tilde a', \tilde b'] \subseteq \tilde \Omega$. By \Cref{lem:boundT} there exist~$\bar h > 0$ and a function~$g$ in~$\Lp{2}$ such that
	\[
		\frac{T_{t,t+h}}{h} \le g(t) \comma \quad \text{for all } t \in \bigl[\tau(\tilde a'), \tau(\tilde b')\bigr] \comma \, \text{ and every } h \in (0,\bar h) \fstop
	\]
	Then, there exists a constant~$C_{\tilde \lambda} > 0$ such that
	\begin{equation} \label{eq:domination}
	\frac{\tilde T_{s,s+\tilde h}}{\tilde h}
	=
	\frac{{\tau(s+\tilde h) - \tau(s)}}{\tilde h}
	\frac{ T_{\tau(s),\tau(s+\tilde h)}}{\tau(s+\tilde h) - \tau(s)}
	\le
	C_{\tilde \lambda} \, g\bigl(\tau(s) \bigr)
	\end{equation}
	for~$s \in [\tilde a',\tilde b']$ and~$h \in \bigl(0,\bar h/C_{\tilde \lambda}\bigr)$. Observe that~$s \mapsto g\bigl(\tau(s)\bigr)$ is square-integrable (hence integrable), thus~\eqref{eq:rep5bis} follows.
	By \Cref{cor:limT}, we have
	\[
		\lim_{h \downarrow 0} \frac{T_{t,t+h}}{h} = 1 \quad \text{for a.e.~} t \in (a,b) \comma
	\]
	hence~\eqref{eq:rep4bis} thanks to
	\begin{equation} \label{eq:pointconv}
		\frac{\tilde T_{s,s+\tilde h}}{\tilde h} =
		\frac{{\tau(s+\tilde h) - \tau(s)}}{\tilde h}
		\frac{ T_{\tau(s),\tau(s+\tilde h)}}{\tau(s+\tilde h) - \tau(s)} \to \tilde \lambda(s) \quad \text{as } h \downarrow 0\text{ for a.e.~} s \in (\tilde a, \tilde b) \fstop
	\end{equation} %

	\textbf{Proof of $\mathbf {\ref{rep2} \Rightarrow \ref{rep1}}$.} 
	By~\eqref{eq:rep4bis}, for a.e.~$t \in (a,b)$, we have
	\begin{equation} \label{eq:convpunt}
		\frac{T_{t,t+h}}{h} = \frac{\tau^{-1}(t+h) - \tau^{-1}(t)}{h} \frac{\tilde T_{\tau^{-1}(t),\tau^{-1}(t+h)}}{\tau^{-1}(t+h) - \tau^{-1}(t)} \to \frac{1}{\tilde \lambda\bigl(\tau^{-1}(t)\bigr)} \tilde \lambda\bigl(\tau^{-1}(t)\bigr) = 1 \fstop
	\end{equation}
	For every~$s,t$ with~$a<s<t<b$, we have
	\[
		\sfd(\mu_s,\mu_t) = \sfd(\tilde \mu_{\tau^{-1}(s)}, \tilde \mu_{\tau^{-1}(t)}) \stackrel{\eqref{eq:rep3}}{\le} \int_{\tau^{-1}(s)}^{\tau^{-1}(t)} \tilde \ell(\tilde r) \, \dif \tilde r
			= \int_s^t \frac{\tilde \ell\bigl(\tau^{-1}(r) \bigr)}{\tilde \lambda\bigl(\tau^{-1}(r)\bigr)} \, \dif r \comma
	\]
	and we notice that
	\[
		\int_a^b \frac{\tilde \ell\bigl(\tau^{-1}(r) \bigr)^2}{\tilde \lambda\bigl(\tau^{-1}(r)\bigr)^2} \, \dif r
			=
		\int_{\tilde a}^{\tilde b} \frac{\tilde \ell(\tilde r)^2}{\tilde \lambda(\tilde r)} \, \dif \tilde r < \infty \fstop
	\]
	This proves Assumption~\ref{ass:ac_bis} in \Cref{cor:actovlasov}, which together with~\eqref{eq:convpunt}, fulfil the hypotheses of \Cref{lem:derrhov_new}. Then,~$\Omega$ has full measure in~$(a,b)$.
	
	To prove Assumption~\ref{ass:conv_bis}, let us fix~$[a',b'] \subseteq  \Omega$. 
	For every~$t \in [a',b']$, we have
	\[
		\frac{T_{t,t+h}}{h} = \frac{\tau^{-1}(t+h) - \tau^{-1}(t)}{h} \frac{\tilde T_{\tau^{-1}(t),\tau^{-1}(t+h)}}{\tau^{-1}(t+h) - \tau^{-1}(t)} \le \norm{\frac{1}{\tilde \lambda}}_{\Lp{\infty}} \frac{\tilde T_{\tau^{-1}(t),\tau^{-1}(t+h)}}{\tau^{-1}(t+h) - \tau^{-1}(t)} \comma
	\]
	and, by~\eqref{eq:rep5bis},
	\[
		\frac{T_{t,t+h}}{h} \le C_{\tilde \lambda} \tilde g \bigl( \tau^{-1}(t) \bigr)
	\]
	for~$h \in (0,\bar h/C_{\tilde \lambda})$, for some constant~$C_{\tilde \lambda}$. The function~$t \mapsto \tilde g \bigl(\tau^{-1}(t)\bigr)$ is integrable on~$[a',b']$, therefore%
~\eqref{eq:Thh1} follows from~\eqref{eq:convpunt} and the dominated convergence theorem. \Cref{cor:actovlasov} provides a force field~$(F_t)_t$ such that~$(\mu_t,F_t)_t$ solves~\eqref{eq:cont} on~$(a, b) \times \Gamma$ and~$(F_t)_t$ belongs to the~$\Lp{2}(\mu_t \, \dif t)$-closure of~$\set{\nabla_v \varphi \, : \, \varphi \in \mathrm{C}^\infty_{\mathrm{c}}\bigl((a,b) \times\Gamma \bigr)}$. As we did in~\eqref{eq:vlasovlambda}, it is possible to prove that the curve~$(\tilde \mu_s)_s$ and the field
	\[
		\tilde F_s \coloneqq \tilde \lambda(s) F_{\tau (s)} \comma \qquad s \in (\tilde a, \tilde b) \comma
	\]
	solve~\eqref{eq:rep1}. From~\eqref{eq:integrChVar}, we infer~\eqref{eq:rep2}. By \Cref{prop:metricder}, we find~\eqref{eq:metricder2}: for a.e.~$s$,
	\[
		\frac{\sfd(\tilde \mu_s,\tilde \mu_{s+\tilde h})}{\tilde h} = \frac{\tau(s+\tilde h)-\tau(s)}{\tilde h} \frac{\sfd(\mu_{\tau(s)},\mu_{\tau(s+\tilde h)})}{\tau(s+\tilde h)-\tau(s)}  \stackrel{\eqref{eq:metricder}}{\to} \tilde \lambda(s) \norm{F_{\tau(s)}}_{\Lp{2}(\mu_{\tau(s)})} = \norm{\tilde F_s}_{\Lp{2}(\tilde \mu_s)} \fstop
	\]
	
	 We show that~$(\tilde F_s)_s$ lies to the $\Lp{2}(\tilde \mu_s \, \dif s)$-closure of~$\set{\nabla_v \tilde \varphi \, : \, \tilde \varphi \in \mathrm{C}^\infty_\mathrm{c}\bigl((\tilde a,\tilde b)\times \Gamma \bigr)}$. Let~$(\varphi_k)_{k \in \N}$ be a sequence of $\mathrm{C}^\infty_\mathrm{c}\bigl((a,b) \times \Gamma \bigr)$ functions such that~$\nabla_v \varphi_k \to F$ in~$\Lp{2}(\mu_t \, \dif t)$, and let~$(\tau_l)_{l \in \N}$ be a sequence of~$\mathrm{C}^\infty$ functions converging to~$\tau$ in~$\mathrm{H}^1(\tilde a, \tilde b)$. For every~$k$,
	\[
		\tilde \varphi_{k,l} \colon (s,x,v) \longmapsto \tau_l'(s) \varphi_k\bigl(\tau_l(s),x,v\bigr)
	\]
	belongs to~$\mathrm{C}^\infty_\mathrm{c}\bigl((\tilde a, \tilde b) \times \Gamma\bigr)$, at least for large~$l$. As~$l \to \infty$, the sequence~$(\nabla _v\varphi_{k,l})_l$ converges to the $v$-gradient of the function
	\[
		\tilde \varphi_k \colon (s,x,v) \longmapsto \tilde \lambda(s) \varphi_k \bigl(\tau(s),x,v\bigr)
	\]
	in~$\Lp{2}(\mu_s \, \dif s)$, since~$\tau_l \to \tau$ uniformly,~$\tau_l' \to \tilde \lambda$ in~$\Lp{2}(\dif s)$, and~$\nabla_v \varphi_k\bigl(\tau_l(s),x,v\bigr)$ is uniformly bounded in~$l$ for fixed~$k$. Moreover, by a change of variables
	\begin{multline*}
		\int_{\tilde a}^{\tilde b} \int \abs{\tilde \lambda(s) \nabla \varphi_k\bigl(\tau(s),\cdot\bigr) - \tilde F_s}^2 \, \dif \tilde \mu_s \, \dif s
			=
		\int_{\tilde a}^{\tilde b} \abs{\tilde \lambda(s)}^2 \int \abs{ \nabla \varphi_k\bigl(\tau(s),\cdot\bigr) - F_{\tau(s)}}^2 \, \dif \mu_{\tau(s)} \, \dif s  \\
			 =
		\int_{a}^{b} \abs{\tilde \lambda(\tau^{-1}(t))} \int \abs{ \nabla \varphi_k - F_{t}}^2 \, \dif \mu_{t} \, \dif t
			\le
		\norm{\tilde \lambda}_{\Lp{\infty}} \int_{a}^{b} \int \abs{ \nabla \varphi_k - F_{t}}^2 \, \dif \mu_{t} \, \dif t \comma
	\end{multline*}
	and the latter tends to~$0$ as~$k \to \infty$.
\end{proof}

\section*{Acknowledgements}
This work was partially inspired by an unpublished note from 2014 by Guillaume Carlier, Jean Dolbeault, and Bruno Nazaret. GB deeply thanks Jean Dolbeault for proposing this problem to him, guiding him into the subject, and sharing the aforementioned note. \\
We are grateful to Karthik Elamvazhuthi for making us aware of the work~\cite{elamvazhuthi2025benamoubrenierformulationoptimaltransport}. \\
The work of GB has received funding from the European Union’s Horizon
2020 research and innovation programme under the Marie Sklodowska-Curie grant
agreement No 101034413.\\
JM and FQ gratefully acknowledge support from the Austrian Science Fund (FWF) project \href{https://doi.org/10.55776/F65}{10.55776/F65}.

\bibliographystyle{siam}

\begin{thebibliography}{10}

\bibitem{MR2401600}
{\sc L.~Ambrosio, N.~Gigli, and G.~Savar\'{e}}, {\em Gradient flows in metric
  spaces and in the space of probability measures}, Lectures in Mathematics ETH
  Z\"{u}rich, Birkh\"{a}user Verlag, Basel, second~ed., 2008.

\bibitem{baudoin2013bakry}
{\sc F.~Baudoin}, {\em Bakry-{Emery} meet {Villani}}, J. Funct. Anal., 273
  (2017), pp.~2275--2291.

\bibitem{benamou2000computational}
{\sc J.-D. Benamou and Y.~Brenier}, {\em A computational fluid mechanics
  solution to the {M}onge-{K}antorovich mass transfer problem}, Numerische
  Mathematik, 84 (2000), pp.~375--393.

\bibitem{benamou2019second}
{\sc J.-D. Benamou, T.~O. Gallou{\"e}t, and F.-X. Vialard}, {\em Second-order
  models for optimal transport and cubic splines on the {W}asserstein space},
  Foundations of Computational Mathematics, 19 (2019), pp.~1113--1143.

\bibitem{Bogachev07}
{\sc V.~I. Bogachev}, {\em Measure theory. {V}ol. {I}, {II}}, Springer-Verlag,
  Berlin, 2007.

\bibitem{boltzmann1872weitere}
{\sc L.~Boltzmann}, {\em Weitere studien {\"u}ber das {W}{\"a}rmegleichgewicht
  unter {G}asmolek{\"u}len}, Sitz.-Ber. Akad. Wiss. Wien (II), 66 (1872),
  pp.~275--370.

\bibitem{brenier1991polar}
{\sc Y.~Brenier}, {\em Polar factorization and monotone rearrangement of
  vector-valued functions}, Communications on pure and applied mathematics, 44
  (1991), pp.~375--417.

\bibitem{MR432846}
{\sc L.~D. Brown and R.~Purves}, {\em Measurable selections of extrema}, Ann.
  Statist., 1 (1973), pp.~902--912.

\bibitem{bunne2024optimal}
{\sc C.~Bunne, G.~Schiebinger, A.~Krause, A.~Regev, and M.~Cuturi}, {\em
  Optimal transport for single-cell and spatial omics}, Nature Reviews Methods
  Primers, 4 (2024), p.~58.

\bibitem{cavalletti2019variational}
{\sc F.~Cavalletti, M.~Sedjro, and M.~Westdickenberg}, {\em A variational time
  discretization for compressible {E}uler equations}, Transactions of the
  American Mathematical Society, 371 (2019), pp.~5083--5155.

\bibitem{chen2018measure}
{\sc Y.~Chen, G.~Conforti, and T.~T. Georgiou}, {\em Measure-valued spline
  curves: {A}n optimal transport viewpoint}, SIAM Journal on Mathematical
  Analysis, 50 (2018), pp.~5947--5968.

\bibitem{chen2021controlling}
{\sc Y.~Chen, T.~T. Georgiou, and M.~Pavon}, {\em Controlling uncertainty},
  IEEE Control Systems Magazine, 41 (2021), pp.~82--94.

\bibitem{chewi2021fast}
{\sc S.~Chewi, J.~Clancy, T.~Le~Gouic, P.~Rigollet, G.~Stepaniants, and
  A.~Stromme}, {\em Fast and smooth interpolation on {W}asserstein space}, in
  International Conference on Artificial Intelligence and Statistics, PMLR,
  2021, pp.~3061--3069.

\bibitem{chizat2022trajectory}
{\sc L.~Chizat, S.~Zhang, M.~Heitz, and G.~Schiebinger}, {\em Trajectory
  inference via mean-field {L}angevin in path space}, Advances in Neural
  Information Processing Systems, 35 (2022), pp.~16731--16742.

\bibitem{clancy2021interpolating}
{\sc J.~Clancy}, {\em Interpolating Spline Curves of Measures}, PhD thesis,
  Massachusetts Institute of Technology, 2021.

\bibitem{dolbeault2015hypocoercivity}
{\sc J.~Dolbeault, C.~Mouhot, and C.~Schmeiser}, {\em Hypocoercivity for linear
  kinetic equations conserving mass}, Transactions of the American Mathematical
  Society, 367 (2015), pp.~3807--3828.

\bibitem{dolbeault2009new}
{\sc J.~Dolbeault, B.~Nazaret, and G.~Savar{\'e}}, {\em A new class of
  transport distances between measures}, Calculus of Variations and Partial
  Differential Equations, 34 (2009), pp.~193--231.

\bibitem{duong2013generic}
{\sc M.~H. Duong, M.~A. Peletier, and J.~Zimmer}, {\em G{ENERIC} formalism of a
  {V}lasov-{F}okker-{P}lanck equation and connection to large-deviation
  principles}, Nonlinearity, 26 (2013), pp.~2951--2971.

\bibitem{duong2014conservative}
\leavevmode\vrule height 2pt depth -1.6pt width 23pt, {\em
  Conservative-dissipative approximation schemes for a generalized {K}ramers
  equation}, Math. Methods Appl. Sci., 37 (2014), pp.~2517--2540.

\bibitem{elamvazhuthi2025benamoubrenierformulationoptimaltransport}
{\sc K.~Elamvazhuthi}, {\em Benamou-{B}renier formulation of optimal transport
  for nonlinear control systems on {R}d}, arXiv preprint arXiv:2407.16088,
  (2025).

\bibitem{ElamvazhuthiLiuLiOsher}
{\sc K.~Elamvazhuthi, S.~Liu, W.~Li, and S.~Osher}, {\em Dynamical optimal
  transport of nonlinear control-affine systems}, J. Comput. Dyn., 10 (2023),
  pp.~425--449.

\bibitem{figalli2017monge}
{\sc A.~Figalli}, {\em The {M}onge-{A}mp\`ere equation and its applications},
  Zurich Lectures in Advanced Mathematics, European Mathematical Society (EMS),
  Z\"urich, 2017.

\bibitem{figalli2010mass}
{\sc A.~Figalli and L.~Rifford}, {\em Mass transportation on sub-{R}iemannian
  manifolds}, Geometric and functional analysis, 20 (2010), pp.~124--159.

\bibitem{MR2643592}
{\sc J.~Fontbona, H.~Gu\'{e}rin, and S.~M\'{e}l\'{e}ard}, {\em Measurability of
  optimal transportation and strong coupling of martingale measures}, Electron.
  Commun. Probab., 15 (2010), pp.~124--133.

\bibitem{friesecke2023gencol}
{\sc G.~Friesecke and M.~Penka}, {\em The {GenCol} algorithm for
  high-dimensional optimal transport: general formulation and application to
  barycenters and {W}asserstein splines}, SIAM Journal on Mathematics of Data
  Science, 5 (2023), pp.~899--919.

\bibitem{gangbo2009optimal}
{\sc W.~Gangbo and M.~Westdickenberg}, {\em Optimal transport for the system of
  isentropic {E}uler equations}, Communications in Partial Differential
  Equations, 34 (2009), pp.~1041--1073.

\bibitem{MR2920736}
{\sc N.~Gigli}, {\em Second order analysis on {$(\mathcal P_2(M),W_2)$}}, Mem.
  Amer. Math. Soc., 216 (2012), pp.~xii+154.

\bibitem{gorban2018hilbert}
{\sc A.~N. Gorban}, {\em Hilbert's sixth problem: the endless road to rigour},
  Philos. Trans. Roy. Soc. A, 376 (2018), pp.~20170238, 10.

\bibitem{grmela1997dynamics}
{\sc M.~Grmela and H.~C. {\"O}ttinger}, {\em Dynamics and thermodynamics of
  complex fluids. {I}. {D}evelopment of a general formalism}, Physical Review
  E, 56 (1997), pp.~6620--6632.

\bibitem{hormander1967hypoelliptic}
{\sc L.~H{\"o}rmander}, {\em Hypoelliptic second order differential equations},
  Acta Mathematica, 119 (1967), pp.~147--171.

\bibitem{huang2000variationall}
{\sc C.~Huang}, {\em A variational principle for the {K}ramers equation with
  unbounded external forces}, Journal of mathematical analysis and
  applications, 250 (2000), pp.~333--367.

\bibitem{huang2000variational}
{\sc C.~Huang and R.~Jordan}, {\em Variational formulations for
  {V}lasov--{P}oisson--{F}okker--{P}lanck systems}, Mathematical methods in the
  applied sciences, 23 (2000), pp.~803--843.

\bibitem{iacobelli2022new}
{\sc M.~Iacobelli}, {\em A new perspective on {W}asserstein distances for
  kinetic problems}, Archive for Rational Mechanics and Analysis, 244 (2022),
  pp.~27--50.

\bibitem{MR1617171}
{\sc R.~Jordan, D.~Kinderlehrer, and F.~Otto}, {\em The variational formulation
  of the {F}okker-{P}lanck equation}, SIAM J. Math. Anal., 29 (1998),
  pp.~1--17.

\bibitem{justiniano2023consistent}
{\sc J.~Justiniano, M.~Rajkovi{\'c}, and M.~Rumpf}, {\em Consistent
  approximation of interpolating splines in image metamorphosis}, Journal of
  Mathematical Imaging and Vision, 65 (2023), pp.~29--52.

\bibitem{justiniano2024approximation}
{\sc J.~Justiniano, M.~Rumpf, and M.~Erbar}, {\em Approximation of splines in
  {W}asserstein spaces}, ESAIM: Control, Optimisation and Calculus of
  Variations, 30 (2024), p.~64.

\bibitem{KECHICHIAN1995357}
{\sc J.~A. Kechichian}, {\em Optimal low-thrust transfer using variable bounded
  thrust}, Acta Astronautica, 36 (1995), pp.~357--365.

\bibitem{MR1321597}
{\sc A.~S. Kechris}, {\em Classical descriptive set theory}, vol.~156 of
  Graduate Texts in Mathematics, Springer-Verlag, New York, 1995.

\bibitem{kolmogoroff1934zufallige}
{\sc A.~Kolmogoroff}, {\em Zuf{\"a}llige bewegungen (zur theorie der
  {B}rownschen bewegung)}, The Annals of Mathematics, 35 (1934), pp.~116--117.

\bibitem{lavenant2021towards}
{\sc H.~Lavenant, S.~Zhang, Y.-H. Kim, and G.~Schiebinger}, {\em Toward a
  mathematical theory of trajectory inference}, Ann. Appl. Probab., 34 (2024),
  pp.~428--500.

\bibitem{longuski2014optimal}
{\sc J.~M. Longuski, J.~J. Guzm{\'a}n, and J.~E. Prussing}, {\em Optimal
  control with aerospace applications}, vol.~32, Springer, 2014.

\bibitem{lott2008some}
{\sc J.~Lott}, {\em Some geometric calculations on {W}asserstein space},
  Communications in Mathematical Physics, 277 (2008), pp.~423--437.

\bibitem{marec2012optimal}
{\sc J.-P. Marec}, {\em Optimal space trajectories}, vol.~1, Elsevier, 2012.

\bibitem{maxwell1867iv}
{\sc J.~C. Maxwell}, {\em {IV}. on the dynamical theory of gases},
  Philosophical transactions of the Royal Society of London,  (1867),
  pp.~49--88.

\bibitem{mwitta2024integration}
{\sc C.~Mwitta and G.~C. Rains}, {\em The integration of {GPS} and visual
  navigation for autonomous navigation of an {A}ckerman steering mobile robot
  in cotton fields}, Frontiers in Robotics and AI, 11 (2024).

\bibitem{ottinger1997dynamics}
{\sc H.~C. {\"O}ttinger and M.~Grmela}, {\em Dynamics and thermodynamics of
  complex fluids. {II}. {I}llustrations of a general formalism}, Physical
  Review E, 56 (1997), pp.~6633--6655.

\bibitem{otto2001geometry}
{\sc F.~Otto}, {\em The geometry of dissipative evolution equations: the porous
  medium equation}, Comm. Partial Differential Equations, 26 (2001),
  pp.~101--174.

\bibitem{park2024variational}
{\sc S.~Park}, {\em A variational perspective on the dissipative hamiltonian
  structure of the {V}lasov-{F}okker-{P}lanck equation}, arXiv preprint
  arXiv:2406.13682,  (2024).

\bibitem{rabin2012wasserstein}
{\sc J.~Rabin, G.~Peyr{\'e}, J.~Delon, and M.~Bernot}, {\em Wasserstein
  barycenter and its application to texture mixing}, in Scale Space and
  Variational Methods in Computer Vision: Third International Conference, SSVM
  2011, Ein-Gedi, Israel, May 29--June 2, 2011, Revised Selected Papers 3,
  Springer, 2012, pp.~435--446.

\bibitem{rohbeck2024modeling}
{\sc M.~Rohbeck, C.~Bunne, E.~De~Brouwer, J.-C. Huetter, A.~Biton, K.~Y. Chen,
  A.~Regev, and R.~Lopez}, {\em Modeling complex system dynamics with flow
  matching across time and conditions}, in NeurIPS 2024 Workshop on AI for New
  Drug Modalities.

\bibitem{salem2021optimal}
{\sc S.~Salem}, {\em An optimal transport approach of hypocoercivity for the 1d
  kinetic {F}okker-{P}lank equation}, arXiv preprint arXiv:2102.10667,  (2021).

\bibitem{santambrogio2015optimal}
{\sc F.~Santambrogio}, {\em Optimal transport for applied mathematicians},
  vol.~87 of Progress in Nonlinear Differential Equations and their
  Applications, Birkh\"auser/Springer, Cham, 2015.
\newblock Calculus of variations, PDEs, and modeling.

\bibitem{schiebinger2021reconstructing}
{\sc G.~Schiebinger}, {\em Reconstructing developmental landscapes and
  trajectories from single-cell data}, Current Opinion in Systems Biology, 27
  (2021), p.~100351.

\bibitem{MR4527757}
{\sc L.~Silvestre}, {\em H\"older estimates for kinetic {Fokker}-{Planck}
  equations up to the boundary}, Ars Inveniendi Analytica,  (2022).

\bibitem{villani2009hypocoercivity}
{\sc C.~Villani}, {\em Hypocoercivity}, Mem. Amer. Math. Soc., 202 (2009),
  pp.~iv+141.

\bibitem{MR2459454}
{\sc C.~Villani}, {\em Optimal transport}, vol.~338 of Grundlehren der
  mathematischen Wissenschaften [Fundamental Principles of Mathematical
  Sciences], Springer-Verlag, Berlin, 2009.
\newblock Old and new.

\bibitem{zhang2022wassersplines}
{\sc P.~Zhang, D.~Smirnov, and J.~Solomon}, {\em Wassersplines for neural
  vector field-controlled animation}, in Computer Graphics Forum, vol.~41,
  Wiley Online Library, 2022, pp.~31--41.

\end{thebibliography}

\end{document}